\documentclass{amsart}

\usepackage[square,sort,comma,numbers]{natbib}
\usepackage{amsmath, amssymb, amstext, amsfonts, textcomp, amsxtra, amsbsy, amsgen, amsopn, amscd, mathrsfs, amsthm, latexsym, array}
\usepackage[textwidth=16cm,textheight=22cm,centering]{geometry}

\usepackage{color}

\usepackage[all]{xy}

\usepackage{bbm, dsfont}

\newtheorem{theorem}             {Theorem}  [section]
\newtheorem{definition} [theorem] {Definition}
\newtheorem{lemma}      [theorem]{Lemma}
\newtheorem{corollary}  [theorem]{Corollary}
\newtheorem{proposition}[theorem]{Proposition}
\newtheorem{remark} [theorem] {Remark}

\newtheorem*{conjecture*}{Conjecture}

\numberwithin{equation}{section} \everymath{\displaystyle}

\newcommand{\Cont}{{\rm C}}

\newcommand{\Sch}{\mathcal{S}}
\newcommand{\sgn}{{\rm sgn}}

\newcommand{\intL}{{\rm L}}

\newcommand{\BesselK}{\mathcal{K}}

\newcommand{\Hol}{{\rm H}}
\newcommand{\Nr}{{\rm Nr}}
\newcommand{\Tr}{{\rm Tr}}
\newcommand{\HyG}{ {}_2F_1 }
\newcommand{\GenHyG}[5]{ {}_{#1}F_{#2} \left( \begin{matrix} #3 \\ #4 \end{matrix} ; #5 \right) }
\newcommand{\HyGI}{ {}_2\mathrm{I}_1 }
\newcommand{\GenHyGI}[5]{ {}_{#1}\mathrm{I}_{#2} \left( \begin{matrix} #3 \\ #4 \end{matrix} ; #5 \right) }
\newcommand{\chf}{\kappa}
\DeclareMathOperator{\arccosh}{arccosh}

\newcommand{\gp}[1]{\mathbf{#1}}
\newcommand{\GL}{{\rm GL}}
\newcommand{\PGL}{{\rm PGL}}
\newcommand{\SL}{{\rm SL}}
\newcommand{\SO}{{\rm SO}}

\newcommand{\id}{\mathbbm{1}}

\newcommand{\ag}[1]{\mathbb{#1}}
\newcommand{\A}{\mathbb{A}}
\newcommand{\C}{\mathbb{C}}
\newcommand{\Q}{\mathbb{Q}}
\newcommand{\R}{\mathbb{R}}
\newcommand{\Z}{\mathbb{Z}}

\newcommand{\Mat}{{\rm M}}

\newcommand{\ud}{\mathrm{d}}

\newcommand{\F}{\mathrm{F}}

\newcommand{\vo}{\mathfrak{o}}

\newcommand{\vp}{\mathfrak{p}}

\newcommand{\Dis}{{\rm D}}

\newcommand{\gCl}{{\rm Cl}}
\newcommand{\idl}[1]{\mathfrak{#1}}

\newcommand{\ProjP}{{\rm P}}
\newcommand{\norm}[1][\cdot]{\lvert #1 \rvert}
\newcommand{\extnorm}[1]{\left\lvert #1 \right\rvert}
\newcommand{\Norm}[1][\cdot]{\lVert #1 \rVert}
\newcommand{\extNorm}[1]{\left\lVert #1 \right\rVert}
\newcommand{\Pairing}[2]{\langle #1, #2 \rangle}

\newcommand{\OFour}{\mathfrak{F}}

\newcommand{\Mellin}[2][]{\mathfrak{M}_{#1} \left[ #2 \right]}

\newcommand{\rpL}{{\rm L}}
\newcommand{\rpR}{{\rm R}}
\newcommand{\Res}{{\rm Res}}

\newcommand{\Bas}{\mathcal{B}}

\newcommand{\Cond}{\mathbf{C}}
\newcommand{\cond}{\mathfrak{c}}
\newcommand{\fin}{{\rm fin}}
\newcommand{\eis}{{\rm E}}

\newcommand{\Reis}{\mathcal{E}}
\newcommand{\reg}{{\rm reg}}

\newcommand{\Zeta}{\mathrm{Z}}

\newcommand{\Minv}{\mathfrak{m}}
\newcommand{\Tinv}{\mathfrak{t}}

\newcommand{\Vol}{{\rm Vol}}

\makeatletter

\newcommand{\Rmnum}[1]{\expandafter\@slowromancap\romannumeral #1@}
\makeatother

\title{On Weyl's Subconvex Bound for Cube-Free Hecke characters: Totally Real Case}

\author{Olga Balkanova}
\address{Steklov Mathematical Institute, 8 Gubkina street, 119991, Moscow, Russia}
\email{balkanova@mi-ras.ru}

\author{Dmitry Frolenkov}
\address{Steklov Mathematical Institute, 8 Gubkina street, 119991, Moscow, Russia}
\email{frolenkov@mi-ras.ru}

\author{Han Wu}
\address{School of Mathematical Sciences, University of Scinece and Technology of China, 96 Jinzhai Road, 230026, Hefei, P. R. China}
\email{wuhan1121@yahoo.com}

\begin{document}

\begin{abstract}
	We prove a Weyl-type subconvex bound for cube-free level Hecke characters over totally real number fields. Our proof relies on an explicit inversion to Motohashi's formula. Schwartz functions of various kinds and the invariance of the relevant Motohashi's distributions discovered in a previous paper play central roles.
\end{abstract}

	\maketitle
	
	\tableofcontents

\section{Introduction}

	\subsection{Weyl-type Subconvexity via Motohashi's Formula}
	
	Being approximations to the Lindel\"of hypothesis, the subconvex bounds of automorphic $L$-functions play an important role in the analytic number theory. In the classical case of (character twists of) $\PGL_2$, they have applications to problems like the bounds of the Fourier coefficients of half-integral weight modular forms \cite[Corollary 1]{BH10} which in turn has applications to some equidistribution problems \cite{Du88}, and the prime geodesic theorems \cite{SY13,BCCFL18,BBCL20,CWZ21}.
	
	The first subconvex bound is due to Hardy-Littlewood for the Riemann zeta function
\begin{equation} \label{WeylSubRZ}
	\zeta \left( \tfrac{1}{2}+it \right) \ll_{\epsilon} (1+\norm[t])^{\frac{1}{6}+\epsilon},
\end{equation}
whose proof uses Weyl's differencing method for estimating exponential sums. In general, for an automorphic (cuspidal) representation $\pi$ of $\GL_n(\A_{\Q})$ of \emph{analytic conductor} $\Cond(\pi)$, the (hybrid) Weyl-type subconvex bound is of a bound for its $L$-function $L(s,\pi)$ of the form
\begin{equation} \label{WeylSubGen}
	L\left( \tfrac{1}{2},\pi \right) \ll_{\epsilon} \Cond(\pi)^{\frac{1}{6}+\epsilon}.
\end{equation}
	The Weyl bound is known only in a few cases, notably for quadratic twists of certain self-dual $\GL_2$ automorphic forms, see \cite{CI00,Iv01,Yo17,PY18}. A recent breakthrough of Petrow-Young \cite{PY19_CF, PY19_All} generalizes Conrey-Iwaniec's work \cite{CI00}, first to certain self-dual $\GL_2$ automorphic forms of twists type by cube-free Dirichlet characters, then without the cube-free restriction. In particular, (\ref{WeylSubGen}) is established for $\pi=\chi$ being any Dirichlet character. The central main result of this paper is the following generalization of Petrow-Young's cube-free result to totally real number fields, which also improves the third author's previous work \cite{Wu19_S} from the Burgess quality ($3/16$) to the Weyl's ($1/6$) in the cube-free case.
\begin{theorem} \label{WeylBd}
	Let $\chi$ be a Hecke character of a totally real number field $\F$ with cube-free level. Then we have for any $\epsilon > 0$
	$$ L \left( \tfrac{1}{2},\chi \right) \ll_{\F, \epsilon} \Cond(\chi)^{\frac{1}{6}+\epsilon}. $$
	When $[\F:\Q]$ is absolutely bounded, the implied constant depends polynomially on the discriminant $D_{\F}$.
\end{theorem}

	All such Weyl bounds are intimately related to a beautiful formula discovered by Motohashi \cite{Mo93} relating the fourth moment of the Riemann zeta function and the cubic moment of the $L$-functions associated with the automorphic representations for $\SL_2(\Z)$, where the transform formula from the weight functions on the fourth moment side to the cubic moment side is explicitly given. In \cite{CI00} Conrey-Iwaniec established a Motohashi-type formula in the inverse direction, bounding a certain cubic moment of some short family of automorphic $L$-functions of $\GL_2$ by the fourth moment of the corresponding dual family of Dirichlet $L$-functions. Since the cubic moment for $L$-functions of Eisenstein series is the sixth moment of some Dirichlet $L$-functions, their result implies the Weyl-type subconvex bounds in the level aspect for Dirichlet $L$-functions of quadratic characters. The above-mentioned Petrow-Young's work \cite{PY19_CF,PY19_All} are refinements of Conrey-Iwaniec's method. In such methods the estimation is prioritized, so that the relevant Motohashi-type formula is only partly established with specific test functions. In particular, no explicit transform of weight functions in the inverse direction of Motohashi's formula has been obtained.

	The third author obtained a version of Motohashi's formula \cite[Theorem 1.5]{Wu22}, which we rewrite in the special case of trivial (central) characters as follows. Let $\F$ be a number field with adele ring $\A$. Let $S_{\F}$ be the set of places of $\F$, and $S_{\infty}$ be the subset of archimedean places.

\begin{theorem} \label{DMTF}
	Let $S_{\infty} \subset S \subset S_{\F}$ be any finite subset containing divisors of the different ideal of $\F/\Q$. At every place $v \in S$ there is a pair of weight functions $h_v$ and $\widetilde{h}_v$ with the auxiliary \emph{normalized} ones at finite places $\vp < \infty$
\begin{equation} \label{eq: NWt}
	H_{\vp}(\pi_{\vp}) := h_{\vp}(\pi_{\vp}) \frac{L(1, \pi_{\vp} \times \pi_{\vp})}{L(1/2, \pi_{\vp})^3}, \quad \widetilde{H}_{\vp}(\chi_{\vp}) := \frac{\widetilde{h}_{\vp}(\chi_{\vp})}{\extnorm{L(1/2, \chi_{\vp})}^4}, 
\end{equation}
	so that the following equation holds (where $\pi$ runs through cuspidal automorphic forms of $\PGL_2(\A)$)
\begin{multline} \label{MTF}
	\sum_{\pi} \frac{L(1/2, \pi)^3}{2 \Lambda_{\F}(2) L(1,\pi,\mathrm{Ad})} \cdot \prod_{v \mid \infty} h_v(\pi_v) \cdot \prod_{\vp \in S} H_{\vp}(\pi_{\vp}) + \\ 
	\sum_{\chi \in \widehat{\R_+ \F^{\times} \backslash \A^{\times}}} \int_{-\infty}^{\infty} \frac{\extnorm{L(1/2+i\tau, \chi)}^6}{2\Lambda_{\F}(2) \extnorm{L(1+2i\tau,\chi^2)}^2} \cdot \prod_{v \mid \infty} h_v(\pi(\chi_v,i\tau)) \cdot \prod_{\vp \in S} H_{\vp}(\pi(\chi_{\vp},i\tau)) \frac{\ud \tau}{2\pi} + \\
	\frac{1}{\zeta_{\F}^*} \Res_{s = \frac{1}{2}} \frac{\zeta_{\F}(1/2+s)^3\zeta_{\F}(1/2-s)^3}{2\Lambda_{\F}(2) \zeta_{\F}(1+2s)\zeta_{\F}(1-2s)} \cdot \prod_{v \mid \infty} h_v(\pi(\id,s)) \cdot \prod_{\vp \in S} H_{\vp}(\pi(\id,s)) \\
	= \frac{1}{\zeta_{\F}^*} \sum_{\chi \in \widehat{\R_+ \F^{\times} \backslash \A^{\times}}} \int_{-\infty}^{\infty} \extnorm{L(1/2+i\tau, \chi)}^4 \cdot \prod_{v \mid \infty} \widetilde{h}_v(\chi_v \norm_v^{i\tau}) \cdot \prod_{\vp \in S} \widetilde{H}_{\vp}(\chi_{\vp} \norm_{\vp}^{i\tau}) \frac{\ud \tau}{2\pi} + \\
	\frac{1}{\zeta_{\F}^*} \sum_{\pm} \pm \Res_{s_1 = \pm \frac{1}{2}} \zeta_{\F}(1/2+s_1)^2 \zeta_{\F}(1/2-s_1)^2 \cdot \prod_{v \mid \infty} \widetilde{h}_v(\norm_v^{s_1}) \cdot \prod_{\vp \in S} \widetilde{H}_{\vp}(\norm_{\vp}^{s_1}),
\end{multline}
	where we have used the abbreviation $\pi(\chi_v,s) := \pi(\chi_v \norm_v^s, \chi_v^{-1} \norm_v^{-s})$.
\end{theorem}
\begin{remark}
	We only mention that the weight functions $h_v(\cdot)$ and $\widetilde{h}_v(\cdot)$ are certain tempered distributions on $\Psi \in \Sch(\Mat_2(\F_v))$, identified with the relevant notation in \cite{Wu22} as: 
	$$ h_v(\pi_v) = h_v(\pi_v)(\Psi) = M_3(\vec{0}, \Psi \mid \pi_v), \quad \widetilde{h}_v(\chi_v) = \widetilde{h}_v(\chi_v)(\Psi) = M_4(\vec{0}, \Psi \mid \chi_v, \tfrac{1}{2}). $$
	Their precise formulae will be recalled from \cite{Wu22} when we need them, see \eqref{eq: LocDWtDisDef} \& \eqref{eq: LocWtDisDef}. They are instances of distributions $\Theta$ satisfying the following invariance property (\cite[Definition 1.2]{Wu22}):
\begin{equation} \label{eq: LocMotInvProp}
	\Theta \left( \rpL_{a(t_1)} \rpR_{d(t_2z,z)} \Psi \right) = \norm[t_1/t_2]_v \norm[z]_v^{-2} \cdot \Theta(\Psi), \quad \forall \ \Psi \in \Sch(\Mat_2(\F_v)), \ t_1,t_2,z \in \F_v^{\times}. 
\end{equation}
\end{remark}
\begin{remark}
	For simplicity of notation, we write the two \emph{degenerate terms} in the above formula as
\begin{equation} \label{DGF}
	DG(\Psi) = \tfrac{1}{\zeta_{\F}^*} \sideset{}{_{\pm}} \sum \pm \Res_{s_1 = \pm \frac{1}{2}} \zeta_{\F}(1/2+s_1)^2 \zeta_{\F}(1/2-s_1)^2 \cdot \sideset{}{_{v \mid \infty}} \prod \widetilde{h}_v(\norm_v^{s_1}) \cdot \sideset{}{_{\vp \in S}} \prod \widetilde{H}_{\vp}(\norm_{\vp}^{s_1}),
\end{equation}
\begin{equation} \label{DSF}
	DS(\Psi) = \tfrac{1}{\zeta_{\F}^*} \Res_{s = \frac{1}{2}} \frac{\zeta_{\F}(1/2+s)^3\zeta_{\F}(1/2-s)^3}{2\Lambda_{\F}(2) \zeta_{\F}(1+2s)\zeta_{\F}(1-2s)} \cdot \sideset{}{_{v \mid \infty}} \prod h_v(\pi(\id,s)) \cdot \sideset{}{_{\vp \in S}} \prod H_{\vp}(\pi(\id,s)). 
\end{equation}	
\end{remark}

	\subsection{Main Local Results}
	
	Consider a real place at which $\F_v = \R$. Note that both $h_v(\pi(\chi_v, s))$ and $\widetilde{h}_v(\chi_v \norm_v^s)$ depend only on $\chi_v \norm_v^s$. Hence for the re-parametrization 
	$$ \chi_v \norm_v^s = \norm_v^{s_v} \sgn^{\varepsilon_v}, \quad s_v \in \C, \ \varepsilon_v \in \{ 0,1 \} $$
	we shall write
\begin{equation} \label{eq: WtFRealRepar}
	h_v(s_v, \varepsilon_v) := h_v(\pi(\chi_v,s)), \quad \widetilde{h}_v(s_v, \varepsilon_v) := \widetilde{h}_v(\chi_v \norm_v^s).
\end{equation}

\noindent Our first main result is an explicit inverse transform of weight functions at the \emph{real} places.
	
\begin{theorem} \label{ExpInvMF}
	Let $v \mid \infty$ be a real place. Assume $s \mapsto h_v(s,\varepsilon) \in \mathrm{Ad}_{\R}(\C)$ is \emph{admissible} in the sense of Definition \ref{def: AdmWtR}. Let $K(x,\tau)$ be the generalized hypergeometric special value
	$$ K(x,\tau) := \Gamma \left( \frac{1}{2}-i x \right) \GenHyGI{3}{2}{\frac{1}{2}+i\tau, \frac{1}{2}+i\tau, \frac{1}{2}+i\tau}{1-i x + i \tau, 1+2i \tau}{1}. $$
	Then we have the formula
\begin{multline*}
	\widetilde{h}(ix, \varepsilon') = \frac{1}{\pi^2} \sum_{\varepsilon \in \Z/2\Z} \int_{-\infty}^{\infty} \left( \sum_{\pm} (\pm 1)^{\varepsilon} K(\pm x, \tau) \right) \cdot h(i\tau, \varepsilon) \tau \tanh(\pi \tau) \ud \tau \\
	+ (-1)^{\varepsilon} \frac{i}{\pi^2} \sum_{\varepsilon \in \Z/2\Z} \int_{-\infty}^{\infty} \left( \sum_{\pm} (\pm 1)^{\varepsilon} K(\pm x, \tau) \right) \cdot h(i\tau, \varepsilon) \frac{\tau}{\cosh(\pi \tau)} \ud \tau + \\
	(-1)^{\varepsilon'} \sum_{\varepsilon \in \Z/2\Z} \frac{i}{\pi^2 \cosh(\pi x)} \int_{-\infty}^{\infty} \left( \sum_{\pm} (\pm 1)^{\varepsilon} K(\pm x, \tau) \right) \cdot h(i\tau, \varepsilon) \tau \ud \tau.
\end{multline*}
\end{theorem}

	At every place $v \in S_{\F}$, we consider a \emph{target} representation $\pi_{0,v}$ in the unitary principal series given by
\begin{equation}
	\pi_{0,v} = \begin{cases}
		\pi(\sgn^{\varepsilon_{0,v}}, i T_v) & \text{if } \F_v = \R \\
		\pi(\chi_{0,\vp}, \chi_{0,\vp}^{-1}) & \text{if } v = \vp < \infty
	\end{cases}.
\end{equation}
	
\noindent Our second main result is the existence of a weight function $h_v \geq 0$ selecting a \emph{short family} containing $\pi_{0,v}$, and a bound of the corresponding dual weight function.

\begin{theorem} \label{DualWtBd}
	(1) At $\F_v = \R$ there is an admissible $h_v \geq 0$ such that for any $\epsilon > 0$ and $C \gg 1$
	$$ h_v(\pi_{0,v})^{-1} \widetilde{h}_v \left( ix, \varepsilon_v \right) \begin{cases} 
		\begin{matrix} \ll_{\epsilon} (1+\norm[T_v])^{\epsilon} & \forall \ x \in \R \\ 
		\ll_{\epsilon,C} (1+\norm[T_v])^{\epsilon}(1+\norm[x])^{-C} & \text{if } \norm[x] \geq T_v \log^2T_2 \end{matrix} 
	\end{cases}.  $$
	
\noindent (2) At $\vp < \infty$ such that $n_{\vp} := \cond(\chi_{0,\vp}) \in \{ 1,2 \}$ there is an admissible $h_{\vp} \geq 0$ such that
	$$ h_{\vp}(\pi_{0,\vp})^{-1} \widetilde{h}_{\vp}(\chi_{\vp})) \ll \mathbbm{1}_{\leq n_{\vp}}(\cond(\chi_{\vp})). $$
\end{theorem}

\begin{remark} \label{rmk: DiffWithCI}
	Although the proof of Theorem \ref{DualWtBd} (2) has a lot of similarity with \cite{PY19_All}, it should be mentioned that the spectral reciprocity formula hidden in the approach of Conrey--Iwaniec or Petrow--Young is the ``$3+1$'' version of Motohashi's formula, where the fourth moment side is an average of $L(1/2, \chi)^3 L(1/2,\chi^{-1})$ instead of $L(1/2, \chi)^2 L(1/2,\chi^{-1})^2$. Consequently our local dual weight function $\widetilde{h}_{\vp}$ is different from the theirs by a local functional equation \`a la Tate. This accounts for the slightly different form of the algebraic exponential sum we met below in the end part of \S \ref{sec: BdExpSums}. However, this relation is far from being obvious. We plan to clarify it in a future paper where the ``$3+1$'' version will be extensively compared with the ``$2+2$'' version given here. We refer the interested reader to \cite{Kw24, Wu24+, Wu25+} for more information regarding the ``$3+1$'' version of Motohashi's formula.
\end{remark}

	\subsection{Main Global Steps}
	
	Let $\F$ be a totally real number field. Let $\chi_0$ be a (unitary) Hecke character of $\F^{\times} \backslash \A^{\times}$ whose analytic conductor $\Cond(\chi_0)$ tends to $\infty$. Write $\chi_{0,v}(t) = t^{iT_v}$ for $t>0$ and some $T_v \in \R$ (hence $\Cond(\chi_{0,v}) \approx 1+\norm[T_v]$). For every $\epsilon > 0$, we associate a set $B(\chi_0,\epsilon)$ of (cuspidal or continuous) automorphic representations $\pi$ of $\GL_2(\A)$ as follows:
\begin{itemize}
	\item[(1)] At any $v \mid \infty$, $\pi_v\simeq \pi(\norm_v^{i\mu}, \norm_v^{-i\mu})$ is in the principal series, then
	$$ \text{either} \quad \mu - T_v \ll (1+\norm[T_v])^{\epsilon} \quad \text{or} \quad \mu + T_v \ll (1+\norm[T_v])^{\epsilon}; $$
	\item[(2)] At any $\vp < \infty$, we have $\cond(\pi_{\vp} \otimes \chi_{0,\vp}^{-1}) \leq \cond(\chi_{0,\vp})$ (for the current paper we assume $\cond(\chi_{0,\vp}) \leq 2$).
\end{itemize}
	
\noindent We define another set $B^{\vee}(\chi_0,\epsilon)$ of (unitary) Hecke characters $\chi$ of $\F^{\times} \backslash \A^{\times}$ as follows:
\begin{itemize}
	\item[(1)] At any $v \mid \infty$, writing $\chi_v(t)=t^{i\mu}$ for $t > 0$, we have (recall $r = [\F:\Q]$) 
	$$ \norm[\mu] \leq \Cond(\chi_0)^{\frac{\epsilon}{r}} (1+\norm[T_v]) \log (1+\norm[T_v]); $$
	\item[(2)] At any $\vp < \infty$, we have $\cond(\chi_{\vp}) \leq \cond(\chi_{0,\vp})$.
\end{itemize}

\noindent Note that $\pi(\chi_0,\chi_0^{-1}) \in B(\chi_0,\epsilon)$. For any (unitary) Hecke character $\chi$ of $\F^{\times} \backslash \A^{\times}$, we also write $\chi \in B(\chi_0,\epsilon)$ or $\id_{B(\chi_0,\epsilon)}(\chi)=1$ instead of $\pi(\chi,\chi^{-1}) \in B(\chi_0,\epsilon)$ by abuse of notation. 

\noindent Plugging the local estimations of Theorem \ref{DualWtBd} into the formula \eqref{MTF}, and taking into account the non-negativity of the $L$-values on the cubic moment side, we obtain the following bound.
	
\begin{proposition} \label{CubicMBd}
	For any $\epsilon > 0$, we have
\begin{align*}
	&\quad \sum_{\pi \in B(\chi_0,\epsilon)} L \left( \tfrac{1}{2},\pi \right)^3 + \sum_{\chi \in \widehat{\F^{\times} \R_{>0} \backslash \A^{\times}}} \int_{\R} \extnorm{L \left( \tfrac{1}{2}+i\tau,\chi \right)}^6 \id_{B(\chi_0,\epsilon)}(\chi \norm_{\A}^{i\tau}) \ud \tau \\
	&\ll_{\epsilon} \Cond(\chi_0)^{1+\epsilon} + \Cond(\chi_0)^{\epsilon} \sum_{\chi \in \widehat{\F^{\times} \R_{>0} \backslash \A^{\times}}} \int_{\R} \extnorm{L \left( \tfrac{1}{2}+i\tau,\chi \right)}^4 \id_{B^{\vee}(\chi_0,\epsilon)}(\chi \norm_{\A}^{i\tau}) \ud \tau.
\end{align*}
\end{proposition}

	We apply a period version of the Motohashi formula (\ref{MTF}) in the opposite direction, variant and generalization of Sarnak's method \cite{Sar85}, to get an analogue of the spectral large sieve inequality over an arbitrary number field $\F$. Since this result has its own interests, we state it as a theorem. Precisely, given an integral ideal $\idl{N}$ and a number $T_v \geq 1$ for every place $v \mid \infty$, we define $B(\vec{T},\idl{N})$ to be the set of (unitary) Hecke characters $\chi$ of $\F^{\times} \backslash \A^{\times}$ such that:
\begin{itemize}
	\item[(1)] At any $v \mid \infty$, we have $\Cond(\chi_v) \leq T_v$;
	\item[(2)] At any $\vp < \infty$, we have $\cond(\chi_{\vp}) \leq \mathrm{ord}_{\vp}(\idl{N})$.
\end{itemize}

\begin{theorem} \label{4thMBd}
	Let $\F$ be any number field (not necessarily totally real). For any $\epsilon > 0$ we have a bound 
	$$ \sum_{\chi \in \widehat{\F^{\times} \R_{>0} \backslash \A^{\times}}} \int_{\R} \extnorm{L \left( \tfrac{1}{2}+i\tau,\chi \right)}^4 \id_{B(\vec{T},\idl{N})}(\chi \norm_{\A}^{i\tau}) \ud \tau \ll_{\F, \epsilon} \left( \prod_{v \mid \infty} T_v \cdot \Nr(\idl{N}) \right)^{1+\epsilon}, $$
	where the implied constant depends polynomially on the discriminant $D_{\F}$ of $\F$ when the degree $[\F:\Q]$ is absolutely bounded.
\end{theorem}

	Combining Proposition \ref{CubicMBd} and Theorem \ref{4thMBd}, one easily deduces
	$$ \int_{-1}^1 \extnorm{L \left( \tfrac{1}{2}+i\tau,\chi_0 \right)}^6 \ud \tau \ll_{\epsilon} \Cond(\chi_0)^{1+\epsilon} $$
	by positivity of every term on the cubic moment side. We readily get Theorem \ref{WeylBd}.

	\subsection{Discussion on Local Methods}
	
	Fix a $v \in S_{\F}$ and omit it from notation for simplicity.
	
		\subsubsection{Disparity of Bounding Dual Weight Functions}
	
	It is reasonable to expect that for all $\F$ an integral transform relating weight functions should exist in the form
\begin{equation} \label{eq: LocWtTrans}
	\widetilde{h}(\chi) = \int_{\widehat{\PGL_2(\F)}} h(\pi) K(\pi, \chi) \ud \mu_{\mathrm{PL}}(\pi).
\end{equation}
	If the weight function $h(\pi)$ selects a \emph{short family} around a target representation $\pi_0$, then the dual weight function $\widetilde{h}(\chi)$ is an approximation of the value $K(\pi_0, \chi)$. The shorter the relevant family is, the closer $\widetilde{h}(\chi)$ is to $K(\pi_0, \chi)$. If $\pi_0$ appears discretely in $\widehat{\PGL_2(\F)}$, the function $h(\pi)$ can be the characteristic function of $\{ \pi_0 \}$, implying $\widetilde{h}(\chi) = K(\pi_0,\chi)$. Hence different methods employed to bound the dual weight function $\widetilde{h}(\chi)$ could be regarded as bounding the kernel function $K(\pi,\chi)$ itself, via possibly different \emph{integral representations}.
	
	In the situation of the current paper, both weight functions are realized as tempered distributions on the same space of Schwartz functions $\Sch(\Mat_2(\F))$. Hence if a construction of $h$ from a function $\Psi \in \Sch(\Mat_2(\F))$ is available, one may try to bound $\widetilde{h}$ directly from its definition via $\Psi$. In other words, one may try to bound $K(\pi,\chi)$ via an integral representation from $\Psi$. Note that our purpose requires the non-negativity of $h$, which is satisfied if $\Psi$ is of \emph{positive type}, namely of the form $\overline{\phi^{\vee}} * \phi$ for some $\phi \in \Sch(\GL_2(\F))$, where $\phi^{\vee}(g) := \phi(g^{-1})$ and the convolution is essentially \emph{three dimensional} over $\PGL_2(\F)$. In particular, we have $\Psi \in \Sch(\GL_2(\F))$.
	
	For non-archimedean $\F$, we are able to find such $\phi$ (see \eqref{TestPhiNA}), adapted from a clever construction due to Nelson \cite{Ne20}. The point is that this $\phi$ is essentially a linear character of a compact open subgroup of $\GL_2(\F)$, so that the convolution $\overline{\phi^{\vee}} * \phi$ can be easily computed (essentially equal to $\phi$ itself up to a scalar). Consequently the integral representation of $K(\pi,\chi)$ obtained via $\Psi$ can be conveniently analyzed. Moreover the adaptation to $h$ selecting some families of \emph{mixed length}, as required in the application to \cite{AW22}, is convenient. For these reasons, we employ the integral representations of $K(\pi,\chi)$ via $\Psi$ at non-archimedean $\F$. (See Remark \ref{rmk: GenLocTransF} for the possibility of an alternative approach.)
	
	By contrast, for an archimedean $\F$ the ``analytic new vector'' construction proposed by Nelson in \cite[Lemma 13.1]{Ne20} does not seem to select relevant families as short as the purpose of the current paper requires (see also \cite{JN19+}, expecially \cite[Theorem 1]{BJN23}). While we are able to produce some $\phi$ for short families at \emph{real} places, we are not able to give simple description of the convolution $\overline{\phi^{\vee}} * \phi$. However, it would definitely be interesting if one finds a method to bound the dual weight by its integral representation from $\Psi = \overline{\phi^{\vee}} * \phi$, say for either our $\phi$, or a new generation of ``analytic new vectors'' which can select short families. It would also be extremely interesting if such a method avoids special functions.\footnote{For our $\phi$ it does not seems to be the case.}
	
	On the other hand, the pattern of the behavior of the dual weight function $\widetilde{h}$ for weight functions $h$ selecting short families turns out to be universal for the ``$3+1$'' generalization of Motohashi's formula (see \cite{Wu25+} for the non-archimedean case). The current paper corresponds to the very special case where the $\GL_3(\F)$ representation $\Pi = \id \boxplus \id \boxplus \id$ is induced from the minimal parabolic subgroup. Taking into account Sakellaridis's work \cite[\S 2.9]{Sa13}, an integral representation of the dual weight function $\widetilde{h}$ via $\Psi \in \Sch(\Mat_2(\F))$ may exist also in the case $\Pi = \pi \boxplus \mu$, where $\pi$ is a dihedral representation of $\GL_2(\F)$. But we have no idea how such integral representation may exist for supercuspidal $\Pi$. 

\begin{remark}
	We do not mean that pure tensors in $\Sch(\F^2) \otimes \Sch(\F^2) \subset \Sch(\Mat_2(\F))$ of the form
	$$ \Psi \left( \begin{pmatrix} x_1 & x_2 \\ x_3 & x_4 \end{pmatrix} \right) = \Psi_1(x_1,x_2) \Psi_2(x_3,x_4) $$
	do not produce interesting $h \geq 0$ (for archimedean $\F$). They are just more difficult to find. In fact, the relative orbital integrals $I(t, \Psi)$ given below in \eqref{GeomIntR} for pure tensors is already large (at least it contains $\Sch(\F^{\times})$), and may contain the $h$ produced in this paper. But it looks difficult to verify if the space of pure tensors contains sufficiently many test functions of positive type. One would need to find another sufficient condition for the non-negativity of $h \geq 0$ if one sticks to pure tensor test functions. Note that the pure tensor test functions correspond to those offered by the period approach of Motohashi's formula, along \cite[\S 4.5.3]{MV10} and \cite{Ne20}. It is for the same reason that Nelson \cite[\S 2.2]{Ne20} needed to pass to \emph{complete} tensor product.
\end{remark}

		\subsubsection{Bessel-like Transforms and Inversions}
		
	Write $\gp{G} = \PGL_2(\F)$. Let $\gp{N} = \gp{N}(\F)$, resp. $\gp{A} = \gp{A}(\F)$ be the subgroup of upper triangular, resp. diagonal matrices. The \emph{Whittaker-Plancherel theorem} is the spectral analysis of the space
	$$ \intL^2(\gp{N} \backslash \gp{G}, \psi) := \left\{ f: \gp{G} \to \C \ \middle| \ f \left( \begin{pmatrix} 1 & x \\ & 1 \end{pmatrix} g \right) = \psi(x) f(g), \ \int_{\gp{N} \backslash \gp{G}} \extnorm{f(g)}^2 \ud g < \infty \right\}, $$
	for the right regular representation of $\gp{G}$, known in quite general setting \cite[\S 15]{Wal92}. The Bessel-Plancherel theorem is the spectral analysis of the space
	$$ \intL^2(\gp{G} // \gp{N}, \psi) := \left\{ f: \gp{G} \to \C \ \middle| \ f \left( \begin{pmatrix} 1 & x_1 \\ & 1 \end{pmatrix} g \begin{pmatrix} 1 & x_2 \\ & 1 \end{pmatrix} \right) = \psi(x_1+x_2) f(g), \ \int_{\gp{N} \backslash \gp{G} / \gp{N}} \extnorm{f(g)}^2 \ud g < \infty \right\}, $$
	and is known only for $\F=\R$ \cite[B.5]{Iw92} and for $\F=\C$ \cite[\S 11]{BrM03}. It is not known for non-archimedean $\F$, nor for other group $\gp{G}$. Note that a naive integration of the Whittaker-Plancherel theorem against $\psi$ on $\gp{N}$ by the right translation has severe convergence issue.
	
	Similarly we may consider the \emph{Linear-Plancherel problem}, which is the spectral analysis of
	$$ \intL^2(\gp{A} \backslash \gp{G}) := \left\{ f: \gp{G} \to \C \ \middle| \ f \left( \begin{pmatrix} t & \\ & 1 \end{pmatrix} g \right) = f(g), \ \int_{\gp{A} \backslash \gp{G}} \extnorm{f(g)}^2 \ud g < \infty \right\}, $$
	for the right regular representation of $\gp{G}$. It turns out that the equation \eqref{eq: LocWtTrans} is essentially equivalent to the \emph{$\mathcal{L}$-Plancherel problem}, which is the spectral analysis of
	$$ \intL^2(\gp{G} // \gp{A}) := \left\{ f: \gp{G} \to \C \ \middle| \ f \left( \begin{pmatrix} t_1 & \\ & 1 \end{pmatrix} g \begin{pmatrix} t_2 & \\ & 1 \end{pmatrix} \right) = f(g), \ \int_{\gp{A} \backslash \gp{G} / \gp{A}} \extnorm{f(g)}^2 \ud g < \infty \right\}. $$
	Our method indicates the possibility to deduce the \emph{spectral decomposition} of $\intL^2(\gp{G} // \gp{A})$ from that of $\intL^2(\gp{A} \backslash \gp{G})$. The key idea is to consider test functions on $\gp{G}$ of convolution type.

\begin{remark}	
	If we restrict the test functions from $\Sch(\Mat_2(\F))$ to $\Sch(\GL_2(\F))$, and consider the transform
	$$ \Sch(\GL_2(\F)) \to \Sch(\gp{G}), \quad \Psi \mapsto \phi(g) := \int_{\F^{\times}} \Psi(zg) \norm[z]_{\F}^2 \ud^{\times} z \cdot \norm[\det g]_{\F}, $$
then the invariance property of $h(\pi)(\Psi)$ and $\widetilde{h}(\chi)(\Psi)$ recalled in \eqref{eq: LocMotInvProp} shows that they are bi-$\gp{A}$-invariant functionals on $\phi \in \Sch(\gp{G})$. This explains the equivalence of our weight transform formula with the above $\mathcal{L}$-Plancherel problem.
\end{remark}
	
\begin{remark}
	Let $\pi \in \widehat{\gp{G}}$. A $\psi$-Whittaker functional $\ell$ is a continuous linear functional on $V_{\pi}^{\infty}$ with
	$$ \ell \left( \begin{pmatrix} 1 & x \\ & 1 \end{pmatrix}.v \right) = \psi(x) \ell(v), \quad \forall \ v \in V_{\pi}^{\infty}, \ x \in \F. $$
	A \emph{linear functional} $\ell$ is a continuous functional on $V_{\pi}^{\infty}$ with
	$$ \ell \left( \begin{pmatrix} t & \\ & 1 \end{pmatrix}.v \right) = \ell(v), \quad \forall \ v \in V_{\pi}^{\infty}, \ t \in \F^{\times}. $$
	The existence and uniqueness of a linear functional are equivalent to those of a $\psi$-Whittaker functional. This has generalization to higher rank groups \cite{JR96}. It is a consequence of some principles in the relative Langlands program that the spectral analysis of $\intL^2(\gp{G} // \gp{N}, \psi)$ and $\intL^2(\gp{G} // \gp{A})$ should be deducible from each other. In fact, an explicit invertible transform $T$ between some dense subspaces of $\intL^2(\gp{G} // \gp{N}, \psi)$ and $\intL^2(\gp{G} // \gp{A})$ commuting with the action of Hecke algebra is found by Sakellaridis \cite[\S 5]{Sa13}. It can be verified that this transform $T$ also commutes with the spectral projections.
\end{remark}

\begin{remark} \label{rmk: GenLocTransF}
	In a forthcoming paper \cite{Wu25+_Temp} the method in this paper will be extended (with technical simplification) to study the spectral analysis of some spaces including $\intL^2(\gp{G} // \gp{A})$. In particular, for $\pi = \pi(\chi_0,\chi_0^{-1})$ a principal or complementary series the formulae of the kernel functions in \eqref{eq: LocWtTrans}
	$$ Q_{\pi}(m) := \zeta_{\F}(1) \int_{\F} \chi_0 \left( \frac{(1-x)x}{x-m} \right) \frac{\ud x}{\norm[(1-x)x(x-m)]_{\F}^{\frac{1}{2}}}, $$
	$$ K(\pi,\chi) := \int_{\F^{\times}} Q_{\pi} \left( \frac{t}{t-1} \right) \chi(t) \norm[t]_{\F}^{\frac{1}{2}} \frac{\ud^{\times} t}{\norm[t-1]_{\F}} $$
will be established uniformly over local fields $\F$. Note that in the case of $\F=\R$ we are able to identify $t \mapsto Q_{\pi} \left( \tfrac{t}{t-1} \right)$ with some hypergeometric functions. We are unable to do so for $\F=\C$. However, we believe this difficulty is resolvable or dispensable. At least Qi succeeded in identifying the kernel functions (in the ``$3+1$'' version) with hypergeometric functions in a recent work \cite{Qi23+}.
\end{remark}

	\subsubsection{Test Functions}
	
	For a test function of positive type $\Psi = \overline{\phi^{\vee}} * \phi$, the local weight function becomes $h(\pi)(\Psi) = \Norm[v(\phi \mid \pi)]^2$, where $v(\phi \mid \pi)$ is the spectral projection of $\widetilde{\phi}(g) := \int_{\gp{A}} \phi^{\vee}(ag) \ud a \in \intL^2(\gp{A} \backslash \gp{G})$ onto the $\pi$-isotypic part. See Remark \ref{rmk: WtAsSqNr} below, which is stated for $\F=\R$ but valid for all local fields $\F$. Hence we are led to finding $\widetilde{\phi}$ whose spectral projections concentrate in a short family of $\pi$. 
	
	At a real place $\F=\R$, we are exclusively interested in short families contained in the spherical series of $\pi$. Hence we consider $\gp{K}$-invariant $\widetilde{\phi} \in \intL^2(\gp{A} \backslash \gp{G} / \gp{K})$ for $\gp{K}=\SO_2(\R)$, which are essentially functions on $\gp{N} \simeq \R$. The spectral projection in $\intL^2(\gp{A} \backslash \gp{G} / \gp{K})$ is essentially the classical Kontorovich--Lebedev transform. Our key observation is that this transform can be decomposed into the composition of two Fourier(-Mellin) transforms with an elementary transform (see Lemma \ref{MPTrans}). We are able to find a subspace of the Schwartz space $\Sch(\R) \simeq \Sch(\gp{N})$, and characterize its image under the Kontorovich--Lebedev transform. Some Gaussian-type functions selecting short families lie in the image.
	
	At a non-archimedean place, we follow similar idea. Since the set of compact subgroups is richer than in the archimedean case, some linear characters $\phi_0$ of certain subgroups of $\gp{K} = \GL_2(\vo)$ discovered by Nelson already select short families of $\pi$. We then adjust the ``free variable'' $\gp{N}$ in $\gp{G} = \gp{A} \gp{N} \gp{K}$ to find a suitable $\phi$ not making $h(\pi)(\Psi)$ identically vanish, which turns out to be a simple right translation of $\phi_0$.
	
\begin{remark}
	It would be worth mentioning an advantage of the above consideration of test functions of convolution type, compared with the analysis of the test functions in the classical Petersson-Kuznetsov formulae. Our method makes use of the Iwasawa decomposition, which is smooth as a map of manifolds. The analysis in the Petersson-Kuznetsov formulae, such as \cite[\S 3, \S 8]{KL13} makes use of the Cartan decomposition $\gp{G} = \gp{K} \gp{A} \gp{K}$, which has non-trivial analytic issue when the $\gp{A}$-coordinate approaches the identity element. Therefore in our analysis of test function we don't have those subtleties encountered in the classical Petersson-Kuznetsov formulae. We also note that our test functions $\Psi$ at a real place is \emph{not} bi-$\SO_2(\R)$-invariant, although $\phi$ is left $\SO_2(\R)$-invariant.
\end{remark}

\section{Preliminaries}

	\subsection{Notation and Convention}
	\label{Notation}
	
		\subsubsection{Special Functions}
		
	It will be notationally convenient to express the formulae in terms of the following variant of the hypergeometric functions.
\begin{definition}
	Let $(p,q) = (2,1)$ or $(3,2)$. We introduce ${}_p\mathrm{I}_q$ by
	$$ \GenHyGI{p}{q}{a_1,a_2,\cdots,a_p}{b_1,b_2,\cdots,b_q}{z} := \frac{\Gamma(a_1)\Gamma(a_2)\cdots \Gamma(a_p)}{\Gamma(b_1)\Gamma(b_2)\cdots \Gamma(b_q)} \GenHyG{p}{q}{a_1,a_2,\cdots,a_p}{b_1,b_2,\cdots,b_q}{z}. $$
	In the case $p=2$ and $q=1$ we also write $\HyGI(a,b;c;z)$ resp. $\HyG(a,b;c;z)$ for the same function above.
\end{definition}
	
\noindent For example, we have for $\Re c > \Re b > 0$
	$$ \HyGI(a,b;c;z) = \frac{\Gamma(a)}{\Gamma(c-b)} \int_0^1 t^{b-1}(1-t)^{c-b-1}(1-zt)^{-a}\ud t, \quad \norm[\arg(1-z)] < \pi $$
	where the integral is Euler's integral for hypergeometric functions. We will encounter many such functions with $c=a+b$ in this paper.
	
	We denote by $\BesselK_{\nu}(x)$ the standard modified Bessel function, given by (see \cite[10.32.9]{OLBC10})
\begin{equation} \label{eq: KBesselIntRep}
	\BesselK_{\nu}(x) = \int_0^{\infty} e^{-x \cosh(t)} \cosh(\nu t) \ud t, \quad \norm[\arg(x)] < \frac{\pi}{2}.
\end{equation}

		\subsubsection{Number Theoretic Notation}
		
	Throughout the paper, $\F$ is a (fixed) totally real number field with ring of integers $\vo$ and of degree $r=[\F : \Q]$. Let its absolute discriminant be $D_{\F}$. $S_{\F}$ resp. $S_{\infty}$ denotes the set of places resp. infinite places of $\F$ and for any $v \in S_{\F}$, $\F_v$ is the completion of $\F$ with respect to the absolute value $\norm_v$ corresponding to $v$. $\norm[x]_{\A} = \sideset{}{_v} \prod \norm[x_v]_v$ for $x=(x_v)_{v \in S_{\F}} \in \A$ is the adelic norm. $\A = \A_\F$ is the ring of adeles of $\F$, while $\A^{\times}$ denotes the group of ideles. We often write a finite place by $\vp$, which also denotes its corresponding prime ideal of $\vo$. We fix a section $s_{\F}$ of the adelic norm map $\norm_{\ag{A}}: \ag{A}^{\times} \to \ag{R}_+$, identifying $\R_+$ as a subgroup of $\A^{\times}$. Hence any character of $\R_+ \F^{\times} \backslash \A^{\times}$ is identified with a character of $\F^{\times} \backslash \A^{\times}$.
	
	We put the standard Tamagawa measure $\ud x = \sideset{}{_v} \prod \ud x_v$ on $\A$ resp. $\ud^{\times} x = \sideset{}{_v} \prod \ud^{\times} x_v$ on $\A^{\times}$. We recall their constructions. Let $\Tr = \Tr_{\Q}^{\F}$ be the trace map, extended to $\A \to \A_{\Q}$. Let $\psi_{\Q}$ be the additive character of $\A_{\Q}$ trivial on $\Q$, restricting to the infinite place as
	$$ \Q_{\infty} = \R \to \C^{(1)}, \quad x \mapsto e^{2\pi i x}. $$
	We put $\psi = \psi_{\Q} \circ \Tr$, which decomposes as $\psi(x) = \sideset{}{_v} \prod \psi_v(x_v)$ for $x=(x_v)_v \in \A$. $\ud x_v$ is the additive Haar measure on $\F_v$, self-dual with respect to $\psi_v$. Precisely, if $\F_v = \R$, then $\ud x_v$ is the usual Lebesgue measure on $\R$; if $v = \vp < \infty$ such that $\vo_{\vp}$ is the valuation ring of $\F_{\vp}$ with prime ideal $\vp$ and a fixed uniformizer $\varpi_{\vp}$, then $\ud x_{\vp}$ gives $\vo_{\vp}$ the mass $D_{\vp}^{-1/2}$, where $D_{\vp}$ is the local component at $\vp$ of the discriminant $D_\F$ of $\F/\Q$ such that $D_\F = \sideset{}{_{\vp < \infty}} \prod D_{\vp}$. Consequently, the quotient space $\F \backslash \A$ with the above measure quotient by the discrete measure on $\F$ admits the total mass $1$ \cite[Ch.\Rmnum{14} Prop.7]{Lan03}. Recall the local zeta-functions: if $\F_v = \R$, then $\zeta_v(s) = \Gamma_{\R}(s) = \pi^{-s/2} \Gamma(s/2)$; if $v=\vp < \infty$ then $\zeta_{\vp}(s) = (1-q_{\vp}^{-s})^{-1}$, where $q_{\vp} := \Nr(\vp)$ is the cardinality of $\vo/\vp$. We then define the local measures on $\F_v^{\times}$ as
	$$ \ud^{\times} x_v := \zeta_v(1) \frac{\ud x_v}{\norm[x]_v}. $$
	In particular, $\Vol(\vo_{\vp}^{\times}, \ud^{\times} x_{\vp}) = \Vol(\vo_{\vp}, \ud x_{\vp})$ for $\vp < \infty$. Their product gives a measure of $\A^{\times}$. Equip $s_{\F}(\R_+)$ with the measure $\ud t/t$ on $\R_+$, where $\ud t$ is the (restriction of the) usual Lebesgue measure on $\R$, and $\F^{\times}$ with the counting measure. Then we have
\begin{equation} \label{eq: VolIdelicQuotient}
	\Vol(\R_+ \F^{\times} \backslash \A^{\times}) = \zeta_{\F}^*, 
\end{equation}
	where $\zeta_{\F}^*$ is the residue at $1$ of the Dedekind zeta function $\zeta_{\F}(s)$.
	
	For $\vp < \infty$, let $\Sch(\F_{\vp}) = \Cont_c^{\infty}(\F_{\vp})$. We call $\Sch(\A) = \otimes_v' \Sch(\F_v)$ the space of Schwartz functions over $\A$. Let $\chi \in \widehat{\F^{\times} \backslash \A^{\times}}$ be a unitary Hecke character, and let $f \in \Sch(\A)$. Tate's global zeta function is defined by
	$$ \Zeta(s,\chi,f) := \int_{\A^{\times}} f(x) \chi(x) \norm[x]_{\A}^s \ud^{\times} x. $$
	These are integral representations of the complete Dedekind zeta function $\Lambda_{\F}(s)$, which has residue $\zeta_{\F}^*$ at $s=1$. If $f=\otimes_v' f_v$ is decomposable, then $\Zeta(s,\chi,f) = \sideset{}{_{v \in S_{\F}}} \prod \Zeta_v(s,\chi_v,f_v)$ is decomposable, too. For $v \mid \infty$, we recall the local functional equation
	$$ \Zeta_v(s,\id,f_v) = \frac{\Gamma_{\R}(s)}{\Gamma_{\R}(1-s)} \Zeta_v(1-s,\id,\hat{f}_v), \quad \Gamma_{\R}(s):= \pi^{-\frac{s}{2}} \Gamma\left( \frac{s}{2} \right), $$
	where $\hat{f}_v(x) := \int_{\R} f_v(y) \psi_v(xy) \ud y$ is the $\psi_v$-Fourier transform of $f_v$.
	
	We also use $\OFour$ for the Fourier transform defined via $\psi$ or $\psi_v$ above. If the transform is defined with respect to a variable indexed by $j$, we write $\OFour_j$ for the partial Fourier transform.

		\subsubsection{Automorphic Representation Theoretic Notation}
		
	For a ring with unity $R$ such as $\F_v$ or $\A$, we define the following subgroups of $\GL_2(R)$
	$$ \gp{Z}(R) = \left\{ z(u) := \begin{pmatrix} u & 0 \\ 0 & u \end{pmatrix} \ \middle| \ u \in R^{\times} \right\}, \quad \gp{N}(R) = \left\{ n(x) := \begin{pmatrix} 1 & x \\ 0 & 1 \end{pmatrix} \ \middle| \ x \in R \right\}, $$
	$$ \gp{A}(R) = \left\{ a(y) := \begin{pmatrix} y & 0 \\ 0 & 1 \end{pmatrix} \ \middle| \ y \in R^{\times} \right\}, \quad \gp{A}(R)\gp{Z}(R) = \left\{ d(t_1,t_2) := \begin{pmatrix} t_1 & \\ & t_2 \end{pmatrix} \ \middle| \ t_1,t_2 \in R^{\times} \right\}. $$
The Haar measures on them are defined in terms of the Haar measures on $R^{\times}, R$ when the latter are defined previously. The product $\gp{B} := \gp{Z} \gp{N} \gp{A}$ is a Borel subgroup of $\GL_2$. We pick the standard maximal compact subgroup $\gp{K} = \sideset{}{_v} \prod \gp{K}_v$ of $\GL_2(\ag{A})$ by defining
	$$ \gp{K}_v = \left\{ \begin{matrix} \SO_2(\ag{R}) & \text{if } v \mid \infty \\ \GL_2(\vo_{\vp}) & \text{if } v = \vp < \infty \end{matrix} \right. , $$
and equip it with the Haar \emph{probability} measure $\ud \kappa_v$. Note that at $v \mid \infty$, this measure coincides with
	$$ \ud g = \frac{\ud X}{\norm[\det X]^2} =  \frac{\ud x_1\ud x_2\ud x_3\ud x_4}{\norm[x_1x_4-x_2x_3]^2}, \quad g = X = \begin{pmatrix} x_1 & x_2 \\ x_3 & x_4 \end{pmatrix} \in \GL_2(\R). $$
	We then define and equip the quotient space
	$$ [\PGL_2] := \gp{Z}(\ag{A}) \GL_2(\F) \backslash \GL_2(\ag{A}) = \PGL_2(\F) \backslash \PGL_2(\ag{A}) $$
with the product measure $\ud \bar{g} := \sideset{}{_v} \prod \ud \bar{g}_v$ on $\PGL_2(\ag{A})$ quotient by the discrete measure on $\PGL_2(\F)$.

	Let $\intL^2(\PGL_2)$ denote the (Hilbert) space of Borel measurable functions $\varphi$ satisfying
	$$ \left\{ \begin{matrix} \varphi(z \gamma g) = \varphi(g), \quad \forall \gamma \in \GL_2(\F), z \in \gp{Z}(\ag{A}), g \in \GL_2(\ag{A}), \\ \int_{[\PGL_2]} \norm[\varphi(g)]^2 \ud \bar{g} < \infty. \end{matrix} \right. $$
	Let $\intL_0^2(\PGL_2)$ denote the subspace of $\varphi \in \intL^2(\PGL_2)$ such that its \emph{constant term}
\begin{equation} \label{eq: ConstTerm}
	\varphi_{\gp{N}}(g) := \int_{\F \backslash \ag{A}} \varphi(n(x)g) \ud x = 0, \quad \text{a.e. } \bar{g} \in [\PGL_2].
\end{equation}
	$\intL_0^2(\PGL_2)$ is a closed subspace of $\intL^2(\PGL_2)$. $\GL_2(\ag{A})$ acts on $\intL_0^2(\PGL_2)$ resp. $\intL^2(\PGL_2)$, giving rise to a unitary representation $\rpR_0$ resp. $\rpR$. The ortho-complement of $\rpR_0$ in $\rpR$ is the orthogonal sum of the one-dimensional spaces
	$$ \ag{C} \left( \xi \circ \det \right) : \quad \xi \text{ Hecke character such that } \xi^2 = \id $$
and $\rpR_c$, which can be identified as a direct integral representation over the unitary dual of $\F^{\times} \backslash \ag{A}^{\times} \simeq \ag{R}_+ \times (\F^{\times} \backslash \ag{A}^{(1)} )$. More precisely, let $\tau \in \ag{R}$ and $\chi$ be a unitary character of $\F^{\times} \R_+ \backslash \ag{A}$, we associate a unitary representation $\pi(\chi,i\tau)$ of $\GL_2(\ag{A})$ on the following Hilbert space $V_{\chi,i\tau}$ of functions via right regular translation
	$$ \left\{ \begin{matrix} f\left( \begin{pmatrix} t_1 & x \\ 0 & t_2 \end{pmatrix} g \right) = \chi(t_1) \chi^{-1}(t_2) \extnorm{\frac{t_1}{t_2}}_{\ag{A}}^{\frac{1}{2}+i\tau} f(g), \quad \forall t_1,t_2 \in \ag{A}^{\times}, x \in \ag{A}, g \in \GL_2(\ag{A}) ; \\ \int_{\gp{K}} \norm[f(\kappa)]^2 \ud \kappa < \infty . \end{matrix} \right. $$
	If $f_{i\tau} \in V_{\chi,i\tau}$ satisfies $f_{i\tau} \mid_{\gp{K}} = f$ is independent of $\tau$, we call it a flat section, which has an obvious extension to $f_s \in \pi(\chi,s)$ for $s \in \C$. Then $\pi(\chi,i\tau)$ is realized as an irreducible integrand of $\rpR_c$ via the Eisenstein series/intertwiner
\begin{equation} \label{eq: EisDef}
	\eis(s,f)(g) = \eis(f_s)(g) := \sum_{\gamma \in \gp{B}(\F) \backslash \PGL_2(\F)} f_s(\gamma g), 
\end{equation}
	absolutely convergent for $\Re s > 1/2$ and admits a meromorphic continuation which is regular at $s=i\tau$.
	
	Let $\pi$ be any irreducible summand (resp. integrand) of $\rpR_0$ (resp. $\rpR_c$), called cuspidal (resp. continuous) automorphic representation. Let $e_2 \in V_{\pi}^{\infty}$, the subspace of the underlying Hilbert space of $\pi$ consisting of smooth vectors, and let $e_1 \in V_{\pi^{\vee}}^{\infty}$ be an element of the smooth dual space of $V_{\pi}^{\infty}$. The associated function on $\GL_2(\A)$
	$$ \beta(g) = \beta(e_2,e_1)(g) := \Pairing{\pi(g).e_2}{e_1} $$
	is called a matrix coefficient of $\pi$. For $\Psi \in \Sch(\Mat_2(\A))$, the Godement-Jacquet zeta function
	$$ \Zeta(s, \Psi, \beta) := \int_{\GL_2(\A)} \Psi(g) \beta(g) \norm[\det g]_{\A}^{s+\frac{1}{2}} \ud g $$
	is absolutely convergent for $\Re s \gg 1$ and admits a meromorphic continuation to $s \in \C$. For $\varphi \in V_{\pi}^{\infty}$ with the constant term defined by the formula \eqref{eq: ConstTerm}, the Hecke-Jacquet-Langlands zeta function
	$$ \Zeta(s,\varphi) := \int_{\F^{\times} \backslash \A^{\times}} \left( \varphi - \varphi_{\gp{N}} \right)(a(y)) \norm[y]_{\A}^s \ud^{\times}y $$
	is absolutely convergent for $\Re s \gg 1$ and admits a meromorphic continuation to $s \in \C$. Both zeta functions have obvious analogues for $\pi(\chi,s)$ with arbitrary $s \in \C$ and for flat sections, in which case we write $\beta_s(e_2,e_1)$ for $\beta(e_{2,s},e_{1,-s})$.
	
	At a finite place $\vp$, we introduce the subgroups $\gp{K}_0[\vp^n]$ and $\gp{K}_1[\vp^n]$ of $\gp{K}_{\vp}$ for $n \in \Z_{\geq 0}$
	$$ \gp{K}_0[\vp^n] = \left\{ \begin{pmatrix} a & b \\ c & d \end{pmatrix} \in \gp{K}_{\vp} \ \middle| \ c \in \vp^n \right\}, \quad \gp{K}_1[\vp^n] = \left\{ \begin{pmatrix} a & b \\ c & d \end{pmatrix} \in \gp{K}_{\vp} \ \middle| \ c,d-1 \in \vp^n \right\}. $$
	For a generic irreducible representation $\pi$ of $\GL_2(\F_{\vp})$, the conductor exponent $\cond(\pi)$ is the least integer $n \geq 0$ such that $\pi$ contains a non zero vector invariant by $\gp{K}_1[\vp^n]$. The conductor $\Cond(\pi)=\Nr(\vp)^{\cond(\pi)}$.
	
		\subsubsection{Other Notation}
	
	$\Mat_2(\A)$ admits an action of $\GL_2(\A) \times \GL_2(\A)$ which induces an action on $\Sch(\Mat_2(\A))$
	$$ \rpL_{g_1} \rpR_{g_2} \Psi(x) := \Psi \left( g_1^{-1} x g_2 \right), \quad \forall g_1,g_2 \in \GL_2(\A), \Psi \in \Sch(\Mat_2(\A)). $$
	This is a smooth Fr\'echet-representation. We also have the obvious local version. The derived action of the universal enveloping algebra of the Lie algebra of $\GL_2(\R)$ at an archimedean place is written as $\rpL(X)$ resp. $\rpR(X)$.

	\subsection{Compatibility of Schwartz Functions}
	
	If $\F = \R$ or $\C$, we define a norm $\Norm$ on $\GL_2(\F)$ by
	$$ \Norm[g] := \Tr (g g^*) + \Tr (g^{-1} (g^{-1})^*), \quad \text{or equivalently} \quad \extNorm{\begin{pmatrix} x_1 & x_2 \\ x_3 & x_4 \end{pmatrix}} := \sum_{i=1}^4 \norm[x_i]^2 + \frac{\sum_{i=1}^4 \norm[x_i]^2}{\norm[x_1x_4-x_2x_3]^2}. $$
	Recall (see \cite[\S 7.1.2]{Wal88}) that a Schwartz function $\phi \in \Sch(\GL_2(\F))$ is a smooth one such that for any $X, Y$ in the enveloping algebra of the complexified Lie algebra of $\GL_2(\F)$ and any $r \geq 0$
	$$ \sup_{g \in \GL_2(\F)} \Norm[g]^r \extnorm{ \rpL(X) \rpR(Y) \phi (g) } < +\infty. $$
	It is closed under convolution \cite[Theorem 7.1.1]{Wal88}.

\begin{proposition} \label{SchExt}
	Let $\F$ be any local field. Any $\phi \in \Sch(\GL_2(\F))$ can be extended by $0$ to a function $\Psi \in \Sch(\Mat_2(\F))$.
\end{proposition}
\begin{proof}
	The non archimedean case is easy. Assume $\F = \R$ for example. Let $\Psi$ be the extension by $0$ of $\phi$. Let $E_i$ be the matrix whose $x_j$ entry is $1$ if $j=i$ and $0$ otherwise. Write $L_i$ for $\rpL(E_i)$ and $\partial_i$ for $\partial/\partial x_i$. Then we have
	$$ \begin{pmatrix} L_1 & L_3 \\ L_2 & L_4 \end{pmatrix} = \begin{pmatrix} x_1 & x_2 \\ x_3 & x_4 \end{pmatrix} \begin{pmatrix} \partial_1 & \partial_3 \\ \partial_2 & \partial_4 \end{pmatrix}. $$
	By induction on the degree of a polynomial $P$, we easily show the existence of finitely many polynomials $Q_{\vec{\alpha}}$ such that
	$$ P(\partial_1, \partial_2, \partial_3, \partial_4) = \sum Q_{\vec{\alpha}}(x_1,x_2,x_3,x_4, (\det)^{-1}) L_1^{\alpha_1} L_2^{\alpha_2} L_3^{\alpha_3} L_4^{\alpha_4}, $$
	where $\vec{\alpha} = (\alpha_1, \alpha_2, \alpha_3, \alpha_4) \in \Z_{\geq 0}^4$ and $\det := x_1x_4-x_2x_3$. The smoothness of $\Psi$ on $\GL_2(\F)$ follows readily. Let $m_0 \in \Mat_2(\F) - \GL_2(\F)$. If $m \in \GL_2(\F)$ with each entry in the ball of radius $\delta$ of the corresponding entry of $m_0$, then $\det m \ll_{m_0} \delta$. Since $\Norm[m] \geq 2^{-1} \norm[\det m]^{-1}$, we get
	$$ 2^{-1} \norm[\det m]^{-1} \extnorm{P(\partial_1, \partial_2, \partial_3, \partial_4) \Psi(m)} \ll \sum_{\vec{\alpha}} \sup_{g \in \GL_2(\F)} \Norm[g] \cdot (1+\Norm[g])^{\deg Q_{\vec{\alpha}}} \extnorm{L_1^{\alpha_1} L_2^{\alpha_2} L_3^{\alpha_3} L_4^{\alpha_4} \phi(m)} < \infty. $$
	Hence $\extnorm{P(\partial_1, \partial_2, \partial_3, \partial_4) \Psi(m)} \ll \norm[\det m] \ll_{m_0} \delta$, proving its continuity at $m_0$. The rapid decay is easy to verify by definition.
\end{proof}

\begin{remark}
	Write $R_i$ for $\rpR(E_i)$, then we have
	$$ \begin{pmatrix} R_1 & R_2 \\ R_3 & R_4 \end{pmatrix} = \begin{pmatrix} x_1 & x_3 \\ x_2 & x_4 \end{pmatrix} \begin{pmatrix} \partial_1 & \partial_2 \\ \partial_3 & \partial_4 \end{pmatrix}. $$
	Hence the $R_i$'s are expressible in terms of the $L_i$'s with coefficients polynomial in $x_i$'s and $(\det)^{-1}$, and conversely. This shows that in the definition of $\Sch(\GL_2(\F))$, one may use only the left or right derivatives.
\end{remark}

\begin{lemma} \label{SchExist}
	Let $\F = \R$ or $\C$. Let $\phi \in \Cont^{\infty}(\GL_2(\F))$ be given as
	$$ \phi \left( \begin{pmatrix} z & \\ & z \end{pmatrix} \kappa \begin{pmatrix} 1 & x \\ & 1 \end{pmatrix} \begin{pmatrix} y & \\ & 1 \end{pmatrix} \right) = \ell(z) s(\kappa) f(x) h(y) $$
	for some functions $\ell, h \in \Sch(\F^{\times})$, $f \in \Sch(\F)$ and $s \in \Cont^{\infty}(\gp{K})$. Then $\phi \in \Sch(\GL_2(\F))$.
\end{lemma}
\begin{proof}
	We treat the real case. The complex case is only notationally more complicated.  Then $\kappa = \begin{pmatrix} \cos \theta & \sin \theta \\ - \sin \theta & \cos \theta \end{pmatrix}$. By induction on the degree of a polynomial $P$, we can show the existence of finitely many polynomials $Q_{\vec{\alpha}}$ such that
	$$ P(\partial_1, \partial_2, \partial_3, \partial_4) = \sum Q_{\vec{\alpha}}(z,y,\cos \theta, \sin \theta) \partial_z^{\alpha_1} \partial_{\theta}^{\alpha_2} \partial_x^{\alpha_3} \partial_y^{\alpha_4}, $$
	the initial step being given by \cite[Proposition 2.2.5]{Bu98}. The statement then follows from
	$$ \extNorm{ \begin{pmatrix} z & \\ & z \end{pmatrix} \begin{pmatrix} \cos \theta & \sin \theta \\ - \sin \theta & \cos \theta \end{pmatrix} \begin{pmatrix} 1 & x \\ & 1 \end{pmatrix} \begin{pmatrix} y & \\ & 1 \end{pmatrix} } \ll \norm[z]^2(1+\norm[x]^2 + \norm[y]^2) + \frac{1+\norm[x]^2 + \norm[y]^2}{\norm[yz]^2}. $$
\end{proof}

	\subsection{Motohashi Invariants}
	\label{sec: MInv}
	
	Consider a $v \in S_{\F}$. For simplicity we drop the subscript $v$ in this subsection. Recall the invariance property of the local weight functions in \eqref{eq: LocMotInvProp}.
	
\begin{definition} \label{MoInvDef}
	We introduce two \emph{Motohashi invariants} as functions on open dense subsets of $\Mat_2(\F)$:
	$$ \Minv: \GL_2(\F) \to \F, \quad x = \begin{pmatrix} x_1 & x_2 \\ x_3 & x_4 \end{pmatrix} \mapsto \frac{x_1x_4}{\det x}; $$
	$$ \Tinv: \left\{ \begin{pmatrix} x_1 & x_2 \\ x_3 & x_4 \end{pmatrix} \in \Mat_2(\F) \ \middle| \ x_2x_3 \neq 0 \right\} \to \F, \quad x = \begin{pmatrix} x_1 & x_2 \\ x_3 & x_4 \end{pmatrix} \mapsto \frac{x_1x_4}{x_2 x_3}. $$
	They are invariant under the left and right multiplication by the diagonal torus of $\GL_2(\F)$.
\end{definition}

\noindent For $t \in \F$ and $\Psi \in \Sch(\Mat_2(\F))$, we define
\begin{equation} \label{GeomIntR}
	I(t, \Psi) := \zeta_{\F}(1)^3 \int_{(\F-\{ 0 \})^3} \Psi \begin{pmatrix} z t_1 t_2 & z t_1 \\ z t_2 & z t \end{pmatrix} \norm[z]_{\F} \ud z \ud t_1 \ud t_2.
\end{equation}

\begin{proposition} \label{M4IntRep}
	The restriction of $\widetilde{h}(\chi)$ to $\Sch(\GL_2(\F))$ ($=\Cont_c^{\infty}(\GL_2(\F))$ if $\F$ is non-archimedean) is represented by $G_{\chi}(g) \norm[\det g]_{\F} \ud g$, where $G_{\chi}(g)$ is a locally integrable function on $\PGL_2(\F)$ determined by
	$$ G_{\chi} \begin{pmatrix} zab & za \\ zb & zc \end{pmatrix} = c_{\F} \chi \left( c \right) \norm[c]_{\F}^{-\frac{1}{2}} \norm[c-1]_{\F}, \quad c_{\F} := \begin{cases} 1 & \text{if } \F = \R \\ \frac{\zeta_{\F}(1)^3}{\zeta_{\F}(2)} \Vol(\vo_{\F}, \ud x)^4 & \text{if } \F \text{ is non-archimedean} \end{cases}. $$
\end{proposition}
\begin{proof}
	We rewrite the definition (see \cite[(1.15)]{Wu22} with an obvious functional equation) as
\begin{align}
	\widetilde{h}(\chi)(\Psi) &= \zeta_{\F}(1)^4 \int_{\F^4} \Psi \begin{pmatrix} x_1 & x_2 \\ x_3 & x_4 \end{pmatrix} \chi \left( \frac{x_1 x_4}{x_2 x_3} \right) \frac{\norm[x_1x_4-x_2x_3]_{\F}^2}{\norm[x_1x_2x_3x_4]_{\F}^{\frac{1}{2}}} \frac{\prod_i \ud x_i}{\norm[x_1x_4-x_2x_3]_{\F}^2} \label{eq: LocDWtDisDef} \\
	&= \zeta_{\F}(1)^4 \int_{\F^4} \Psi \begin{pmatrix} x_1 & x_2 \\ x_3 & x_4 \end{pmatrix} \chi \left( \Tinv \right) \frac{\norm[\Tinv - 1]_{\F}}{\norm[\Tinv]_{\F}^{\frac{1}{2}}} \norm[x_1x_4-x_2x_3]_{\F} \frac{\prod_i \ud x_i}{\norm[x_1x_4-x_2x_3]_{\F}^2}, \nonumber
\end{align}
	where we have written $\Tinv = \Tinv(g)$ and have the measure identification for
	$$ g = \begin{pmatrix} x_1 & x_2 \\ x_3 & x_4 \end{pmatrix} \quad \Rightarrow \quad \zeta_{\F}(1)^4 \frac{\prod_i \ud x_i}{\norm[x_1x_4-x_2x_3]_{\F}^2} = c_{\F} \ud g. $$
	The stated formula for $G_{\chi}(g)$ follows readily.
\end{proof}

\begin{lemma} \label{GeomToM4}
	For any $\Psi \in \Sch(\Mat_2(\F))$, $\widetilde{h}(\chi)(\Psi)$ and $I(t, \Psi)$ determine each other via
	$$ \widetilde{h}(\chi)(\Psi) = \int_{\F^{\times}} I(t,\Psi) \chi(t) \norm[t]_{\F}^{\frac{1}{2}} \ud^{\times}t. $$
\end{lemma}
\begin{proof}
	A direct change of variables gives
\begin{align*}
	\widetilde{h}(\chi)(\Psi) &= \zeta_{\F}(1)^4 \int_{\F^4} \Psi \begin{pmatrix} x_1 & x_2 \\ x_3 & x_4 \end{pmatrix} \chi \left( \frac{x_1 x_4}{x_2 x_3} \right) \norm[x_1x_2x_3x_4]_{\F}^{-\frac{1}{2}} \prod_i \ud x_i \\
	&= \zeta_{\F}(1)^4 \int_{(\F-\{ 0 \})^4} \Psi \begin{pmatrix} z t_1 t_2 & z t_1 \\ z t_2 & z t \end{pmatrix} \chi(t) \norm[t]_{\F}^{-\frac{1}{2}} \norm[z]_{\F} \ud z \ud t_1 \ud t_2 \ud t \\
	&= \zeta_{\F}(1) \int_{\F-\{ 0 \}} I(t,\Psi) \chi(t) \norm[t]_{\F}^{-\frac{1}{2}} \ud t,
\end{align*}
	which obviously is equal to the right hand side of the desired equality.
\end{proof}

\section{Local Analysis at Real Places}

	We prove Theorem \ref{ExpInvMF} and Theorem \ref{DualWtBd} (1) in this section, while providing some auxiliary results needed in the estimation of the degenerate terms. We fix a real place $v \mid S_{\infty}$ and omit it from the subscript for simplicity of notation. In particular $\F=\R$ throughout this section.

	\subsection{Analysis with Test Function}
	
		\subsubsection{Choice of Test Function}
		
	Consider $\phi_j$ (for $j=1,2$) given by
\begin{equation} \label{TestPhiR}
	\phi_j(\kappa n(x) t(y) z) = \ell_j(z) f_j(x) h_j(y), \quad t(y) := \begin{pmatrix} y^{\frac{1}{2}} & \\ & y^{-\frac{1}{2}} \end{pmatrix}, \kappa \in \SO_2(\R)
\end{equation} 
	for some $\ell_j \in \Cont_c^{\infty}(\R^{\times}), h_j \in \Cont_c^{\infty}(\R_{>0}), f_j \in \Sch(\R)$. We take $\Psi = \phi_1^{\vee} * \phi_2$, where $\phi_1^{\vee}(g)=\phi_1(g^{-1})$ and the convolution is taken over $\GL_2(\R)^+$. By an obvious variant of Lemma \ref{SchExist}, we have $\phi_j \in \Sch(\GL_2(\R)^+)$. Hence $\Psi \in \Sch(\Mat_2(\R))$ by Proposition \ref{SchExt}.
	
\begin{definition} \label{def: VecRepZeta}
	For $\phi \in \Sch(\GL_2(\R))$, we denote by $v(\phi \mid \pi)$ (see \cite[Proposition 3.2]{Sha74}) the smooth vector in $\pi$ which represents the linear functional on $V_{\pi^{\vee}}^{\infty}$
	$$ \Pairing{v(\phi \mid \pi)}{f} = \Zeta \left( \tfrac{1}{2}, \pi^{\vee}(\phi^{\vee}).W_f \right), \quad \forall f \in V_{\pi^{\vee}}^{\infty}. $$
\end{definition}

\begin{lemma} \label{M3VSNormLoc}
	Let $\alpha(g)=\norm[\det g]$. For $\Psi = \phi_1^{\vee}*\phi_2$ with any $\phi_j\in \Sch(\GL_2(\R))$, we have
	$$ h(\pi)(\Psi) = \Pairing{v(\phi_2 \cdot \alpha \mid \pi)}{v(\phi_1 \cdot \alpha^{-1} \mid \pi^{\vee})}. $$
\end{lemma}
\begin{proof}
	The definition of $v(\phi \mid \pi)$ implies
\begin{equation} \label{eq: LocMotVec}
	v(\phi \mid \pi) = \sideset{}{_{e \in \Bas(\pi)}} \sum \Zeta \left( \tfrac{1}{2}, W_{e^{\vee}} \right) \cdot \pi(\phi).e, 
\end{equation}
	where $e$, resp. $e^{\vee}$, traverses any orthogonal basis, resp. the dual basis. Hence we get by \cite[(1.17)]{Wu22}
\begin{align}
	h(\pi)(\Psi) &:= \sum_{e_1,e_2 \in \Bas(\pi)} \int_{\GL_2(\F)} \Psi(g) \Pairing{g.e_2}{e_1^{\vee}} \norm[\det g] \ud g \cdot \Zeta \left( \tfrac{1}{2}, W_{e_1} \right) \cdot \Zeta \left( \tfrac{1}{2}, W_{e_2^{\vee}} \right) \label{eq: LocWtDisDef} \\
	&= \sum_{e_1,e_2 \in \Bas(\pi)} \int_{\GL_2(\F)^2} \phi_1^{\vee}(g_1) \phi_2(g_2) \norm[\det(g_1g_2)] \Pairing{g_1g_2.e_2}{e_1^{\vee}} \ud g_1\ud g_2 \cdot \Zeta \left( \tfrac{1}{2}, W_{e_1} \right) \cdot \Zeta \left( \tfrac{1}{2}, W_{e_2^{\vee}} \right) \nonumber \\
	&= \sum_{e_1,e_2 \in \Bas(\pi)} \Pairing{ \pi(\phi_2 \cdot \alpha)e_2 }{\pi^{\vee}(\phi_1 \cdot \alpha^{-1}).e_1^{\vee}} \cdot \Zeta \left( \tfrac{1}{2}, W_{e_1} \right) \cdot \Zeta \left( \tfrac{1}{2}, W_{e_2^{\vee}} \right) \nonumber \\
	&= \Pairing{v(\phi_2 \cdot \alpha \mid \pi)}{v(\phi_1 \cdot \alpha^{-1} \mid \pi^{\vee})}. \nonumber
\end{align}
\end{proof}

\begin{remark} \label{rmk: WtAsSqNr}
	If we let $\phi_2 \cdot \alpha = \phi = \overline{\phi_1} \cdot \alpha^{-1}$ for some $\phi \in \Sch(\GL_2(\R))$, then $h(\pi)(\Psi) = \Norm[v(\phi \mid \pi)]^2$. The definition of $v(\phi \mid \pi)$ implies for any $f \in V_{\pi^{\vee}}^{\infty}$ we have
\begin{equation} \label{eq: VecRepZeta}
	\Pairing{v(\phi \mid \pi)}{f} = \int_{\GL_2(\R)} \phi^{\vee}(g) \cdot \Zeta \left( \frac{1}{2}, \pi^{\vee}(g).W_f \right) \ud g = \int_{\gp{A}(\R) \backslash \GL_2(\R)} \widetilde{\phi}(g) \cdot Z_f(g) \ud g, 
\end{equation}
	where we have written
	$$ \widetilde{\phi}(g) := \int_{\gp{A}(\R)} \phi^{\vee}(ag) \ud a, \quad Z_f(g) := \Zeta \left( \frac{1}{2}, \pi^{\vee}(g).W_f \right). $$
	Note that $A_{\pi^{\vee}}: V_{\pi^{\vee}}^{\infty} \to \Cont^{\infty}(\gp{A}(\R) \backslash \GL_2(\R)), \ f \mapsto Z_f(g)$ is $\GL_2(\R)$-intertwining, and the image is precisely the $\pi^{\vee}$-isotypic part of $\intL^2(\gp{A}(\R) \backslash \GL_2(\R))$. Therefore $v(\phi \mid \pi) = \ProjP_{\pi}(\widetilde{\phi})$ by \eqref{eq: VecRepZeta} since $\ProjP_{\pi}$ is the adjoint of $A_{\pi^{\vee}}$ by definition.
\end{remark}

\begin{corollary} \label{M3Exp}
	Consider $\phi_j$ of the form \eqref{TestPhiR}. Let $w_0 \in V_{\pi}^{\infty}$ resp. $w_0^{\vee} \in V_{\pi^{\vee}}^{\infty}$ be a pair of dual unitary \emph{spherical} vector such that $\Pairing{w_0}{w_0^{\vee}}=1$, whose $\psi_{\R}$-Kirillov functions are $K_{\pi}$ resp. $K_{\pi^{\vee}}$, namely
	$$ \int_{\R^{\times}} K_{\pi}(t) K_{\pi^{\vee}}(-t) \ud^{\times} t = 1. $$
	Then we have the formula
\begin{align*}
	h(\pi)(\Psi) &= \int_{\R^{\times}} \ell_1(z) \norm[z]^{-2} \ud^{\times}z \cdot \int_{\R^{\times}} \ell_2(z) \norm[z]^2 \ud^{\times}z \cdot \int_{\R^{\times}} h_1(y) \ud^{\times}y \cdot \int_{\R^{\times}} h_2(y) \ud^{\times}y \cdot \\
	&\quad \int_{\R^{\times}} \OFour(f_1)(t) K_{\pi}(t) \ud^{\times}t \cdot \int_{\R^{\times}} \OFour(f_2)(t) K_{\pi^{\vee}}(t) \ud^{\times}t.
\end{align*}
\end{corollary}
\begin{proof}
	Since $\phi_j$ is left-invariant by $\SO_2(\R)$, the vectors $v(\phi_2 \cdot \alpha \mid \pi)$ resp. $v(\phi_1 \cdot \alpha^{-1} \mid \pi^{\vee})$ are spherical. We get by Definition \ref{def: VecRepZeta} and an Iwasawa decomposition
\begin{align*}
	\Pairing{v(\phi_2 \cdot \alpha \mid \pi)}{w_0^{\vee}} &= \int_{\R^{\times}} \left( \int_{\GL_2(\R)} \norm[\det g] \phi_2(g) \pi^{\vee}(g^{-1}).K_{\pi^{\vee}}(t) \ud g \right) \ud^{\times}t \\
	&= \int_{(\R^{\times})^2} \left( \int_{\R_{>0} \times \R} \phi_2(n(x)t(y)z) K_{\pi^{\vee}}(ty^{-1}) \psi_{\R}(-txy^{-1}) \ud^{\times}y \ud x \right) \norm[z]^2 \ud^{\times}t \ud^{\times}z \\
	&= \int_{\R^{\times}} \ell_2(z) \norm[z]^2 \ud^{\times}z \cdot \int_{\R^{\times}} h_2(y) \ud^{\times}y  \cdot \int_{\R^{\times}} \OFour(f_2)(t) K_{\pi^{\vee}}(t) \ud^{\times}t.
\end{align*}
	Similarly, we have
	$$ \Pairing{w_0}{v(\phi_1 \cdot \alpha^{-1} \mid \pi^{\vee})} = \int_{\R^{\times}} \ell_1(z) \norm[z]^{-2} \ud^{\times}z \cdot \int_{\R^{\times}} h_1(y) \ud^{\times}y  \cdot \int_{\R^{\times}} \OFour(f_1)(t) K_{\pi}(t) \ud^{\times}t. $$
	Since $\Pairing{v(\phi_2 \cdot \alpha \mid \pi)}{v(\phi_1 \cdot \alpha^{-1} \mid \pi^{\vee})} = \Pairing{v(\phi_2 \cdot \alpha \mid \pi)}{w_0^{\vee}} \Pairing{w_0}{v(\phi_1 \cdot \alpha^{-1} \mid \pi^{\vee})}$, we deduce from Lemma \ref{M3VSNormLoc} the stated formula.
\end{proof}

		\subsubsection{A Distributional Fourier Transform}
		
	We have parametrized the unitary \emph{spherical} series by $i\R \times \{ \pm 1 \} \ni (i\tau, \epsilon) \mapsto \pi(\sgn^{\frac{\epsilon-1}{2}}, i\tau)$. Write $K_{i\tau}^{\epsilon}$ for $K_{\pi}$ in Corollary \ref{M3Exp}. These functions are related to the classical Whittaker functions, because they satisfy the following equations (see \cite[\S 2 (8.2)]{Bu98})
	$$ \Delta K_{i\tau}^{\epsilon} = - \left( \tfrac{1}{4} + \tau^2 \right) K_{i\tau}^{\epsilon}, \quad \Delta := t^2 \left( \tfrac{\ud^2}{\ud t^2} - 4\pi^2 \right); $$
	$$ K_{i\tau}^{\epsilon}(-t) = \epsilon \cdot K_{i\tau}^{\epsilon}(t). $$
	If we impose $K_{i\tau}^{\epsilon}$ to be unitary in the Kirillov model, then they can be rewritten in terms of the $K$-Bessel functions $\BesselK_{\lambda}$ defined in \eqref{eq: KBesselIntRep} as
\begin{equation} \label{KirToKBessel}
	K_{i\tau}^{+}(t) = \sqrt{\tfrac{\cosh(\pi \tau)}{\pi}} \norm[t]^{\frac{1}{2}} \BesselK_{i\tau}(2\pi \norm[t]), \quad K_{i\tau}^{-}(t) = \sqrt{\tfrac{\cosh(\pi \tau)}{\pi}} \sgn(t) \norm[t]^{\frac{1}{2}} \BesselK_{i\tau}(2\pi \norm[t]),
\end{equation}
	Note that the functions $t \mapsto \norm[t]^{-1} K_{i\tau}^{\epsilon}(t)$ are \emph{real valued} and are eigenfunctions for the differential operator
	$$ \widetilde{\Delta} := t^{-1} \circ \Delta \circ t = t^2 \tfrac{\ud^2}{\ud t^2} + 2t \tfrac{\ud}{\ud t} - (2\pi)^2 t^2 $$
	which is self-dual for $\intL^2(\R, \ud t)$. Their distributional $\psi_{\R}$-Fourier transform are eigen-distributions of
\begin{equation} \label{InvOp}
	D := (t^2+1) \tfrac{\ud^2}{\ud t^2} + 2t \tfrac{\ud}{\ud t} 
\end{equation}
	with eigenvalue $-\left( \tfrac{1}{4}+\tau^2 \right)$. These distributions are represented by smooth functions $e_{i\tau}^{\epsilon}$ satisfying
	$$ \overline{e_{i\tau}^{\epsilon}(x)}=e_{i\tau}^{\epsilon}(-x), \quad D e_{i\tau}^{\epsilon} = -\left( \tfrac{1}{4}+\tau^2 \right) e_{i\tau}^{\epsilon} $$
	since $D$ is elliptic. Hence $e_{i\tau}^+(t)$ resp. $e_{i\tau}^-(t)$ is proportional to (see \cite[(1.9) \& (1.10)]{Du13} or verify directly)
\begin{equation} \label{DualEigenFunc}
	F_+(\tau,t) := \HyG \left( \frac{1}{4} - \frac{i\tau}{2}, \frac{1}{4} + \frac{i\tau}{2}; \frac{1}{2}, -t^2 \right) \quad \text{resp.} \quad F_-(\tau,t) := t \cdot \HyG \left( \frac{3}{4} - \frac{i\tau}{2}, \frac{3}{4} + \frac{i\tau}{2}; \frac{3}{2}, -t^2 \right), 
\end{equation}
	which are real valued and satisfy (using the definition via Taylor expansion for hypergeometric functions)
	$$ (F_+(\tau,0), \partial_t F_+(\tau,0)) = (1,0), \quad (F_-(\tau,0), \partial_t F_-(\tau,0)) = (0,1). $$

		\subsubsection{Admissible Weights}
		
	The transforms $f \mapsto F^{\epsilon}$ on $\Sch(\R)$ given by
\begin{equation} \label{NormKLTrans}
	F^{\epsilon}(\tau) = F^{\epsilon}[f](\tau) := \int_{\R^{\times}} f(t) K_{i\tau}^{\epsilon}(t) \ud^{\times}t \left(= \int_{\R} \OFour(f)(x) \overline{e_{i\tau}^{\epsilon}(x)} \ud x\right)
\end{equation}
	are intimately related to the classical Kontorovich-Lebedev transform (see \cite[\S 6]{YL94})
	$$ K[f](\tau) := \int_0^{\infty} f(t) t^{-\frac{1}{2}} \BesselK_{i\tau}(t) \ud t. $$
	In fact, it is easy to see (recall (\ref{KirToKBessel}))
	$$ F^{\epsilon}[f](\tau) = \frac{1}{\pi} \sqrt{\frac{\cosh(\pi \tau)}{2}} \int_0^{\infty} \left( f \left( \frac{t}{2\pi} \right) + \epsilon f \left( -\frac{t}{2\pi} \right) \right) t^{-\frac{1}{2}} \BesselK_{i\tau}(t) \ud t. $$
	We shall be interested in the following subspace of Schwartz functions.
	
\begin{definition} \label{SchwartzSpaces}
	Let $\Sch_0(\R) \subset \Sch(\R)$ be the space of Schwartz functions $f$ satisfying $f^{(n)}(0)=0$ for all $n \in \Z_{\geq 0}$. Write $\Sch^0(\R)$ for the image of $\Sch_0(\R)$ under Fourier transform. Let $\Sch(\R_{\geq 0})$ be the space of restrictions of functions in $\Sch_0(\R)$ to $[0,\infty)$. Let $\Sch(\R_{>0}^{\times})$ be the space of smooth functions $f$ on $\R_{>0}$ such that for any $m,n \in \Z_{\geq 0}$
	$$ \sup_{t \in \R_{>0}} \extnorm{ (t+t^{-1})^m f^{(n)}(t) } < \infty. $$
\end{definition}
\noindent It is easy to see that for either $\epsilon = \pm 1$
	$$ \Sch_0(\R) \to \Sch(\R_{\geq 0}), \quad f \mapsto F(t) := f \left( \frac{t}{2\pi} \right) + \epsilon f \left( -\frac{t}{2\pi} \right) $$
	is a surjective map. Hence the space of functions $F^{\epsilon}[f](\tau)$ contains $K[f](\tau)$ for $f \in \Sch(\R_{\geq 0})$.
\begin{lemma} \label{MPTrans}
	For $f \in \Sch(\R_{\geq 0})$, define the transforms
	$$ g(x) := \int_0^{\infty} \sin(tx) f(t) t^{-\frac{1}{2}} \ud t, \quad h(y) := g\left( \frac{1}{2}(y-y^{-1}) \right). $$
	Then $K[f]$ is given in terms of the Mellin transform of $h$ as
	$$ \sin \left( \frac{i \pi \tau}{2} \right) K[f](\tau) = \Mellin{h}(i\tau) := \int_0^{\infty} h(y) y^{i\tau} \frac{\ud y}{y}. $$
\end{lemma}
\begin{proof}
	This is essentially \cite[(6.86)]{YL94}, based on the following integral representation \cite[(6.84)]{YL94} of $K_{i\tau}(t)$ due to A.Erdelyi
	$$ \sin \left( \frac{i \pi \tau}{2} \right) \BesselK_{i\tau}(t) = \int_0^{\infty} \sin (t \sinh(u)) \sin(\tau u) \ud u, $$
	where the integral is in the Cauchy principal sense. We also recall the estimation
	$$ \extnorm{ \int_0^N \sin(\tau u) \sin(t \sinh(u)) \ud u} \leq \frac{1}{t \cosh(N)} + \frac{1}{t} \int_0^N \frac{\tau \cosh(u) + \sinh(u)}{\cosh^2(u)} \ud u $$
	which justifies the change of order of integrations in
	$$ \sin \left( \frac{i \pi \tau}{2} \right) K[f](\tau) = \int_0^{\infty} \sin(\tau u) \left( \int_0^{\infty} \sin(t \sinh(u)) f(t) t^{-\frac{1}{2}} \ud t \right) \ud u $$
	by Lebesgue's dominated convergence theorem.
\end{proof}

\begin{definition} \label{HolSpace}
	Let $\Hol^*(\C)$ be the space of \emph{even} entire functions $m(s)$ such that for any given real number $a < b$ and $N \in \Z_{\geq 0}$, we have the following bound uniform in $\tau = \Im(s) \in \R$
	$$ m(\sigma + i\tau) \ll e^{-\frac{\pi}{2} \norm[\tau]} (1+\norm[\tau])^{-N}, \quad a \leq \sigma \leq b. $$
\end{definition}
\begin{proposition} \label{KLRange}
	If $f \in \Sch(\R_{\geq 0})$, then $K[f](is) \in \Hol^*(\C)$. Conversely, for any $m(s) \in \Hol^*(\C)$, there is $f \in \Sch(\R_{\geq 0})$ such that $m(s) = K[f](is)$.
\end{proposition}
\begin{proof}
	When $f$ traverses $\Sch(\R_{\geq 0})$, $f(t)t^{-\frac{1}{2}}$ also traverses $\Sch(\R_{\geq 0})$, and can be extended to an odd function $f_- \in \Sch(\R)$ whose derivatives $f_-^{(n)}(0)=0$ vanishes at $0$. Using notation from Lemma \ref{MPTrans},
	$$ g(x) = \frac{1}{2i} \int_{-\infty}^{\infty} e^{itx} f_-(t) \ud t $$
traverses odd functions in $\Sch(\R)$ such that
	$$ \int_{-\infty}^{\infty} g(x) x^{2n+1} \ud x = 0, \quad \forall n \in \Z_{\geq 0}. $$
	It is easy to see that $\R_{>0} \to \R, t \mapsto \frac{1}{2}(t-t^{-1})$ is a $\Cont^{\infty}$-diffeomorphism. Hence $h(y) = g \left( \frac{1}{2}(y-y^{-1}) \right)$ traverses functions in $\Sch(\R_{>0}^{\times})$ satisfying $h(y^{-1}) = -h(y)$ and
	$$ \int_0^{\infty} h(y) (y-y^{-1})^{2n+1} (y+y^{-1}) \frac{\ud y}{y} = 0, \quad \forall n \in \Z_{\geq 0}. $$
	By induction in $n$, the last condition above is equivalent to
	$$ \int_0^{\infty} h(y) y^{2n} \frac{\ud y}{y} = \pm \int_0^{\infty} h(y) y^{-2n} \frac{\ud y}{y} = 0, \quad \forall n \in \Z_{\geq 0}. $$
	Hence $\sin \left( \frac{\pi s}{2} \right) K[f](-is) = \Mellin{h}(s)$ traverses the space of odd entire functions 
\begin{itemize}
	\item with rapid decay in $\Im(s)$ as $\Re(s)$ lying in any given compact interval,
	\item satisfying $\Mellin{h}(2n) = 0$ for $n \in \Z$.
\end{itemize}
	Since $\sin \left( \frac{\pi s}{2} \right)$ is odd and has precisely $s=2n, n \in \Z$ as simple zeroes, this is equivalent to $K[f](-is) \in \Hol^*(\C)$.
\end{proof}
\begin{corollary} \label{AdmWtR}
	For either $\epsilon \in \{ \pm 1 \}$, the space of functions $F^{\epsilon}[f](\tau)$ contains functions of the form 
	$$ \sqrt{\frac{e^{\pi \tau} + e^{-\pi \tau}}{4\pi}} \cdot m(i\tau) $$
	where $m(s) \in \Hol^*(\C)$ (see Definition \ref{HolSpace}).
\end{corollary}
\begin{proof}
	This follows directly from Proposition \ref{KLRange}.
\end{proof}

\begin{definition} \label{def: AdmWtR}
	An entire function $f(s)$ is an \emph{admissible weight function}, written as $f \in \mathrm{Ad}_{\R}(\C)$, if it is a finite linear combination of functions of the form $\cos(\pi s) m_1(s) m_2(s)$, where $m_j(s)$ are even entire functions such that $e^{\frac{\pi}{2}\norm[\Im s]} m_j(s)$ has rapid decay in any vertical region of the form $a \leq \Re s \leq b$.
\end{definition}

\begin{proposition} \label{prop: AdmWtR}
	For any pair of $f_{\epsilon} \in \mathrm{Ad}_{\R}(\C)$ with $\epsilon \in \{ \pm 1 \}$, there is $\Psi \in \Sch(\GL_2(\R))$ such that $h(\pi(\sgn^{\frac{\epsilon - 1}{2}}, i\tau))(\Psi) = f_{\epsilon}(i\tau)$.
\end{proposition}
\begin{proof}
	This follows readily from Corollary \ref{M3Exp} \& \ref{AdmWtR}.
\end{proof}

	\subsection{Dual Weight Formula}
	
	We would like to express $\widetilde{h}(\chi)(\Psi)$ in terms of $h(\pi)(\Psi)$ for $\Psi = \phi_1^{\vee} * \phi_2$ as constructed in (\ref{TestPhiR}) above. By Lemma \ref{GeomToM4}, we are reduced to expressing $I(t,\Psi)$ in terms of $h(\pi)(\Psi)$. The idea is to view $I(t,\Psi)$ as a tempered distribution on $\Sch(\R)$. Hence we take $g \in \Sch(\R)$ and consider
	$$ b(\Psi, g) := \int_{\GL_2(\R)} \Psi(X) g(\Minv(X)) \norm[\det X]^{-1} \ud X, \quad \Minv \left( \begin{pmatrix} x_1 & x_2 \\ x_3 & x_4 \end{pmatrix} \right) := \frac{x_1 x_4}{x_1 x_4 - x_2 x_3}, $$
	where $\ud X =  \ud x_1 \ud x_2 \ud x_3 \ud x_4$ is the additive Haar measure on $\Mat_2(\R)$, with which a Haar measure 
	$$ \ud g := \frac{\ud X}{\norm[\det X]^2} $$
	on $\GL_2(\R)$ is associated. One checks easily that it is consistent with the Haar measure given via the Iwasawa decomposition. We shall compute $b(\Psi, g)$ in two ways. The first one is the following lemma.
\begin{lemma} \label{Bilin1}
	For any $\Psi \in \Sch(\Mat_2(\R))$, we have
	$$ b(\Psi, g) = \int_{\R} I \left( \frac{y+\frac{1}{2}}{y-\frac{1}{2}}, \Psi \right) g\left( y+\frac{1}{2} \right) \extnorm{y-\frac{1}{2}}^{-1} \ud y. $$
\end{lemma}
\begin{proof}
	The same change of variables as in the proof of Lemma \ref{GeomToM4} gives
	$$ b(\Psi, g) = \int_{\R} I(t,\Psi) g \left( \frac{t}{t-1} \right) \norm[t-1]^{-1} \ud t. $$
	Another change of variables $t = (y+1/2)(y-1/2)^{-1}$ gives the desired equality.
\end{proof}

\noindent To derive another form of $b(\Psi, g)$ and link it to $h(\pi)(\Psi)$, we specialize $\Psi$ to the form $\Psi = \phi_1^{\vee} * \phi_2$. Using the coordinates of $\GL_2(\R)$ given in (\ref{TestPhiR}) (a variant of Iwasawa decomposition), we get
\begin{multline} \label{eq:b}
	b(\Psi, g) = \int_{\GL_2(\R)^2} \phi_1(g_1) \phi_2(g_2) g \left( \Minv(g_1^{-1} g_2) \right) \norm[\det g_1^{-1}g_2] \ud g_1 \ud g_2 \\
	= \int_{\R^{\times}} \ell_1(z) \norm[z]^{-2} \ud^{\times}z \cdot \int_{\R^{\times}} \ell_2(z) \norm[z]^2 \ud^{\times}z \cdot \int_0^{\infty} h_1(y) \ud^{\times}y \cdot \int_0^{\infty} h_2(y) \ud^{\times}y \cdot \\
	\int_{\R^2} f_1(x_1) f_2(x_2) \left( \int_0^{2\pi} g \left( \Minv \left( n(-x_1)\begin{pmatrix} \cos \theta & \sin \theta \\ - \sin \theta & \cos \theta \end{pmatrix} n(x_2) \right) \right) \frac{d\theta}{2\pi} \right) \ud x_1 \ud x_2.
\end{multline}

\begin{lemma} \label{KerFuncInv}
	To any $g \in \Cont^2(\R)$ we associate a kernel function 
	$$ K_g(x_1,x_2) := \int_0^{2\pi} g \left( \Minv \left( n(-x_1)\begin{pmatrix} \cos \theta & \sin \theta \\ - \sin \theta & \cos \theta \end{pmatrix} n(x_2) \right) \right) \frac{d\theta}{2\pi}. $$
	Recall the differential operator $D$ in (\ref{InvOp}) and write $D_j$ for its counterpart with respect to the variable $x_j$. Then $K_g$ is in $\Cont^2(\R^2)$ and satisfies
	$$ K_g(-x_1,-x_2) = K_g(x_1,x_2) = K_g(x_2,x_1), \quad D_2 K_g(x_1,x_2) = D_1 K_g(x_1,x_2). $$
\end{lemma}
\begin{proof}
	It is clear that $K_g(x_1,x_2)$ is in $\Cont^2(\R^2)$ since
	$$ \Minv := \Minv \left( n(-x_1)\begin{pmatrix} \cos \theta & \sin \theta \\ - \sin \theta & \cos \theta \end{pmatrix} n(x_2) \right) = (\cos \theta + x_1 \sin \theta) (\cos \theta - x_2 \sin \theta) $$
	is a smooth function in $(\theta, x_1, x_2)$. For $j=1,2$, we have (note that $\Minv$ is linear in $x_j$)
	$$ D_j g(\Minv) = (x_j^2+1) \left( \frac{\partial \Minv}{\partial x_j} \right)^2 g''(\Minv) + 2x_j \frac{\partial \Minv}{\partial x_j} g'(\Minv). $$
	A little calculation shows
	$$ (x_1^2+1) \left( \frac{\partial \Minv}{\partial x_1} \right)^2 - (x_2^2+1) \left( \frac{\partial \Minv}{\partial x_2} \right)^2 = (x_1+x_2) (\sin \theta)^2 \frac{\partial \Minv}{\partial \theta}, $$
	$$ 2x_1 \frac{\partial \Minv}{\partial x_1} - 2x_2 \frac{\partial \Minv}{\partial x_2} = 2 (x_1+x_2) \sin \theta \cos \theta. $$
	Consequently, we get
	$$ D_1 g(\Minv) - D_2 g(\Minv) = (x_1+x_2) \frac{\partial}{\partial \theta} \left\{ (\sin \theta)^2 g'(\Minv) \right\}. $$
	It follows that
	$$ \left( D_1 K_g - D_2 K_g \right)(x_1,x_2) = (x_1+x_2) \int_0^{2\pi} \frac{\partial}{\partial \theta} \left\{ (\sin \theta)^2 g'(\Minv) \right\} d\theta = 0. $$
	Since changing $(x_1,x_2, \theta)$ to $(-x_1,-x_2, -\theta)$ or $(x_2,x_2,-\theta)$ leaves $\Minv$ invariant, we get the first equation and conclude.
\end{proof}
\begin{remark}
	In other words, the operator from $\mathcal{D}(\R) = \Cont_c^{\infty}(\R)$ to $\mathcal{D}'(\R)$ (the space of distributions on $\R$) defined by $K_g(x_1,x_2)$ commutes with $D$ and with the involution $f(x) \mapsto f(-x)$. Hence it is expected that all $K_g$ are simultaneously diagonalizable on the space of even or odd functions.
\end{remark}
\begin{remark}
	It is easy to deduce and is worth writing another expression of
\begin{equation} \label{eq: KgIntRep}
	K_g(x_1,x_2) = \int_0^{\pi} g \left( \frac{1-x_1x_2}{2} + \frac{\cos \theta}{2} \sqrt{(1+x_1^2)(1+x_2^2)} \right) \frac{d\theta}{\pi}. 
\end{equation}
\end{remark}

\begin{lemma} \label{KerFuncEst}
	For $g \in \Sch(\R)$ and any $\epsilon > 0$, we have the bounds
	$$ K_g(x_1,x_2) \ll_{\epsilon, x_1} (1+\norm[x_2])^{-1+\epsilon}, \quad \frac{\partial K_g}{\partial x_1}(x_1,x_2) \ll_{\epsilon, x_1} (1+\norm[x_2])^{-1+\epsilon} $$
	with implied constants bounded by some Schwartz norm of $g$.
\end{lemma}
\begin{proof}
	Note that we may and will assume $\norm[x_2]$ to be larger than any constant depending only on $\epsilon, x_1$.

\noindent (1) We first treat the bound for $K_g$. For simplicity of notations, write
	$$ a(x_1,x_2) := \frac{1}{2} \sqrt{(1+x_1^2)(1+x_2^2)}, \quad \epsilon(x_1,x_2) := \frac{1-x_1x_2}{\sqrt{(1+x_1^2)(1+x_2^2)}}. $$
	The integrand in \eqref{eq: KgIntRep} becomes $g(a(x_1,x_2)(\cos \theta + \epsilon(x_1,x_2)))$. Introduce $\widetilde{\epsilon_0} = \min (1/16, \epsilon_0)$ with
	$$ \epsilon_0 := \frac{1}{4} \left( 1 - \frac{\norm[x_1]}{\sqrt{1+x_1^2}} \right) = \frac{1}{4 \sqrt{1+x_1^2} (\norm[x_1]+\sqrt{1+x_1^2})}. $$
	Note that $\lim_{x_2 \to \pm \infty} \epsilon(x_1,x_2) \in \{ 1-4\epsilon_0, -1+4\epsilon_0 \}$. For any constant $0 < \delta <1$, we can assume $\norm[x_2]$ to be sufficiently large, so that
	$$ \epsilon(x_1,x_2) \in [-1+3\widetilde{\epsilon_0}, 1-3\widetilde{\epsilon_0}], \quad \frac{(1+x_2^2)^{\frac{\delta}{2}}}{a(x_1,x_2)} \leq \widetilde{\epsilon_0}. $$
	Let $I_+$ resp. $I_-$ be the subset of $[0,\pi)$ such that
	$$ \theta \in I_+ \text{ resp. } I_- \Leftrightarrow \extnorm{\cos \theta + \epsilon(x_1,x_2)} \geq \text{ resp. } < \frac{(1+x_2^2)^{\frac{\delta}{2}}}{a(x_1,x_2)}. $$
	In particular, $I_-$ is an interval of length $\norm[I_-] \ll_{\delta, x_1} (1+x_2^2)^{\frac{\delta-1}{2}}$. Bounding $g$ by its sup-norm we get
	$$ \int_{I_-} g \left( a(x_1,x_2) ( \cos \theta + \epsilon(x_1,x_2) ) \right) \frac{d\theta}{\pi} \ll_{\delta, x_1} (1+x_2^2)^{\frac{\delta-1}{2}}. $$
	While on $I_+$ we have the bound $g \left( a(x_1,x_2) ( \cos \theta + \epsilon(x_1,x_2) ) \right) \ll (1+x_2^2)^{-\frac{A\delta}{2}}$ for any constant $A > 0$, we deduce
	$$ \int_{I_+} g \left( a(x_1,x_2) ( \cos \theta + \epsilon(x_1,x_2) ) \right) \frac{d\theta}{\pi} \ll_{\delta, A, x_1} (1+x_2^2)^{-\frac{A\delta}{2}} $$
	and conclude by choosing $\delta$ sufficiently small and $A$ sufficiently large.

\noindent (2) We maintain the notation and assumption in (1). Note that
	$$ \frac{\partial K_g}{\partial x_1}(x_1,x_2) = \int_0^{\pi} \left( -\frac{x_2}{2} + \frac{\cos \theta}{2} \frac{x_1}{\sqrt{1+x_1^2}} \sqrt{1+x_2^2} \right) g' \left( a(x_1,x_2) (\cos \theta + \epsilon(x_1,x_2)) \right) \frac{d\theta}{\pi}. $$
	Write $I_- = (\theta_-, \theta_+)$. Then we have
	$$ 1-2\widetilde{\epsilon_0} \geq \cos \theta_- = \frac{(1+x_2^2)^{\frac{\delta}{2}}}{a(x_1,x_2)} - \epsilon(x_1,x_2) > \cos \theta_+ = -\frac{(1+x_2^2)^{\frac{\delta}{2}}}{a(x_1,x_2)} - \epsilon(x_1,x_2) \geq -1 + 2 \widetilde{\epsilon_0}. $$
	By integration by parts, we get for any constant $A > 0$
\begin{align*}
	&\quad \int_{I_-} \left( -\frac{x_2}{2} + \frac{\cos \theta}{2} \frac{x_1}{\sqrt{1+x_1^2}} \sqrt{1+x_2^2} \right) g' \left( a(x_1,x_2) (\cos \theta + \epsilon(x_1,x_2)) \right) \frac{d\theta}{\pi} \\
	&= \frac{\frac{x_2}{\sqrt{1+x_2^2}} - \frac{x_1}{\sqrt{1+x_1^2}} \cos \theta_+}{\pi \sqrt{1+x_1^2} \sin \theta_+} g(-(1+x_2^2)^{\frac{\delta}{2}}) - \frac{\frac{x_2}{\sqrt{1+x_2^2}} - \frac{x_1}{\sqrt{1+x_1^2}} \cos \theta_-}{\pi \sqrt{1+x_1^2} \sin \theta_-} g((1+x_2^2)^{\frac{\delta}{2}}) \\
	&\quad - \int_{I_-} \frac{\frac{x_1}{\sqrt{1+x_1^2}} - \frac{x_2}{\sqrt{1+x_2^2}} \cos \theta}{\sqrt{1+x_1^2} \sin^2 \theta} g\left( a(x_1,x_2) (\cos \theta + \epsilon(x_1,x_2)) \right) \frac{d\theta}{\pi} \\
	&\ll_{\delta, A, x_1} (1+x_2^2)^{-\frac{A\delta}{2}} + (1+x_2^2)^{\frac{\delta - 1}{2}}.
\end{align*}
	While on $I_+$ we have the bound $g' \left( a(x_1,x_2) ( \cos \theta + \epsilon(x_1,x_2) ) \right) \ll (1+x_2^2)^{-\frac{A\delta}{2}}$ for any constant $A > 0$, we deduce
	$$ \int_{I_+} \left( -\frac{x_2}{2} + \frac{\cos \theta}{2} \frac{x_1}{\sqrt{1+x_1^2}} \sqrt{1+x_2^2} \right) g' \left( a(x_1,x_2) (\cos \theta + \epsilon(x_1,x_2)) \right) \frac{d\theta}{\pi} \ll_{\delta, A, x_1} (1+x_2^2)^{\frac{1-A\delta}{2}} $$
	and conclude by choosing $\delta$ sufficiently small and $A$ sufficiently large.
\end{proof}

\begin{remark}
	The implicit constants in the above bounds depends polynomially on $x_1$.
\end{remark}

\begin{lemma} \label{DualWtFC2}
	Let $g \in \Sch(\R)$. Then for any $f_1,f_2 \in \Sch^0(\R)$ we have
	$$ \int_{\R} K_g(x_1,x_2) \overline{e_{i\tau}^{\epsilon}(x_2)} \ud x_2 = \lambda_g^{\epsilon}(\tau) \overline{e_{i\tau}^{\epsilon}(x_1)}, $$
	$$ \int_{\R^2} f_1(x_1) K_g(x_1,x_2) f_2(x_2) \ud x_1 \ud x_2 = \frac{1}{2\pi} \sum_{\epsilon = \pm 1} \int_{-\infty}^{\infty} \tau \tanh(\pi \tau) \lambda_g^{\epsilon}(\tau) \cdot \epsilon F^{\epsilon}[\OFour(f_1)](\tau) F^{\epsilon}[\OFour(f_2)](\tau) \ud \tau, $$
	where $\lambda_g^{\epsilon}$ is given by (recall $F_{\pm}(\tau,t)$ given in (\ref{DualEigenFunc}))
	$$ \lambda_g^+(\tau) = \int_{\R} K_g(0,t) F_+(\tau,t) \ud t, \quad \lambda_g^-(\tau) = \int_{\R} \frac{\partial K_g}{\partial x_1}(0,t) F_-(\tau,t) \ud t $$
	with absolutely convergent integrals. The maps $g \mapsto \lambda_g^{\pm}(\tau)$ are tempered distributions.
\end{lemma}
\begin{proof}
	The inversion formula for Kontorovich-Lebedev transform (see \cite[(6.3)]{YL94} for example) implies
\begin{equation} \label{eq: KLTransInv}
	\OFour(f)(x) = \frac{1}{2\pi} \sum_{\epsilon = \pm 1} \int_{-\infty}^{\infty} \tau \tanh(\pi \tau) F^{\epsilon}[\OFour(f)](\tau) K_{i\tau}^{\epsilon}(x) \ud \tau \cdot \norm[x]^{-1}. 
\end{equation}
	Proposition \ref{KLRange} ensures the absolute convergence. Taking Fourier inversion yields
	$$ f(x) = \frac{1}{2\pi} \sum_{\epsilon = \pm 1} \int_{-\infty}^{\infty} \tau \tanh(\pi \tau) F^{\epsilon}[\OFour(f)](\tau) \overline{e_{i\tau}^{\epsilon}(x)} \ud \tau, $$
	Applying the above formula for $f_2(x_2)$ gives
\begin{align*} 
	&\quad \int_{\R^2} f_1(x_1) K_g(x_1,x_2) f_2(x_2) \ud x_1 \ud x_2 \\
	&= \frac{1}{2\pi} \sum_{\epsilon = \pm 1} \int_{-\infty}^{\infty} \tau \tanh(\pi \tau) F^{\epsilon}[\OFour(f_2)](\tau) \left( \int_{\R^2} f_1(x_1) K_g(x_1,x_2) \overline{e_{i\tau}^{\epsilon}(x_2)} \ud x_1 \ud x_2 \right) \ud \tau.
\end{align*}
	By \cite[(1.19) \& (1.20)]{Du13}, we have $F_{\pm}(\tau, t) \ll \norm[x]^{-1/2}$ as $x \to \pm \infty$. By Lemma \ref{KerFuncInv} \& \ref{KerFuncEst}, the smooth function $h(x_1) := \int_{\R} K_g(x_1,x_2) \overline{e_{i\tau}^{\epsilon}(x_2)} \ud x_2$ satisfies
	$$ Dh(x_1) = \int_{\R} D_2 K_g(x_1,x_2) \overline{e_{i\tau}^{\epsilon}(x_2)} \ud x_2 = \int_{\R} K_g(x_1,x_2) D\overline{e_{i\tau}^{\epsilon}(x_2)} \ud x_2 = \left( \frac{1}{4}+\tau^2 \right) h(x_1), $$
	$$ h(-x_1) = \int_{\R} K_g(x_1,-x_2) \overline{e_{i\tau}^{\epsilon}(x_2)} \ud x_2 = \epsilon h(x_1). $$
	Hence $h(x_1) = \lambda_g^{\epsilon}(\tau) \overline{e_{i\tau}^{\epsilon}(x_1)}$ for some $\lambda_g^{\epsilon}(\tau) \in \C$ by the uniqueness of the solutions, proving the first equation. Since $e_{i\tau}^{\epsilon}$ is proportional to $F_{\epsilon}$ given in (\ref{DualEigenFunc}), we have
	$$ \lambda_g^{\epsilon}(\tau) F_{\epsilon}(\tau,x_1) = \int_{\R} K_g(x_1,x_2) F_{\epsilon}(\tau, x_2) \ud x_2. $$
	Evaluating the above equation resp. its derivative at $x_1=0$ for $\epsilon = +1$ resp. $\epsilon = -1$, we get the desired formulas for $\lambda_g^{\epsilon}(\tau)$. Recalling the definition \eqref{NormKLTrans} of $F^{\epsilon}[\cdot]$ we deduce the second equation by
\begin{align*}
	\int_{\R^2} f_1(x_1) K_g(x_1,x_2) f_2(x_2) \ud x_1 \ud x_2 &= \frac{1}{2\pi} \sum_{\epsilon = \pm 1} \int_{-\infty}^{\infty} \tau \tanh(\pi \tau) F^{\epsilon}[\OFour(f_2)](\tau) \left( \lambda_g^{\epsilon}(\tau) \int_{\R} f_1(x_1) \overline{e_{i\tau}^{\epsilon}(x_1)} \right) \ud \tau \\
	&= \frac{1}{2\pi} \sum_{\epsilon = \pm 1} \epsilon \int_{-\infty}^{\infty} \tau \tanh(\pi \tau) \lambda_g^{\epsilon}(\tau) F^{\epsilon}[\OFour(f_1)](\tau)F^{\epsilon}[\OFour(f_2)](\tau) \ud \tau.
\end{align*}
	The statement on the maps $g \mapsto \lambda_g^{\pm}(\tau)$ follows directly from Lemma \ref{KerFuncEst}.
\end{proof}

\begin{lemma} \label{VarMF}
\begin{itemize}
	\item[(1)] For $s \in \C$ with $\Re s > \frac{1}{2}$, we have
	$$ \int_{\R} (1+x^2)^{-\frac{s}{2}} F_+(\tau,x) \ud x = \Gamma \left( \frac{1}{2} \right) \frac{\Gamma \left( \frac{s+i\tau}{2} - \frac{1}{4} \right) \Gamma \left( \frac{s-i\tau}{2} - \frac{1}{4} \right)}{\Gamma \left( \frac{s}{2} \right)^2}. $$
	\item[(2)] For $s \in \C$ with $\Re s > \frac{3}{2}$, we have
	$$ \int_{\R} (1+x^2)^{-\frac{s}{2}} x F_-(\tau,x) \ud x = \Gamma \left( \frac{3}{2} \right) \frac{\Gamma \left( \frac{s+i\tau}{2} - \frac{3}{4} \right) \Gamma \left( \frac{s-i\tau}{2} - \frac{3}{4} \right)}{\Gamma \left( \frac{s}{2} \right)^2}. $$
	\item[(3)] Define a function $Q_+(t; \tau) = Q_+^1(t; \tau) \mathbbm{1}_{\norm[t] < \frac{1}{2}} +(Q_+^2(t; \tau) + Q_+^2(t; -\tau)) \mathbbm{1}_{\norm[t] > \frac{1}{2}}$ with
	$$ Q_+^1(t; \tau) = \frac{2}{\pi} \cdot \HyGI \left( \frac{1}{4}+\frac{i\tau}{2}, \frac{1}{4}-\frac{i\tau}{2}; \frac{1}{2}; 4t^2 \right), $$
	$$ Q_+^2(t; \tau) = \frac{1}{2\pi} \frac{1 + i\sinh(\pi \tau)}{i\sinh(\pi \tau)} \frac{2^{\frac{1}{2} + i\tau}}{\norm[t]^{\frac{1}{2} - i\tau}} \HyGI \left( \frac{1}{4} - \frac{i\tau}{2}, \frac{3}{4} - \frac{i\tau}{2}; 1 - i\tau; \frac{1}{4t^2} \right). $$
	Then the (tempered) distribution $g \mapsto \lambda_g^+(\tau)$ is represented by $Q_+(t; \tau)$ in the sense that
	$$ \lambda_g^+(\tau) = \int_{\R} g\left( t + \frac{1}{2} \right) Q_+(t;\tau) \ud t. $$
	\item[(4)] Define a function $Q_-(t; \tau) = Q_-^1(t; \tau) \mathbbm{1}_{\norm[t] < \frac{1}{2}} +(Q_-^2(t; \tau) + Q_-^2(t; -\tau)) \mathbbm{1}_{\norm[t] > \frac{1}{2}}$ with
	$$ Q_-^1(t; \tau) = \frac{4}{\pi} t \cdot \HyGI \left( \frac{3}{4}+\frac{i\tau}{2}, \frac{3}{4}-\frac{i\tau}{2}; \frac{3}{2}; 4t^2 \right), $$
	$$ Q_-^2(t; \tau) =  \frac{-1 + i\sinh(\pi \tau)}{2\pi i\sinh(\pi \tau)} \frac{2^{\frac{1}{2} + i\tau}}{\norm[t]^{\frac{1}{2} - i\tau}} \sgn(t) \HyGI \left( \frac{1}{4} - \frac{i\tau}{2}, \frac{3}{4} - \frac{i\tau}{2}; 1 - i\tau; \frac{1}{4t^2} \right). $$
	Then the (tempered) distribution $g \mapsto \lambda_g^-(\tau)$ is represented by $Q_-(t; \tau)$ in the sense that
	$$ \lambda_g^-(\tau) = \int_{\R} g\left( t + \frac{1}{2} \right) Q_-(t;\tau) \ud t. $$
\end{itemize}
\end{lemma}
\begin{proof}
	(1) We apply Pfaff's relation (see \cite[15.8.1]{OLBC10}) and insert the Taylor expansion
\begin{align*} 
	F_+(\tau,x) &= \HyG \left( \frac{1}{4}-\frac{i\tau}{2}, \frac{1}{4}+\frac{i\tau}{2}; \frac{1}{2}; -x^2 \right) = (1+x^2)^{-\left( \frac{1}{4} + \frac{i\tau}{2} \right)} \HyG \left( \frac{1}{4}+\frac{i\tau}{2}, \frac{1}{4}+\frac{i\tau}{2}; \frac{1}{2}; \frac{x^2}{1+x^2} \right) \\
	&= \frac{\Gamma \left( \frac{1}{2} \right)}{\Gamma \left( \frac{1}{4}+\frac{i\tau}{2} \right)^2} \sum_{j=0}^{\infty} \frac{\Gamma \left( \frac{1}{4}+\frac{i\tau}{2}+j \right)^2}{\Gamma \left( \frac{1}{2}+j \right) j!} \frac{x^{2j}}{(1+x^2)^{j+\frac{1}{4}+\frac{i\tau}{2}}}.
\end{align*}
	Since we have (see \cite[5.12.3]{OLBC10})
	$$ \int_{\R} \frac{x^{2j}}{(1+x^2)^{j+\frac{1}{4}+\frac{i\tau}{2}+\frac{s}{2}}} \ud x = \int_0^{\infty} \frac{y^{j-\frac{1}{2}}}{(1+y)^{j+\frac{1}{4}+\frac{i\tau}{2}+\frac{s}{2}}} \ud y = \frac{\Gamma \left( j+\frac{1}{2} \right) \Gamma \left( \frac{s}{2}-\frac{1}{4}+\frac{i\tau}{2} \right)}{\Gamma \left( \frac{s}{2}+\frac{1}{4}+\frac{i\tau}{2}+j \right)}, $$
	we can change the order of the integration and the summation for $\Re s > 1/2$ to obtain
	$$ \int_{\R} (1+x^2)^{-\frac{s}{2}} F_+(\tau,x) \ud x = \frac{\Gamma \left( \frac{1}{2} \right) \Gamma \left( \frac{s}{2}-\frac{1}{4}+\frac{i\tau}{2} \right)}{\Gamma \left( \frac{s}{2}+\frac{1}{4}+\frac{i\tau}{2} \right)} \HyG \left( \frac{1}{4}+\frac{i\tau}{2}, \frac{1}{4}+\frac{i\tau}{2}; \frac{s}{2}+\frac{1}{4}+\frac{i\tau}{2}; 1 \right) $$
	and conclude the first equation by Gauss's summation theorem (see \cite[15.1.20]{Ab72} for example).
	
\noindent (2) The proof is quite similar to the one of (1). We only remark
\begin{align*}
	x F_-(\tau,x) &= x^2 \HyG \left( \frac{3}{4}-\frac{i\tau}{2}, \frac{3}{4}+\frac{i\tau}{2}; \frac{3}{2}; -x^2 \right) = \frac{x^2}{(1+x^2)^{\frac{3}{4} + \frac{i\tau}{2}}} \HyG \left( \frac{3}{4}+\frac{i\tau}{2}, \frac{3}{4}+\frac{i\tau}{2}; \frac{3}{2}; \frac{x^2}{1+x^2} \right) \\
	&= \frac{\Gamma \left( \frac{3}{2} \right)}{\Gamma \left( \frac{3}{4}+\frac{i\tau}{2} \right)^2} \sum_{j=0}^{\infty} \frac{\Gamma \left( \frac{3}{4}+\frac{i\tau}{2}+j \right)^2}{\Gamma \left( \frac{3}{2}+j \right) j!} \frac{x^{2j+2}}{(1+x^2)^{j+\frac{3}{4}+\frac{i\tau}{2}}}
\end{align*}
	and leave the details to the reader. 
	
\noindent (3) For the second equation, note that for $a > 0, b \in \R$, the following formula is obvious
	$$ \int_0^{\infty} \left( g \left( \frac{t}{a} + b \right) + g \left( -\frac{t}{a} + b \right) \right) t^s \frac{\ud t}{t} = a^s \int_{\R^{\times}} g(t+b) \norm[t]^s \ud^{\times}t. $$
	Applying Mellin inversion formula, we get for $0 < c < 1$
\begin{align*} 
	K_g(x_1,x_2) &= \int_0^{\pi} g \left( \frac{1-x_1x_2}{2} + \frac{\cos \theta}{2} \sqrt{(1+x_1^2)(1+x_2^2)} \right) \frac{d\theta}{\pi} \\
	&= \int_{(c)} \left( \int_{\R^{\times}} g \left( t+\frac{1-x_1x_2}{2} \right) \norm[t]^s \ud^{\times}t \cdot \int_0^{\frac{\pi}{2}} \left( \frac{2}{\cos \theta \sqrt{(1+x_1^2)(1+x_2^2)}} \right)^s \frac{d\theta}{\pi} \right) \frac{\ud s}{2\pi i} \\
	&= \int_{(c)} \frac{2^{s-1} \Gamma \left( \frac{1-s}{2} \right)}{\sqrt{\pi} \Gamma \left( 1-\frac{s}{2} \right)} [(1+x_1^2)(1+x_2^2)]^{-\frac{s}{2}} \int_{\R^{\times}} g \left( t+\frac{1-x_1x_2}{2} \right) \norm[t]^s \ud^{\times}t \frac{\ud s}{2\pi i}.
\end{align*}
	It follows that for $1/2 < c < 1$, we get by (1)
\begin{align*}
	\lambda_g^+(\tau) &= \int_{\R} K_g(0,x) F_+(\tau,x) \ud x \\
	&= \int_{(c)} 2^{s-1} \frac{\Gamma \left( \frac{1-s}{2} \right) \Gamma \left( \frac{s+i\tau}{2}-\frac{1}{4} \right) \Gamma \left( \frac{s-i\tau}{2}-\frac{1}{4} \right)}{\Gamma \left( \frac{s}{2} \right)^2 \Gamma \left( 1-\frac{s}{2} \right)} \left( \int_{\R^{\times}} g\left( t + \frac{1}{2} \right) \norm[t]^s \ud^{\times}t \right) \frac{\ud s}{2\pi i} \\
	&= \int_{(c)} \frac{\sin \left( \frac{\pi s}{2} \right) \Gamma \left( \frac{1-s}{2} \right) \Gamma \left( \frac{s+i\tau}{2}-\frac{1}{4} \right) \Gamma \left( \frac{s-i\tau}{2}-\frac{1}{4} \right)}{2\pi \Gamma \left( \frac{s}{2} \right)} \left( \int_{\R^{\times}} g\left( t + \frac{1}{2} \right) (2\norm[t])^s \ud^{\times}t \right) \frac{\ud s}{2\pi i}.
\end{align*}
	If $g$ has support contained in $\left\{ t+1/2 \ \middle| \ \norm[t] < 1/2 \right\}$, then we can shift the contour to $+\infty$, picking up the residues at $s \in 1+2\Z_{\geq 0}$ and get
\begin{align*} 
	\lambda_g^+(\tau) &= \frac{1}{\pi} \sum_{n=0}^{\infty} \frac{\Gamma \left( n+\frac{1}{4}+\frac{i\tau}{2} \right) \Gamma \left( n+\frac{1}{4}-\frac{i\tau}{2} \right)}{\Gamma \left( n+\frac{1}{2} \right) n!} \int_{\norm[t] < \frac{1}{2}} g\left( t + \frac{1}{2} \right) (2\norm[t])^{2n+1} \ud^{\times}t \\
	&=  \int_{\norm[t]<\frac{1}{2}} g\left( t + \frac{1}{2} \right) Q_+^1(t;\tau) \ud t.
\end{align*}
	If $g$ has support contained in $\left\{ t+1/2 \ \middle| \ \norm[t] > 1/2 \right\}$, then we can shift the contour to $-\infty$, picking up the residues at $s \in 1/2\pm i\tau - 2\Z_{\geq 0}$ and get (note $\sin (ix) = i \sinh(x)$)
\begin{align*} 
	\lambda_g^+(\tau) &= \frac{1}{\pi} \sum_{\pm} \sum_{n=0}^{\infty} \frac{(-1)^n \sin \left( \pi \left( \frac{1}{4} \pm \frac{i\tau}{2} - n \right) \right) \Gamma \left( n +\frac{1}{4} \mp \frac{i\tau}{2} \right) \Gamma \left( \pm i\tau -n \right) }{\Gamma \left( \frac{1}{4} \pm \frac{i\tau}{2} -n \right) n!} \cdot \\
	&\qquad \qquad \quad \int_{\norm[t] > \frac{1}{2}} g\left( t + \frac{1}{2} \right) (2\norm[t])^{\frac{1}{2} \pm i\tau -2n} \ud^{\times}t \\
	&= \frac{1}{\pi} \sum_{\pm} \sum_{n=0}^{\infty} \frac{\sin^2 \left( \pi \left( \frac{1}{4} \pm \frac{i\tau}{2} \right) \right) \Gamma \left( n +\frac{1}{4} \mp \frac{i\tau}{2} \right) \Gamma \left( n + \frac{3}{4} \mp \frac{i\tau}{2} \right)}{\sin \left( \pm i \pi \tau \right) \Gamma \left( n+1 \mp i\tau \right) n!} \int_{\norm[t] > \frac{1}{2}} g\left( t + \frac{1}{2} \right) (2\norm[t])^{\frac{1}{2} \pm i\tau -2n} \ud^{\times}t \\
	&=  \int_{\norm[t]>\frac{1}{2}} g\left( t + \frac{1}{2} \right) (Q_+^2(t;\tau) + Q_+^2(t;-\tau)) \ud t.
\end{align*}
	Let $\ell(g) := \lambda_{\tilde{g}}^+(\tau)$ where $\tilde{g}(t) := g(t+1/2)$. Then $D := \ell - Q_+(t;\tau)\ud t$ is an even tempered distribution with support contained in $\{ \pm 1/2 \}$. Thus $D$ must be a finite linear combination of $\delta_{1/2}^{(n)} + \delta_{-1/2}^{(n)}$, where $\delta_t$ is the Dirac measure at $t$. On the one hand, the (even) Mellin transform (on $\R^{\times}$) for $1/2 < \Re s < 1$, in the sense of distributions is
	$$ \Mellin[+]{D}(-s) = \pi 2 ^{s-1} \frac{\Gamma \left( \frac{1-s}{2} \right) \Gamma \left( \frac{s+i\tau}{2}-\frac{1}{4} \right) \Gamma \left( \frac{s-i\tau}{2}-\frac{1}{4} \right)}{\Gamma \left( \frac{s}{2} \right)^2 \Gamma \left( 1-\frac{s}{2} \right)} - \frac{1}{2} \int_{\R^{\times}} Q_+(t;\tau) \norm[t]^{1-s} \ud^{\times}t. $$
	Due to the usual asymptotic estimation of hypergeometric functions, the above integral is absolutely convergent. Hence by Stirling's estimation of the gamma functions and the Riemann-Lebesgue lemma, we have $\lim_{\Im s \to \pm \infty} \Mellin[+]{D}(-s) = 0$. On the other hand, we have
	$$ \Mellin[+]{\delta_{1/2}^{(n)} + \delta_{-1/2}^{(n)}}(-s) = (-1)^n s(s+1) \cdots (s+n-1) 2^{s+n}, $$
whose finite linear combination has limit $0$ as $\Im s \to \pm \infty$ for any $1/2 < \Re s < 1$ only if it is the trivial linear combination $0$. Hence $D = 0$.

\noindent (4) Similar to (3), we have for $0 < c < 1$
\begin{align*} 
	\frac{\partial K_g}{\partial x_1}(0,x) &= -\frac{x}{2} \int_0^{\pi} g' \left( \frac{1}{2} + \frac{\cos \theta}{2} \sqrt{1+x^2} \right) \frac{d\theta}{\pi} \\
	&= -\frac{x}{2} \int_{(c)} \frac{2^{s-1} \Gamma \left( \frac{1-s}{2} \right)}{\sqrt{\pi} \Gamma \left( 1-\frac{s}{2} \right)} (1+x^2)^{-\frac{s}{2}} \int_{\R^{\times}} g' \left( t+\frac{1}{2} \right) \norm[t]^s \ud^{\times}t \frac{\ud s}{2\pi i}.
\end{align*}
	By integration by parts, we have for $\Re s > 1$
	$$ \int_{\R^{\times}} g' \left( t+\frac{1}{2} \right) \norm[t]^s \ud^{\times}t = -(s-1) \int_{\R^{\times}} g \left( t + \frac{1}{2} \right) \norm[t]^{s-1} \sgn(t) \ud^{\times}t. $$
	Thus we can shift the contour to $1 < c < 2$ and obtain
	$$ \frac{\partial K_g}{\partial x_1}(0,x) = -x \int_{(c)} \frac{2^{s-1} \Gamma \left( \frac{3-s}{2} \right)}{\sqrt{\pi} \Gamma \left( 1-\frac{s}{2} \right)} (1+x^2)^{-\frac{s}{2}} \int_{\R^{\times}} g \left( t+\frac{1}{2} \right) \norm[t]^{s-1} \sgn(t) \ud^{\times}t \frac{\ud s}{2\pi i}. $$
	It follows that for $3/2 < c < 2$ we get by (2)
\begin{align*}
	\lambda_g^-(\tau) &= \int_{\R} K_g(0,x) F_-(\tau,x) \ud x \\
	&= -\frac{1}{4} \int_{(c)}2^s \frac{\Gamma \left( \frac{3-s}{2} \right) \Gamma \left( \frac{s+i\tau}{2}-\frac{3}{4} \right) \Gamma \left( \frac{s-i\tau}{2}-\frac{3}{4} \right)}{\Gamma \left( \frac{s}{2} \right)^2 \Gamma \left( 1-\frac{s}{2} \right)} \left( \int_{\R^{\times}} g\left( t + \frac{1}{2} \right) \norm[t]^{s-1} \sgn(t) \ud^{\times}t \right) \frac{\ud s}{2\pi i} \\
	&= -\frac{1}{4} \int_{(c)}2^s \frac{\sin(\pi s/2)\Gamma \left( \frac{3-s}{2} \right) \Gamma \left( \frac{s+i\tau}{2}-\frac{3}{4} \right) \Gamma \left( \frac{s-i\tau}{2}-\frac{3}{4} \right)}{\pi \Gamma \left( \frac{s}{2} \right) } \left( \int_{\R^{\times}} g\left( t + \frac{1}{2} \right) \norm[t]^{s-1} \sgn(t) \ud^{\times}t \right) \frac{\ud s}{2\pi i}.
\end{align*}
	The rest of the proof is quite similar to the proof of (3). We leave the details to the reader.
\end{proof}

\begin{lemma} \label{TechInts}
	We introduce four integrals
	$$ I_1^-(x, \tau) = \int_0^{\infty} \HyG \left( \frac{1}{4}+\frac{i\tau}{2}, \frac{1}{4}-\frac{i\tau}{2}; \frac{1}{2}; \left( \frac{1-y}{1+y} \right)^2 \right) \frac{y^{ix-\frac{1}{2}}}{1+y} \ud y, $$
	$$ I_2^-(x, \tau) = \int_0^{\infty} \HyG \left( \frac{3}{4}+\frac{i\tau}{2}, \frac{3}{4}-\frac{i\tau}{2}; \frac{3}{2}; \left( \frac{1-y}{1+y} \right)^2 \right) \frac{1-y}{1+y} \frac{y^{ix-\frac{1}{2}}}{1+y} \ud y; $$
	$$ I_1^+(x, \tau) = \int_0^1 \HyG \left( \frac{1}{4}-\frac{i\tau}{2}, \frac{3}{4}-\frac{i\tau}{2}; 1-i\tau; \left( \frac{1-y}{1+y} \right)^2 \right) (1-y)^{-\frac{1}{2}-i\tau} (1+y)^{-\frac{1}{2}+i\tau} y^{ix-\frac{1}{2}} \ud y, $$
	$$ I_2^+(x, \tau) = \int_1^{\infty} \HyG \left( \frac{1}{4}-\frac{i\tau}{2}, \frac{3}{4}-\frac{i\tau}{2}; 1-i\tau; \left( \frac{1-y}{1+y} \right)^2 \right) (y-1)^{-\frac{1}{2}-i\tau} (1+y)^{-\frac{1}{2}+i\tau} y^{ix-\frac{1}{2}} \ud y. $$
	Then we have the formulae
	$$ I_1^-(x, \tau) = \frac{\extnorm{\Gamma \left( \frac{1}{2}+ix \right)}^2 \extnorm{ \Gamma \left( \frac{3}{4} + \frac{i\tau}{2} \right) }^2}{\Gamma \left( \frac{1}{2} \right)} \Re \left( \GenHyG{3}{2}{\frac{1}{2}+ix, \frac{1}{2}+i\tau, \frac{1}{2}-i\tau}{1,1}{1} \right), $$
	$$ I_2^-(x, \tau) =-i \frac{\extnorm{\Gamma \left( \frac{1}{2}+ix \right)}^2 \extnorm{\Gamma \left( \frac{1}{4} + \frac{i\tau}{2} \right) }^2}{2\Gamma \left( \frac{1}{2} \right)} \Im \left( \GenHyG{3}{2}{\frac{1}{2}+ix, \frac{1}{2}+i\tau, \frac{1}{2}-i\tau}{1,1}{1} \right), $$
	$$ I_1^+(x, \tau) = 2^{-\frac{1}{2}+i\tau} \frac{\Gamma \left( \frac{1}{2}+ix \right) \Gamma \left( \frac{1}{2}-i\tau \right)}{\Gamma \left( 1+ix-i\tau \right)} \GenHyG{3}{2}{\frac{1}{2}-i\tau, \frac{1}{2}-i\tau, \frac{1}{2}-i\tau}{1+ix-i\tau, 1-2i\tau}{1}, $$
	$$ I_2^+(x, \tau) = 2^{-\frac{1}{2}+i\tau} \frac{\Gamma \left( \frac{1}{2}-ix \right) \Gamma \left( \frac{1}{2}-i\tau \right)}{\Gamma \left( 1-ix-i\tau \right)} \GenHyG{3}{2}{\frac{1}{2}-i\tau, \frac{1}{2}-i\tau, \frac{1}{2}-i\tau}{1-ix-i\tau, 1-2i\tau}{1}. $$
\end{lemma}
\begin{proof}
	For the first integral, the quadratic relation \cite[15.8.27]{OLBC10} implies
\begin{align*} 
	&\quad \HyG \left( \frac{1}{4}+\frac{i\tau}{2}, \frac{1}{4}-\frac{i\tau}{2}; \frac{1}{2}; \left( \frac{1-y}{1+y} \right)^2 \right) \\
	&= \frac{\extnorm{ \Gamma \left( \frac{3}{4} + \frac{i\tau}{2} \right) }^2}{2 \Gamma \left( \frac{1}{2} \right)} \cdot \left( \HyG \left( \frac{1}{2}+i\tau, \frac{1}{2}-i\tau; 1; \frac{1}{1+y} \right) + \HyG \left( \frac{1}{2}+i\tau, \frac{1}{2}-i\tau; 1; \frac{y}{1+y} \right) \right). 
\end{align*}
	Inserting the Taylor series expansion, we get
\begin{align}\label{I1-2F1int}
	&\quad \int_0^{\infty} \HyG \left( \frac{1}{2}+i\tau, \frac{1}{2}-i\tau; 1; \frac{1}{1+y} \right) \frac{y^{ix-\frac{1}{2}}}{1+y} \ud y \nonumber\\
	&= \sum_{n=0}^{\infty} \frac{\extnorm{\left( \frac{1}{2}+i\tau \right)_n}^2}{(n!)^2} \int_0^{\infty} \frac{y^{ix-\frac{1}{2}}}{(1+y)^{n+1}} \ud y = \sum_{n=0}^{\infty} \frac{\extnorm{\left( \frac{1}{2}+i\tau \right)_n}^2}{(n!)^2} \frac{\Gamma \left( \frac{1}{2}+ix \right) \Gamma \left( \frac{1}{2}-ix+n \right)}{\Gamma \left( n+1 \right)} \nonumber\\
	&= \extnorm{\Gamma \left( \frac{1}{2}+ix \right)}^2 \GenHyG{3}{2}{\frac{1}{2}+i\tau, \frac{1}{2}-i\tau, \frac{1}{2}-ix}{1,1}{1}.
\end{align}
	The following integral is obtained by the change of variables $y \mapsto 1/y, x \to -x$ in the previous one
	$$ \int_0^{\infty} \HyG \left( \frac{1}{2}+i\tau, \frac{1}{2}-i\tau; 1; \frac{y}{1+y} \right) \frac{y^{ix-\frac{1}{2}}}{1+y} \ud y = \extnorm{\Gamma \left( \frac{1}{2}+ix \right)}^2 \GenHyG{3}{2}{\frac{1}{2}+i\tau, \frac{1}{2}-i\tau, \frac{1}{2}+ix}{1,1}{1}. $$
	We thus obtain the first formula. To prove the second formula we apply \cite[15.8.28]{OLBC10} 
\begin{align*} 
	&\quad \frac{1-y}{1+y} \cdot \HyG \left( \frac{3}{4}+\frac{i\tau}{2}, \frac{3}{4}-\frac{i\tau}{2}; \frac{3}{2}; \left( \frac{1-y}{1+y} \right)^2 \right) \\
	&= \frac{\extnorm{ \Gamma \left( \frac{1}{4} + \frac{i\tau}{2} \right) }^2}{2 \Gamma \left( -\frac{1}{2} \right)} \cdot \left( \HyG \left( \frac{1}{2}+i\tau, \frac{1}{2}-i\tau; 1; \frac{y}{1+y} \right) - \HyG \left( \frac{1}{2}+i\tau, \frac{1}{2}-i\tau; 1; \frac{1}{1+y} \right) \right). 
\end{align*}
Therefore, using $\Gamma(-1/2)=-2\Gamma(1/2)$, we obtain 
\begin{align*} 
	&\quad I_2^-(x,\tau)=\frac{\extnorm{ \Gamma \left( \frac{1}{4} + \frac{i\tau}{2} \right) }^2}{4 \Gamma \left( \frac{1}{2} \right)}
	\int_0^{\infty} \HyG \left( \frac{1}{2}+i\tau, \frac{1}{2}-i\tau; 1; \frac{1}{1+y} \right) \frac{y^{ix-\frac{1}{2}}-y^{-ix-\frac{1}{2}}}{1+y} \ud y. 
\end{align*}
Then \eqref{I1-2F1int} yields the second formula.

	 The change of variables $y \mapsto 1/y$ easily identifies $I_2^+(x,\tau) = I_1^+(-x,\tau)$. Hence it suffices to treat $I_1^+(x,\tau)$. The quadratic relation \cite[15.8.13]{OLBC10} implies
	$$ \HyG \left( \frac{1}{4}-\frac{i\tau}{2}, \frac{3}{4}-\frac{i\tau}{2}; 1-i\tau; \left( \frac{1-y}{1+y} \right)^2 \right) = 2^{-\frac{1}{2}+i\tau} (1+y)^{\frac{1}{2}-i\tau} \HyG \left( \frac{1}{2}-i\tau, \frac{1}{2}-i\tau; 1-2i\tau; 1-y \right). $$
	Inserting the Taylor series expansion, we get
\begin{align*}
	\frac{I_1^+(x, \tau)}{2^{-\frac{1}{2}+i\tau}} &=  \int_0^1 \HyG \left( \frac{1}{2}-i\tau, \frac{1}{2}-i\tau; 1-2i\tau; 1-y \right) (1-y)^{-\frac{1}{2}-i\tau} y^{ix-\frac{1}{2}} \ud y \\
	&= \sum_{n=0}^{\infty} \frac{\left( \frac{1}{2}-i\tau \right)_n^2}{(1-2i\tau)_n n!} \int_0^1 (1-y)^{n-\frac{1}{2}-i\tau} y^{ix-\frac{1}{2}} \ud y = \sum_{n=0}^{\infty} \frac{\left( \frac{1}{2}-i\tau \right)_n^2}{(1-2i\tau)_n n!} \frac{\Gamma \left( \frac{1}{2}+ix \right) \Gamma \left( \frac{1}{2}-i\tau+n \right)}{\Gamma \left( 1+ix-i\tau+n \right)} \\
	&= \frac{\Gamma \left( \frac{1}{2}+ix \right) \Gamma \left( \frac{1}{2}-i\tau \right)}{\Gamma \left( 1+ix-i\tau \right)} \GenHyG{3}{2}{\frac{1}{2}-i\tau, \frac{1}{2}-i\tau, \frac{1}{2}-i\tau}{1+ix-i\tau, 1-2i\tau}{1},
\end{align*}
	and conclude the proof.
\end{proof}

\begin{proposition} \label{M3ToM4}
	Let $Q_{\pm}(t;\tau)$ be defined in Lemma \ref{VarMF} (3) \& (4). For simplicity of notation write $\pi(i\tau, \epsilon) := \pi(\sgn^{\frac{\epsilon-1}{2}}, i\tau)$. 
\begin{itemize}
	\item[(1)] The function $I(y,\Psi)$ is given as
	\begin{equation}\label{eq:I(y,Psi)}
	I(y,\Psi) = \norm[y-1]^{-1} \int_{-\infty}^{\infty} \left( \sum_{\pm} Q_{\pm}\left( \frac{y+1}{2(y-1)} ;\tau \right) h \left( \pi(i\tau,\pm) \right)(\Psi) \right) \tau \tanh(\pi \tau) \frac{\ud \tau}{2\pi}. 
	\end{equation}
	\item[(2)] Let $\chi(t) = \norm[t]^{ix} \sgn(t)^{\varepsilon'}$ for some $x \in \R$ and $\varepsilon' \in \{ 0,1 \}$. Define $K_{1,+}(x,\tau)$, $K_{1,-}(x,\tau)$, $K_{2,+}(x,\tau)$ and $K_{2,-}(x,\tau)$ by \eqref{K1+Alt}, \eqref{K1-Alt}, \eqref{K2+Alt} and \eqref{K2-Alt} below respectively. Then we have the formula
\begin{align*}
	\widetilde{h}(\chi)(\Psi) = & (-1)^{\varepsilon'} \int_{-\infty}^{\infty} \left( \sum_{\pm} K_{1,\pm}(x,\tau) h \left( \pi(i\tau,\pm) \right)(\Psi) \right) \tau \tanh(\pi \tau) \frac{\ud \tau}{2\pi} + \\
	&\int_{-\infty}^{\infty} \left( \sum_{\pm} K_{2,\pm}(x,\tau) h \left( \pi(i\tau,\pm) \right)(\Psi) \right) \tau \tanh(\pi \tau) \frac{\ud \tau}{2\pi}.
\end{align*}
\end{itemize}
\end{proposition}
\begin{proof}
	(1) Note that for $\pi=\pi(i\tau,\epsilon)$ the normalizing condition in Corollary \ref{M3Exp} imposes $K_{\pi^{\vee}} = \epsilon K_{\pi}$, if we choose $K_{\pi} = K_{i\tau}^{\epsilon}$ by \eqref{KirToKBessel}. We can thus rewrite Corollary \ref{M3Exp} as
\begin{multline} \label{eq: M3ExpVar}
	h(\pi(i\tau, \epsilon))(\Psi) = \int_{\R^{\times}} \ell_1(z) \norm[z]^{-2} \ud^{\times}z \cdot \int_{\R^{\times}} \ell_2(z) \norm[z]^2 \ud^{\times}z \cdot \int_{\R^{\times}} h_1(y) \ud^{\times}y \cdot \int_{\R^{\times}} h_2(y) \ud^{\times}y \cdot \\
	\epsilon F^{\epsilon}[\OFour(f_1)](\tau) F^{\epsilon}[\OFour(f_2)](\tau).
\end{multline}

\noindent Combining \eqref{eq:b} and Lemma \ref{DualWtFC2} we get
\begin{multline} \label{eq: bExpVar}
	b(\Psi,g)= \int_{\R^{\times}} \ell_1(z) \norm[z]^{-2} \ud^{\times}z \cdot \int_{\R^{\times}} \ell_2(z) \norm[z]^2 \ud^{\times}z \cdot \int_0^{\infty} h_1(y) \ud^{\times}y \cdot \int_0^{\infty} h_2(y) \ud^{\times}y \cdot \\
	\frac{1}{2\pi} \sum_{\epsilon = \pm 1} \int_{-\infty}^{\infty} \tau \tanh(\pi \tau) \lambda_g^{\epsilon}(\tau) \cdot \epsilon F^{\epsilon}[\OFour(f_1)](\tau) F^{\epsilon}[\OFour(f_2)](\tau) \ud \tau
\end{multline}
	
\noindent Comparing \eqref{eq: M3ExpVar} and \eqref{eq: bExpVar} and taking into account the formulae of $\lambda_g^{\pm}$ in Lemma \ref{VarMF} (3) (4), we get
\begin{multline} \label{bpsig2}
	b(\Psi,g)
	= \frac{1}{2\pi} \sum_{\pm} \int_{-\infty}^{\infty} \tau \tanh(\pi \tau) \lambda_g^{\pm}(\tau) \cdot h \left( \pi(i\tau,\pm) \right)(\Psi) \ud \tau \\
	= \int_{\R} g \left( t+\frac{1}{2} \right) \int_{-\infty}^{\infty} \left( \sum_{\pm} Q_{\pm}(t;\tau) h \left( \pi(i\tau,\pm) \right)(\Psi) \right) \tau \tanh(\pi \tau) \frac{\ud \tau}{2\pi} \ud t.
\end{multline}

\noindent On the other hand, according to Lemma \ref{Bilin1} we have
\begin{equation} \label{bpsig}
	b(\Psi,g) = \int_{\R} I \left( \frac{y+\frac{1}{2}}{y-\frac{1}{2}}, \Psi \right) \extnorm{y-\frac{1}{2}}^{-1} \cdot g\left( y+\frac{1}{2} \right) \ud y.
\end{equation}

\noindent Since $g \in \Sch(\R)$ is arbitrary, comparing \eqref{bpsig} and \eqref{bpsig2} we deduce
	$$ I \left( \frac{t+\frac{1}{2}}{t-\frac{1}{2}}, \Psi \right) \extnorm{t-\frac{1}{2}}^{-1} = \int_{-\infty}^{\infty} \left( \sum_{\pm} Q_{\pm}(t;\tau) h \left( \pi(i\tau,\pm) \right)(\Psi) \right) \tau \tanh(\pi \tau) \frac{\ud \tau}{2\pi}. $$
	The desired formula follows by a change of variables.
	
\noindent (2) By Lemma \ref{GeomToM4}, we can decompose
	\begin{equation} \label{eq:M4}
		\widetilde{h}(\chi)(\Psi) =  \chi(-1) \int_0^{\infty} I(-y,\Psi) y^{ix-\frac{1}{2}} \ud y + \int_0^{\infty} I(y,\Psi) y^{ix-\frac{1}{2}} \ud y. 
	\end{equation}
	Consider the first integral in \eqref{eq:M4}. It follows from \eqref{eq:I(y,Psi)} that
	\begin{multline} \label{eq:int of I(-y)}
	\int_0^{\infty} I(-y,\Psi) y^{ix-\frac{1}{2}} \ud y \\
	= \int_{-\infty}^{\infty} \sum_{\pm} \int_{0}^{\infty}Q_{\pm}\left( \frac{y-1}{2(1+y)},\tau\right) \frac{y^{ix-1/2}\ud y}{|1+y|} h \left( \pi(i\tau,\pm) \right)(\Psi) \tau \tanh(\pi \tau) \frac{\ud \tau}{2\pi}.
	\end{multline}
	For $y > 0$, we have $\extnorm{\frac{y-1}{2(y+1)}} < \frac{1}{2}$, hence by Lemma \ref{VarMF} 
	$$Q_{\pm } \left( \frac{y-1}{2(1+y)}; \tau \right) = Q_{\pm }^1 \left( \frac{y-1}{2(1+y)}; \tau \right).$$ 
	Inserting the first two formulae in Lemma \ref{TechInts}, we get
\begin{multline} \label{K1+Alt}
	\int_0^{\infty} Q_+^1 \left( \frac{y-1}{2(1+y)}; \tau \right) (y+1)^{-1} y^{ix-\frac{1}{2}} \ud y =  \frac{2\extnorm{\Gamma \left( \frac{1}{4}+\frac{i\tau}{2} \right)}^2}{\pi\Gamma \left( \frac{1}{2} \right)} I_1^-(x, \tau) \\
	= \frac{2}{\pi^2} \extnorm{\Gamma \left( \frac{1}{2}+ix \right)}^2 \extnorm{\Gamma \left( \frac{1}{4}+\frac{i\tau}{2} \right)}^2 \extnorm{ \Gamma \left( \frac{3}{4} + \frac{i\tau}{2} \right) }^2 \Re \left( \GenHyG{3}{2}{\frac{1}{2}+ix, \frac{1}{2}+i\tau, \frac{1}{2}-i\tau}{1,1}{1} \right) \\
	= \frac{4}{\pi} \extnorm{\Gamma \left( \frac{1}{2}+ix \right)}^2 \extnorm{\Gamma \left( \frac{1}{2}+i\tau \right)}^2 \Re \left( \GenHyG{3}{2}{\frac{1}{2}+ix, \frac{1}{2}+i\tau, \frac{1}{2}-i\tau}{1,1}{1} \right) \\
	= \frac{2}{\pi} \sum_{\varepsilon = \pm 1} \Gamma \left( \frac{1}{2}-i\varepsilon x \right) \GenHyGI{3}{2}{\frac{1}{2}+i\varepsilon x, \frac{1}{2}+i\tau, \frac{1}{2}-i\tau}{1,1}{1} =: K_{1,+}(x,\tau),
\end{multline}
\begin{multline} \label{K1-Alt}
	\int_0^{\infty} Q_-^1 \left( \frac{y-1}{2(1+y)}; \tau \right) (y+1)^{-1} y^{ix-\frac{1}{2}} \ud y =  \frac{-2\extnorm{\Gamma \left( \frac{3}{4}+\frac{i\tau}{2} \right)}^2}{\pi\Gamma \left( \frac{3}{2} \right)} I_2^-(x, \tau) \\
	= \frac{2i}{\pi^2} \extnorm{\Gamma \left( \frac{1}{2}+ix \right)}^2 \extnorm{\Gamma \left( \frac{1}{4}+\frac{i\tau}{2} \right)}^2 \extnorm{ \Gamma \left( \frac{3}{4} + \frac{i\tau}{2} \right) }^2 \Im \left( \GenHyG{3}{2}{\frac{1}{2}+ix, \frac{1}{2}+i\tau, \frac{1}{2}-i\tau}{1,1}{1} \right) \\
	= \frac{2}{\pi} \sum_{\varepsilon = \pm 1} \varepsilon \Gamma \left( \frac{1}{2}-i\varepsilon x \right) \GenHyGI{3}{2}{\frac{1}{2}+i\varepsilon x, \frac{1}{2}+i\tau, \frac{1}{2}-i\tau}{1,1}{1} =: K_{1,-}(x,\tau),
\end{multline}
	by Legendre duplication formula for Gamma functions
\begin{equation} \label{duplication}
	\Gamma(1/4+i\tau/2)\Gamma(3/4+i\tau/2)=\pi^{1/2}2^{1/2-i\tau}\Gamma(1/2+i\tau).
\end{equation}
\noindent Hence we get by \eqref{eq:int of I(-y)}
	\begin{equation} \label{eq:I(-y)toK1pm}
	 \int_0^{\infty} I(-y,\Psi) y^{ix-\frac{1}{2}} \ud y = \int_{-\infty}^{\infty} \left( \sum_{\pm} K_{1,\pm}(x,\tau) h \left( \pi(i\tau,\pm)(\Psi) \right) \right) \tau \tanh(\pi \tau) \frac{\ud \tau}{2\pi},
	 \end{equation} 
	which gives the first summand of the formula. Turn to the second integral in \eqref{eq:M4}. By \eqref{eq:I(y,Psi)} we have
	\begin{multline} \label{eq:int of I(y)}
	\int_0^{\infty} I(y,\Psi) y^{ix-\frac{1}{2}} \ud y \\
	= \int_{-\infty}^{\infty} \sum_{\pm} \int_{0}^{\infty}Q_{\pm}\left( \frac{y+1}{2(y-1)},\tau\right) \frac{y^{ix-1/2}\ud y}{|1-y|} h \left( \pi(i\tau,\pm) \right)(\Psi) \tau \tanh(\pi \tau) \frac{\ud \tau}{2\pi}.
	\end{multline}

\noindent Similarly for $y > 0$ , we have $\extnorm{\frac{y+1}{2(y-1)}} > \frac{1}{2}$, hence by Lemma \ref{VarMF} 
	$$Q_{\pm } \left( \frac{y+1}{2(y-1)}; \tau \right) = Q_{\pm }^2 \left( \frac{y+1}{2(y-1)}; \tau \right) + Q_{\pm }^2 \left( \frac{y+1}{2(y-1)}; -\tau \right).$$ Inserting the definitions of $Q_+^2$ and $I_{k}^+(x,\tau),k=1,2$ (see Lemma \ref{VarMF} and \ref{TechInts}), we get
	$$ \int_0^{\infty} Q_+^2 \left( \frac{y+1}{2(y-1)}; \tau \right) \frac{y^{ix-\frac{1}{2}}}{\norm[y-1]} \ud y =  \frac{1 + i\sinh(\pi \tau)}{\pi i\sinh(\pi \tau)} \frac{\Gamma \left( \frac{1}{4} - \frac{i\tau}{2} \right) \Gamma \left( \frac{3}{4} - \frac{i\tau}{2} \right)}{\Gamma \left( 1 - i\tau \right)} \left( I_1^+(x,\tau) + I_2^+(x,\tau) \right). $$
Using Lemma \ref{TechInts}, the duplication formula \eqref{duplication} and the identity $\frac{1}{\Gamma(1-i\tau)}=\frac{\Gamma(1/2-i\tau)}{\pi^{1/2}2^{2i\tau}\Gamma(1-2i\tau)}$, we infer
	$$ \int_0^{\infty} Q_+^2 \left( \frac{y+1}{2(y-1)}; \tau \right) \frac{y^{ix-\frac{1}{2}}}{\norm[y-1]} \ud y = \sum_{\varepsilon_1 \in \{ \pm 1 \}} \Gamma \left( \frac{1}{2}+i\varepsilon_1 x \right) \frac{i\sinh(\pi\tau)+1}{\pi i\sinh(\pi\tau)} \GenHyGI{3}{2}{\frac{1}{2}-i \tau, \frac{1}{2}-i \tau, \frac{1}{2}-i \tau}{1+i\varepsilon_1 x -i \tau, 1-2i \tau}{1}. $$
Therefore, we get
\begin{multline} \label{K2+Alt}
	 \int_0^{\infty} \left( Q_+^2 \left( \frac{y+1}{2(y-1)}; \tau \right) + Q_+^2 \left( \frac{y+1}{2(y-1)}; -\tau \right) \right) \frac{y^{ix-\frac{1}{2}}}{\norm[y-1]} \ud y \\
	=  \frac{1}{\pi} \sum_{\varepsilon_1, \varepsilon_2 \in \{ \pm 1 \}} \Gamma \left( \frac{1}{2}+i\varepsilon_1 x \right) \frac{i\sinh(\pi \varepsilon_2 \tau)+1}{i\sinh(\pi \varepsilon_2 \tau)} \GenHyGI{3}{2}{\frac{1}{2}-i\varepsilon_2 \tau, \frac{1}{2}-i\varepsilon_2 \tau, \frac{1}{2}-i\varepsilon_2 \tau}{1+i\varepsilon_1 x - i\varepsilon_2 \tau, 1-2i\varepsilon_2 \tau}{1} =: K_{2,+}(x,\tau).
\end{multline}
Similarly, we obtain
\begin{multline} \label{K2-Alt}
	\int_0^{\infty} \left( Q_-^2 \left( \frac{y+1}{2(y-1)}; \tau \right) + Q_-^2 \left( \frac{y+1}{2(y-1)}; -\tau \right) \right) \frac{y^{ix-\frac{1}{2}}}{\norm[y-1]} \ud y \\
	= \sum_{\varepsilon_1, \varepsilon_2 \in \{ \pm 1 \}} \varepsilon_1 \Gamma \left( \frac{1}{2}+i\varepsilon_1 x \right) \frac{1-i\sinh(\pi \varepsilon_2 \tau)}{\pi i\sinh(\pi \varepsilon_2 \tau)} \GenHyGI{3}{2}{\frac{1}{2}-i\varepsilon_2 \tau, \frac{1}{2}-i\varepsilon_2 \tau, \frac{1}{2}-i\varepsilon_2 \tau}{1+i\varepsilon_1 x - i\varepsilon_2 \tau, 1-2i\varepsilon_2 \tau}{1} =: K_{2,-}(x,\tau).
\end{multline}
Hence we get by \eqref{eq:int of I(y)}
	\begin{equation}\label{eq:I(y)toK2pm}
	 \int_0^{\infty} I(y,\Psi) y^{ix-\frac{1}{2}} \ud y = \int_{-\infty}^{\infty} \left( \sum_{\pm} K_{2,\pm}(x,\tau) h \left( \pi(i\tau,\pm) \right)(\Psi) \right) \tau \tanh(\pi \tau) \frac{\ud \tau}{2\pi}. 
	 \end{equation}
	 Substituting \eqref{eq:I(-y)toK1pm} and \eqref{eq:I(y)toK2pm} to \eqref{eq:M4}, we complete the proof.
\end{proof}

\begin{remark} \label{K1Alt}
	Applying \cite[(7.4.4.3)]{PBM86}, we can rewrite $K_{1,\pm}(x,\tau)$ as
	$$ K_{1,+}(x,\tau) = \frac{i}{\pi} \sum_{\varepsilon_1, \varepsilon_2 \in \{ \pm 1 \}} \varepsilon_2 \frac{\Gamma \left( \frac{1}{2}-i\varepsilon_1 x \right)}{\cosh(\pi x)} \frac{\cosh(\pi \tau)}{\sinh(\pi \tau)} \GenHyGI{3}{2}{\frac{1}{2}+i\varepsilon_2 \tau, \frac{1}{2}+i\varepsilon_2 \tau, \frac{1}{2}+i\varepsilon_2 \tau}{1+i\varepsilon_2 \tau - i\varepsilon_1 x, 1+2i\varepsilon_2 \tau}{1}, $$
	$$ K_{1,-}(x,\tau) = \frac{i}{\pi} \sum_{\varepsilon_1, \varepsilon_2 \in \{ \pm 1 \}} \varepsilon_1 \varepsilon_2 \frac{\Gamma \left( \frac{1}{2}-i\varepsilon_1 x \right)}{\cosh(\pi x)} \frac{\cosh(\pi \tau)}{\sinh(\pi \tau)} \GenHyGI{3}{2}{\frac{1}{2}+i\varepsilon_2 \tau, \frac{1}{2}+i\varepsilon_2 \tau, \frac{1}{2}+i\varepsilon_2 \tau}{1+i\varepsilon_2 \tau - i\varepsilon_1 x, 1+2i\varepsilon_2 \tau}{1}. $$
\end{remark}

\begin{remark} \label{K2Alt}
	If we introduce the function
	$$ F_2(\tau,t) :=  t^{\frac{1}{2}+i\tau} \GenHyGI{2}{1}{\frac{1}{2}+i\tau, \frac{1}{2}+i\tau}{1+2i\tau}{t}, $$
	then it is easy to get the following integral formulae by (\ref{K2+Alt}) \& (\ref{K2-Alt}) (or \cite[16.5.2]{OLBC10})
		$$ K_{2,+}(x,\tau) = \frac{4}{\pi} \int_0^1 \left\{ \Re (F_2(\tau,t)) - \frac{1}{\sinh(\pi \tau)} \Im (F_2(\tau,t)) \right\} \cdot \Re ((1-t)^{ix}) \frac{\ud t}{t\sqrt{1-t}}, $$
	$$ K_{2,-}(x,\tau) = \frac{-4i}{\pi} \int_0^1 \left\{ \Re (F_2(\tau,t)) + \frac{1}{\sinh(\pi \tau)} \Im (F_2(\tau,t)) \right\} \cdot \Im ((1-t)^{ix}) \frac{\ud t}{t\sqrt{1-t}}. $$
	If we introduce
	$$ \Lambda_2(\tau,t) := t^{-\frac{1}{2}-i\tau} \GenHyGI{2}{1}{\frac{1}{2}+i\tau, \frac{1}{2}+i\tau}{1+2i\tau}{-t^{-1}}, $$
	then we have the relation
	$$ F_2(\tau,t) = \Lambda_2(\tau,t^{-1}-1), \quad 0 < t < 1, $$
	and we can rewrite
	$$ K_{2,+}(x,\tau) = \frac{4}{\pi} \int_0^{\infty} \left\{ \Re (\Lambda_2(\tau,t)) - \frac{1}{\sinh(\pi \tau)} \Im (\Lambda_2(\tau,t)) \right\} \cdot \Re ((1+t^{-1})^{ix}) \frac{\ud t}{\sqrt{t(t+1)}}, $$
	whose kernel is identified with the Motohashi's in \cite[(4.7.2)]{Mo97}.
\end{remark}

	\subsection{Dual Weight Estimation}
	
		\subsubsection{Qualitative Decay}
	
	We are ready to give a proof of Theorem \ref{ExpInvMF}.
	
\begin{proof}[Proof of Theorem \ref{ExpInvMF}]
	We introduce the following transform
\begin{equation} \label{BasicTrans}
	\widetilde{h}(ix) := \int_{-\infty}^{\infty} K(x,\tau) h(i\tau) \frac{\ud \tau}{2\pi}, 
\end{equation}
where the kernel function $K(x,\tau)$ is
\begin{equation} \label{BasicKer}
	K(x,\tau) := \Gamma \left( \frac{1}{2}-i x \right) \GenHyGI{3}{2}{\frac{1}{2}+i\tau, \frac{1}{2}+i\tau, \frac{1}{2}+i\tau}{1-i x + i \tau, 1+2i \tau}{1}. 
\end{equation}
Precisely, we have the following equalities implied by Remark \ref{K1Alt}
	$$ K_{1,+}(x,\tau) = \frac{i}{\pi} \frac{1}{\cosh(\pi x)} \frac{\cosh(\pi \tau)}{\sinh(\pi \tau)} \sum_{\varepsilon_1, \varepsilon_2 \in \{ \pm \}} \varepsilon_2 K(\varepsilon_1 x, \varepsilon_2 \tau), $$
	$$ K_{1,-}(x,\tau) = \frac{i}{\pi} \frac{1}{\cosh(\pi x)} \frac{\cosh(\pi \tau)}{\sinh(\pi \tau)} \sum_{\varepsilon_1, \varepsilon_2 \in \{ \pm \}} \varepsilon_1 \varepsilon_2 K(\varepsilon_1 x, \varepsilon_2 \tau), $$
	$$ K_{2,+}(x,\tau) = \frac{1}{\pi} \sum_{\varepsilon_1, \varepsilon_2 \in \{ \pm \}} K(\varepsilon_1x, \varepsilon_2 \tau) - \frac{1}{\pi i \sinh(\pi \tau)} \sum_{\varepsilon_1, \varepsilon_2 \in \{ \pm \}} \varepsilon_2 K(\varepsilon_1x, \varepsilon_2 \tau), $$
	$$ K_{2,-}(x,\tau) = \frac{1}{\pi} \sum_{\varepsilon_1, \varepsilon_2 \in \{ \pm \}} \varepsilon_1 K(\varepsilon_1x, \varepsilon_2 \tau) + \frac{1}{
	\pi i \sinh(\pi \tau)} \sum_{\varepsilon_1, \varepsilon_2 \in \{ \pm \}} \varepsilon_1 \varepsilon_2 K(\varepsilon_1x, \varepsilon_2 \tau). $$
	Theorem \ref{ExpInvMF} readily follows from Proposition \ref{M3ToM4} (2) and the above equalities.
\end{proof}

\begin{lemma} \label{AuxTrans}
	Suppose $h(s)$ is a holomorphic function in $\Re s > -\delta$ for some $\delta > 0$, which has rapid decay in vertical region $0 \leq \Re s \leq c$ for any $c>0$. Define for $0 < \Re s < 1/2$ a function
	$$ h^*(s) := \Gamma(2s) \int_{-\infty}^{\infty} \frac{\Gamma \left( \frac{1}{2}+i\tau-s \right)}{\Gamma \left( \frac{1}{2}+i\tau+s \right)} h(i\tau) \frac{\ud \tau}{2\pi}. $$
	Then $h^*(s)$ has analytic continuation to $\Re s > 0$ with rapid decay in any vertical region $0 < a \leq \Re s < b$.
\end{lemma}
\begin{proof}
	For any $c > 0$ we have by contour shifting
	$$ h^*(s) = \Gamma(2s) \int_{(0)} \frac{\Gamma \left( \frac{1}{2}+s_1-s \right)}{\Gamma \left( \frac{1}{2}+s_1+s \right)} h(s_1) \frac{\ud s_1}{2\pi i} = \Gamma(2s) \int_{(c)} \frac{\Gamma \left( \frac{1}{2}+s_1-s \right)}{\Gamma \left( \frac{1}{2}+s_1+s \right)} h(s_1) \frac{\ud s_1}{2\pi i}. $$
	The right most expression is a well defined holomorphic function in $0 < \Re s < 1/2+c$. To see the rapid decay in $a \leq \Re s \leq b$, we take $c > b$ large enough. Write $s_1=c+i\tau_1, s=\sigma+i\tau$ with $a \leq \sigma \leq b$. We treat the case $\tau > 0$ in detail, leaving the case $\tau < 0$ as an exercise. Stirling's bound \cite[(B.8)]{Iw02} implies the existence of constants $A_j > 0$ depending only on $a,b,c$ such that
	$$ \frac{\Gamma \left( \frac{1}{2}+c-\sigma+i(\tau_1-\tau) \right)}{\Gamma \left( \frac{1}{2}+c+\sigma+i(\tau_1+\tau) \right)} \ll_{a,b,c} \left\{ \begin{matrix} e^{-\pi \tau} (1+\norm[\tau_1])^{A_1} & \text{if } \tau_1 \leq -\tau \\ e^{\pi \tau_1} (1+\norm[\tau])^{A_2} & \text{if } -\tau < \tau_1 < \tau \\ e^{\pi \tau} (1+\norm[\tau_1])^{A_3} & \text{if } \tau_1 \geq \tau \end{matrix} \right. . $$
	Invoking the rapid decay of $h(s_1)$, we easily see for any $C > 0$
	$$ \int_{(c)} \frac{\Gamma \left( \frac{1}{2}+s_1-s \right)}{\Gamma \left( \frac{1}{2}+s_1+s \right)} h(s_1) \frac{\ud s_1}{2\pi i} \ll_{a,b,c,C} e^{\pi \tau} (1+\norm[\tau])^{-C}. $$
	For example, we have
	$$ \int_{-\tau}^{\tau} e^{\pi \tau_1}(1+\norm[\tau_1])^{-C} \ud \tau_1 \ll \int_{-\tau}^{\tau/2} e^{\pi \tau_1} \ud \tau_1 + \int_{\tau/2}^{\tau} e^{\pi \tau_1} \left(1+\extnorm{\frac{\tau}{2}} \right)^{-C} \ud \tau_1 \ll_C e^{\pi \tau} (1+\norm[\tau])^{-C}. $$
	The desired rapid decay of $h^*(s)$ follows readily by applying Stirling's bound to $\Gamma(2s)$.
\end{proof}

\begin{proposition} \label{BasisTransProp}
	The transform $h \to \widetilde{h}$ given by (\ref{BasicTrans}) has the following properties.
\begin{itemize}
	\item[(1)] The kernel function $K(x,\tau)$ satisfies the following uniform bound
	$$ K(x,\tau) \ll (1+\norm[\tau])^{-1}. $$
	\item[(2)] The kernel function $K(x,\tau)$ has an alternative expression, valid for any $0 < c < 1/2$, as
	$$ K(x,\tau) = \int_{(c)} \frac{\Gamma \left( \frac{1}{2}+i\tau-s \right) \Gamma(s)^2}{\Gamma \left( \frac{1}{2}+i\tau+s \right)} \cdot \frac{\Gamma(s) \Gamma \left( \frac{1}{2}-ix-s \right)}{\Gamma \left( \frac{1}{2}-ix \right)} \frac{\ud s}{2\pi i}. $$
	\item[(3)] Suppose $h(s)$ is a holomorphic function in $\Re s > -\delta$ for some $\delta > 0$, which has rapid decay in any vertical region $0 \leq \Re s \leq a$. Then $\widetilde{h}(ix)$ has rapid decay as $\norm[x] \to \infty$.
\end{itemize}
\end{proposition}
\begin{proof}	
	(1) The proof is highly technical. We postpone it to \S \ref{TechKerBd}.
	
\noindent (2) By Remark \ref{K2Alt} (or \cite[16.5.2]{OLBC10}), we see
	$$ K(x,\tau) = 
	\int_0^1 F_2(\tau,t) (1-t)^{-ix} \frac{\ud t}{t \sqrt{1-t}} =
	\int_0^{\infty} \Lambda_2(\tau,t) \frac{t^{\frac{1}{2}-ix}}{(1+t)^{\frac{1}{2}-ix}} \frac{\ud t}{t}. $$
	For $0 < \Re s < 1/2$, we have the Mellin transforms
	$$ \int_0^{\infty} \frac{t^{\frac{1}{2}-ix-s}}{(1+t)^{\frac{1}{2}-ix}} \frac{\ud t}{t} = \frac{\Gamma(s) \Gamma \left( \frac{1}{2}-ix-s \right)}{\Gamma \left( \frac{1}{2}-ix \right)}, $$
	$$ \int_0^{\infty} \Lambda_2(\tau,t) t^s \frac{\ud t}{t} = \int_0^{\infty} \GenHyGI{2}{1}{\frac{1}{2}+i\tau, \frac{1}{2}+i\tau}{1+2i\tau}{-t} t^{\frac{1}{2}+i\tau-s} \frac{\ud t}{t} = \frac{\Gamma \left( \frac{1}{2}+i\tau-s \right) \Gamma(s)^2}{\Gamma \left( \frac{1}{2}+i\tau+s \right)}. $$
	The desired formula follows by Mellin inversion.

\noindent (3) Using $h^*(s)$ defined in Lemma \ref{AuxTrans}, we have by contour shifting
\begin{align}
	\widetilde{h}(ix) &= \int_{(c)} h^*(s) \cdot \frac{\Gamma(s)^3 \Gamma \left( \frac{1}{2}-ix-s \right)}{\Gamma(2s) \Gamma \left( \frac{1}{2}-ix \right)} \frac{\ud s}{2\pi i} \nonumber \\
	&= \sum_{k=0}^{n-1} h^*\left( \frac{1}{2}-ix+k \right) \frac{(-1)^k \Gamma \left( \frac{1}{2}-ix+k \right)^3}{k! \Gamma \left( 1-2ix+2k \right) \Gamma \left( \frac{1}{2}-ix \right)} + \int_{(c+n)} h^*(s) \frac{\Gamma(s)^3 \Gamma \left( \frac{1}{2}-ix-s \right)}{\Gamma(2s) \Gamma \left( \frac{1}{2}-ix \right)} \frac{\ud s}{2\pi i} \label{ContShift}
\end{align}
	for any integer $n \geq 1$. Each summand indexed by $k$ above has rapid decay by Lemma \ref{AuxTrans} and Striling's bound. Arguing similarly as in the proof of Lemma \ref{AuxTrans} we see
	$$ \int_{(c+n)} h^*(s) \frac{\Gamma(s)^3 \Gamma \left( \frac{1}{2}-ix-s \right)}{\Gamma(2s) \Gamma \left( \frac{1}{2}-ix \right)} \frac{\ud s}{2\pi i} \ll (1+\norm[x])^{-c-n}. $$
	Since $n$ can be taken arbitrarily large, we conclude the desired rapid decay.
\end{proof}

		\subsubsection{Quantitative Decay}
		\label{DGAuxA}

	For simplicity of notation, we shall write $h(i\tau,\pm)$ resp. $\chi(-1) \widetilde{h}(ix,1) + \widetilde{h}(ix,2)$ instead of $h(\pi(i\tau,\pm))(\Psi)$ resp. $\widetilde{h}(\chi)(\Psi)$. By Proposition \ref{prop: AdmWtR}, for any large parameters $T \gg 1$ and $\Delta = T^{\epsilon}$ the following function is admissible and non-negative for $\tau \in \R \cup i (-1/2,1/2)$
\begin{equation} \label{ConcWtR}
	h(i\tau,\pm) = \frac{\cosh(\pi \tau)}{2\sqrt{\pi} \Delta} \left\{ \exp \left( - \frac{(\tau-T)^2}{2\Delta^2} - \frac{\pi}{2} \tau \right) + \exp \left( - \frac{(\tau+T)^2}{2\Delta^2} + \frac{\pi}{2} \tau \right) \right\}^2. 
\end{equation}
\begin{remark}
	Define $g_0(x) := G(x; T, \Delta) - e^{-\frac{\pi}{2}(2T-\pi \Delta^2)} G(x; T-\pi \Delta^2, \Delta)$ with
	$$ G(x; T, \Delta) := \tfrac{2i \Delta}{\sqrt{2\pi}} e^{- \frac{\Delta^2}{2} \left( \log (x+\sqrt{x^2+1}) \right)^2} \cdot \sin \left( T \log (x+\sqrt{x^2+1}) \right). $$
	One verifies readily $g_0^{(n)}(0)=0$ for all integer $n \geq 0$ via the following elementary calculation
	$$ \int_0^{\infty} G \left( \tfrac{1}{2}(y-y^{-1}); T,\Delta \right) y^{2n} \frac{\ud y}{y} = e^{\frac{4n^2-T^2}{2\Delta^2}} \sin \left( \tfrac{2nT}{\Delta^2} \right), \quad \forall \ n \in \Z. $$
	Let $\varepsilon_0 \in \{ \pm 1 \}$ and $f_j \in \Sch(\R)$ be determined by
\begin{equation*}
	\begin{cases} f_j(-t) = \varepsilon_0 f_j(t) & \forall \ t \in \R \\ 2\varepsilon_0 \OFour(f_1) \left( \tfrac{t}{2\pi} \right) = 2 \OFour(f_2) \left( \tfrac{t}{2\pi} \right) = \tfrac{i}{\pi} t^{\frac{1}{2}} \int_0^{\infty} g_0(x) \sin(xt) \ud x & \forall \ t > 0 \end{cases}. 
\end{equation*}
	For $\Psi = \phi_1^{\vee} * \phi_2$ with $\phi_j$ defined via $f_j$ by \eqref{TestPhiR} one checks easily that $h(i\tau, \varepsilon_0)(\Psi)$ is equal to \eqref{ConcWtR} and $h(i\tau, -\varepsilon_0)(\Psi)=0$ without appealing to the Paley-Wiener result Proposition \ref{prop: AdmWtR}.
\end{remark}
	
\noindent By Theorem \ref{ExpInvMF}, we need to consider $h(i\tau)$ of the form
\begin{equation} \label{eq: ConcWtRBis}
	h_j(i\tau) = \frac{w_j(\tau)}{2\sqrt{\pi} \Delta} \left\{ \exp \left( - \frac{(\tau-T)^2}{2\Delta^2} - \frac{\pi}{2} \tau \right) + \exp \left( - \frac{(\tau+T)^2}{2\Delta^2} + \frac{\pi}{2} \tau \right) \right\}^2,
\end{equation}
	where the functions $w_j$ are given by
	$$ w_1(\tau) = \tau \sinh(\pi \tau), \quad w_2(\tau) = \tau \cosh(\pi \tau), \quad w_3(\tau) = \tau, $$
	and to bound the corresponding dual weights $\widetilde{h}_j(ix)$. It would be convenient to introduce
	$$ g(\tau; T, \Delta) := \frac{1}{2\sqrt{\pi} \Delta} \left\{ \exp \left( - \frac{(\tau-T)^2}{2\Delta^2} - \frac{\pi}{2} \tau \right) + \exp \left( - \frac{(\tau+T)^2}{2\Delta^2} + \frac{\pi}{2} \tau \right) \right\}^2, $$
	whose Fourier transform is given by
\begin{align*} 
	\hat{g}(x; T, \Delta) &:= \int_{-\infty}^{\infty} g(\tau; T, \Delta) e^{-ix\tau} \ud \tau \\
	&= \frac{e^{-\pi T}}{2} \left\{ \exp \left( -\frac{1}{4}(\Delta (x-\pi i))^2 \right) e^{-iTx} + \exp \left( -\frac{1}{4}(\Delta (x+\pi i))^2 \right) e^{iTx} \right\} \\
	&\quad + e^{-\frac{T^2}{\Delta^2}} \exp \left( - \frac{1}{4} (\Delta x)^2 \right). 
\end{align*}
	We easily compute its derivative with respect to $x$
\begin{align*} 
	\hat{g}'(x; T,\Delta) &= - e^{-\pi T} \exp \left( -\frac{1}{4}(\Delta (x-\pi i))^2 \right) e^{-iTx} \cdot \left\{ \frac{1}{4} \Delta^2 (x-\pi i) + \frac{i}{2} T \right\} \\
	&\quad - e^{-\pi T} \exp \left( -\frac{1}{4}(\Delta (x+\pi i))^2 \right) e^{iTx} \cdot \left\{ \frac{1}{4} \Delta^2 (x + \pi i) - \frac{i}{2} T \right\} \\
	&\quad - e^{-\frac{T^2}{\Delta^2}} \exp \left( - \frac{1}{4} (\Delta x)^2 \right) \cdot \frac{1}{2} \Delta^2 x.
\end{align*}
	
\begin{lemma} \label{TechEst}
	There is an absolute constant $\delta > 0$ such that for $\norm[\theta] \leq \delta$ and for all $\lambda \geq 0$ we have
	$$ f(\lambda, \theta) := \frac{1}{4} \left( \log (1+2\lambda \cos \theta + \lambda^2) \right)^2 - \frac{1}{2} (\log(1+\lambda))^2 - \left( \min (\norm[\theta], \lambda \sin \norm[\theta]) \right)^2 \geq 0. $$
\end{lemma}
\begin{proof}
	This is \cite[(5.1.17)]{Mo97}. We leave the elementary details to the reader.
\end{proof}

\begin{lemma} \label{RefAuxTrans}
	Let $h_j^*(s)$ be the transform introduced in Lemma \ref{AuxTrans} of $h_j$. Write $\hat{g}(x)$ resp. $\hat{g}'(x)$ for $\hat{g}(x; T, \Delta)$ resp. $\hat{g}'(x; T, \Delta)$ for simplicity of notation.
\begin{itemize}
	\item[(1)] We have the formulae valid for $\Re s > 0$
	$$ h_1^*(s) = \frac{i}{4\pi} \int_0^{\infty} \frac{y^{2s-1}}{(1+y)^{s+\frac{1}{2}}} \left\{ \hat{g}'(\log(1+y)+\pi i) - \hat{g}'(\log(1+y)-\pi i) \right\} \ud y, $$
	$$ h_2^*(s) = \frac{i}{4\pi} \int_0^{\infty} \frac{y^{2s-1}}{(1+y)^{s+\frac{1}{2}}} \left\{ \hat{g}'(\log(1+y)+\pi i) + \hat{g}'(\log(1+y)-\pi i) \right\} \ud y, $$
	$$ h_3^*(s) = \frac{i}{2\pi} \int_0^{\infty} \frac{y^{2s-1}}{(1+y)^{s+\frac{1}{2}}} \hat{g}'(\log(1+y)) \ud y. $$	
	\item[(2)] Suppose $\Re s$ is positive and bounded. We have the following bounds
	$$ h_1^*(s), h_2^*(s) \ll T \Delta^{-2\Re s} e^{-\frac{\norm[\Im s]}{T}}, \quad \extnorm{h_3^*(s)} \ll_A 
	T^{-A} e^{-\frac{\norm[\Im s]}{T}}, \forall A > 1. $$
\end{itemize}
\end{lemma}
\begin{proof}
	(1) Inserting an integral representation of the beta function, we get
	$$ h^*(s) = \int_0^{\infty} \frac{y^{2s-1}}{(1+y)^{s+\frac{1}{2}}} \int_{-\infty}^{\infty} \frac{h(i\tau)}{(1+y)^{i\tau}} \frac{\ud \tau}{2\pi} \ud y $$
	valid for $0 < \Re s < 1/2$. Inserting the expressions of $h_j(i\tau)$ and applying some elementary relations of Fourier transform, we get the desired formulae for $0 < \Re s < 1/2$. Their validity for $\Re s > 0$ is ensured by the rapid decay of $\hat{g}'(\log(1+y); T,\Delta)$ and $\hat{g}'(\log(1+y) \pm \pi i; T,\Delta)$ in terms of $y \to \infty$.
	
\noindent (2) We shall turn the line of integration in the formulae obtain in (1) through a small angle $\theta = \sgn(\Im s) T^{-1}$, then bound the integrands in absolute value. Putting $y = \lambda e^{i\theta}$ with $\lambda > 0$, we have
	$$ 1+y = (1+2\lambda \cos\theta + \lambda^2)^{\frac{1}{2}} e^{i \vartheta} $$
	where the angle $\vartheta$ satisfies
	$$ \norm[\vartheta] = \arctan \left( \frac{\lambda \sin \norm[\theta]}{1+\lambda \cos \theta} \right) \leq \min \left\{ \norm[\theta], \lambda \sin \norm[\theta] \right\}. $$
	Then provided $\norm[\theta]$ is small, we easily deduce the following bounds (from Lemma \ref{TechEst})
	$$ \extnorm{y^{2s}} = \lambda^{2 \Re s} e^{- 2 \frac{\norm[\Im s]}{T}}, \quad \extnorm{(1+y)^{-s-\frac{1}{2}}} \leq \left( \frac{1}{2}+\lambda \right)^{-\Re s - \frac{1}{2}} e^{\frac{\norm[\Im s]}{T}}, $$
	$$ \Re \left( \log(1+y) \right)^2 = \frac{1}{4} \left( \log(1+2\lambda \cos\theta + \lambda^2) \right)^2 - \vartheta^2 \geq \frac{1}{2} (\log(1+\lambda))^2, $$
	$$ \Re \left( \log(1+y) \pm \pi i \right)^2 = \frac{1}{4} \left( \log(1+2\lambda \cos\theta + \lambda^2) \right)^2 - \vartheta^2 - \pi^2 \mp 2\pi \vartheta \geq \frac{1}{2} (\log(1+\lambda))^2 - \pi^2 - \frac{2\pi}{T}, $$
	$$ \Re \left( \log(1+y) \pm 2\pi i \right)^2 = \frac{1}{4} \left( \log(1+2\lambda \cos\theta + \lambda^2) \right)^2 - \vartheta^2 - 4\pi^2 \mp 4\pi \vartheta \geq \frac{1}{2} (\log(1+\lambda))^2 - 4\pi^2 - \frac{4\pi}{T}, $$
	$$ \extnorm{\log (1+y)} \leq \log(1+\lambda) + \frac{1}{T}. $$
	It follows that
\begin{multline*}
	h_1^*(s), h_2^*(s) \ll e^{-\frac{\norm[\Im s]}{T}} \left\{ T + \Delta^2 e^{-\frac{T^2}{\Delta^2}} \exp \left( \frac{\pi^2+1}{4} \Delta^2 \right) + T e^{-2\pi T} \exp((\pi^2+1)\Delta^2) \right\} \cdot \\
	\int_0^{\infty} \frac{\lambda^{2\Re s}}{(\frac{1}{2}+\lambda)^{\Re s + \frac{1}{2}}} \exp \left( - \frac{\Delta^2}{8} (\log(1+\lambda))^2 \right) \left( \log(1+\lambda) + 1 \right) \frac{d\lambda}{\lambda} \\
	= \Delta^{\frac{1}{2}-\Re s} e^{-\frac{\norm[\Im s]}{T}} \left\{ T + \Delta^2 e^{-\frac{T^2}{\Delta^2}} \exp \left( \frac{\pi^2+1}{4} \Delta^2 \right) + T e^{-2\pi T} \exp((\pi^2+1)\Delta^2) \right\} \cdot \\
	\int_0^{\infty} \frac{\lambda^{2\Re s}}{(\frac{\Delta}{2}+\lambda)^{\Re s + \frac{1}{2}}} \exp \left( - \frac{\Delta^2}{8} (\log(1+\frac{\lambda}{\Delta}))^2 \right) \left( \log(1+\frac{\lambda}{\Delta}) + 1 \right) \frac{d\lambda}{\lambda} \\
	\ll \Delta^{-2\Re s} e^{-\frac{\norm[\Im s]}{T}} \left\{ T + \Delta^2 e^{-\frac{T^2}{\Delta^2}} \exp \left( \frac{\pi^2+1}{4} \Delta^2 \right) + T e^{-2\pi T} \exp((\pi^2+1)\Delta^2) \right\},
\end{multline*}
\begin{multline*}
	h_3^*(s) \ll e^{-\frac{\norm[\Im s]}{T}} \left\{ \Delta^2 e^{-\frac{T^2}{\Delta^2}} + T \exp \left( \frac{\pi^2+1}{4} \Delta^2 \right) e^{-\pi T} \right\} \cdot \\
	\int_0^{\infty} \frac{\lambda^{2\Re s}}{(\frac{1}{2}+\lambda)^{\Re s + \frac{1}{2}}} \exp \left( - \frac{\Delta^2}{8} (\log(1+\lambda))^2 \right) \left( \log(1+\lambda) + 1 \right) \frac{d\lambda}{\lambda} \\
	\ll \Delta^{-2\Re s} e^{-\frac{\norm[\Im s]}{T}} \left\{ \Delta^2 e^{-\frac{T^2}{\Delta^2}} + T e^{-\pi T} \right\},
\end{multline*}
	where the implied constants depend only on $\Re s$, and we have used the inequalities for $\Delta \geq 1$
    $$ \Delta \log \left( 1+\frac{\lambda}{\Delta} \right) \geq \log(1+\lambda) \geq 0, \quad \left( \frac{\Delta}{2}+\lambda \right)^{\Re s + \frac{1}{2}} \geq \left( \frac{\Delta}{2} \right)^{\Re s + \frac{1}{2}}, \quad \log \left( 1+\frac{\lambda}{\Delta} \right) \leq \log(1+\lambda). $$
\end{proof}

\begin{proposition} \label{DualWtRD}
	Let $\widetilde{h}_j(ix)$ be the transforms of $h_j(s)$ given by (\ref{BasicTrans}). Suppose $\norm[x] \geq T \log^2T$. Then for any constant $C > 0$, we have
	$$ \widetilde{h}_j(ix) \ll_C \norm[x]^{-C}. $$
\end{proposition}
\begin{proof}
	We assume $\norm[x] \geq T$ in the sequel. We only write details for $h_1,h_2$, since the case for $h_3$ is quite similar. We apply (\ref{ContShift}) to $h=h_j$ and bound the discrete sum in absolute value by Lemma \ref{RefAuxTrans} as
	$$ \sum_{k=0}^{n-1} \extnorm{ h_j^*\left( \frac{1}{2}-ix+k \right) \frac{(-1)^k \Gamma \left( \frac{1}{2}-ix+k \right)^3}{k! \Gamma \left( 1-2ix+2k \right) \Gamma \left( \frac{1}{2}-ix \right)} } \ll_n T e^{-\frac{\norm[x]}{T}} \left( \frac{1+\norm[x]}{\Delta} \right)^{n-\frac{3}{2}}. $$
	For any constant $C > 0$, we write $y = \norm[x]/T$ and find that if $y \geq \log^2T$
\begin{align*} 
	(1+\norm[x])^C \cdot T e^{-\frac{\norm[x]}{T}} \left( \frac{1+\norm[x]}{\Delta} \right)^{n-\frac{3}{2}} &\ll T^{C+n-\frac{1}{2}} e^{-y} (1+y)^{C+n-\frac{3}{2}} \\
	&\begin{cases} \leq e^{-\frac{y}{2}} (1+y)^{C+n-\frac{3}{2}} \ll_{n,C} 1 & \text{if } \log T \geq 2(C+n-1/2) \\ \ll_{n,C} 1 & \text{otherwise} \end{cases}.
\end{align*}
	Hence these terms are bounded by $(1+\norm[x])^{-C}$ as required. Writing $m=c+n$, we bound the remaining integral using Lemma \ref{RefAuxTrans} and Stirling's bound as
\begin{multline*}
	\int_{(c+n)} h_j^*(s) \frac{\Gamma(s)^3 \Gamma \left( \frac{1}{2}-ix-s \right)}{\Gamma(2s) \Gamma \left( \frac{1}{2}-ix \right)} \frac{\ud s}{2\pi i} \\
	\ll \Delta^{-2m} T \int_{-\infty}^{\infty} e^{-\frac{\norm[\tau]}{T}} e^{-\frac{\pi}{2}(\norm[\tau]+\norm[x+\tau]-\norm[x])} (1+\norm[\tau])^{m-1} (1+\norm[x+\tau])^{-m} \ud \tau.
\end{multline*}
	The major contribution of the common integral above comes from $\norm[\tau+x/2] \leq \norm[x]/2$, for which we have
\begin{align*}
	&\quad \int_{\norm[\tau+\frac{x}{2}] \leq \norm[\frac{x}{2}]} e^{-\frac{\norm[\tau]}{T}} e^{-\frac{\pi}{2}(\norm[\tau]+\norm[x+\tau]-\norm[x])} (1+\norm[\tau])^{m-1} (1+\norm[x+\tau])^{-m} \ud \tau \\
	&\ll_m (1+\norm[x])^{-m} \int_{\norm[\tau+\frac{x}{4}] \leq \norm[\frac{x}{4}]} e^{-\frac{\norm[\tau]}{T}} (1+\norm[\tau])^{m-1} \ud \tau + e^{-\frac{\norm[x]}{2T}} (1+\norm[x])^{m-1} \int_{\norm[\tau+\frac{3x}{4}] \leq \norm[\frac{x}{4}]} (1+\norm[x+\tau])^{-m} \ud \tau \\
	&\ll (1+\norm[x])^{-m} T^m + e^{-\frac{\norm[x]}{2T}} (1+\norm[x])^{m-1}.
\end{align*}
	Arguing similarly as for the discrete sum above, we find
	$$ \Delta^{-2m} T e^{-\frac{\norm[x]}{2T}} (1+\norm[x])^{m-1} \ll_{n,C} (1+\norm[x])^{-C}, \quad \text{if } \norm[x] \geq T\log^2T. $$
	Assuming $n$ large so that $m > C$ and writing again $y = \norm[x]/T \geq 1$, we find
	$$ \Delta^{-2m} T (1+\norm[x])^{C-m} T^m \leq \Delta^{-2m} T^{C+1} y^{C-m} \leq \Delta^{-2m} T^{C+1}. $$
	We take $n$ so large that $\Delta^{-2m} T^{C+1} \leq 1$, and conclude the proof.
\end{proof}

\begin{proof}[Proof of Theorem \ref{DualWtBd} (1)]
	Applying Proposition \ref{BasisTransProp} (1) to the weight function \eqref{ConcWtR} \& \eqref{eq: ConcWtRBis} we get
	$$ \widetilde{h}(ix,\varepsilon') \ll \sideset{}{_{\varepsilon}} \sum \int_{-\infty}^{\infty} (1+\norm[\tau])^{-1} \extnorm{h(i\tau, \varepsilon)} \norm[\tau \tanh(\pi \tau)] \ud \tau \ll 1. $$
	The bound for $\norm[x] \geq T \log^2 T$ is given by Proposition \ref{DualWtRD}.
\end{proof}

\section{Local Analysis at non-Archimedean Places}

	We turn to the proof of Theorem \ref{DualWtBd} (2) in this section, while providing auxiliary results needed for the estimation of the degenerate terms. The proof is formally summarized in the end of \S 4.2.

	\subsection{Choice of Test Function}
	\label{sec: TFNA}
	
	Recall that a unitary character $\chi_0$ of $\F^{\times}$ is fixed as parameter. Write $\cond(\chi_0)=n \geq 1$, $t := \varpi^{-n}$ and $q=\norm[\varpi]^{-1}$. Take
\begin{equation} \label{TestPhiNA}
	\Psi = \rpL_{n(t)} \rpR_{n(t)} \phi_0, \quad \phi_0 \begin{pmatrix} x_1 & x_2 \\ x_3 & x_4 \end{pmatrix} = \chi_0 \left( \frac{x_4}{x_1} \right) \id_{\gp{K}_0[\vp^n]} \begin{pmatrix} x_1 & x_2 \\ x_3 & x_4 \end{pmatrix}.
\end{equation}

\noindent Since $\phi_0$ is a character on its support (group) $\gp{K}_0[\vp^n]$, we have $\phi_0 * \phi_0 = \Vol(\gp{K}_0[\vp^n]) \phi_0$, where the convolution is taken over $\GL_2(\F)$. Consequently, we have
	$$ \Psi = \Vol(\gp{K}_0[\vp^n])^{-1} \left( \rpR_{n(t)} \phi_0^{\vee} \right)^{\vee} * \left( \rpR_{n(t)} \phi_0 \right). $$
	The proof of Lemma \ref{M3VSNormLoc} applies, giving for any irreducible admissible representation $\pi$ of $\GL_2(\F)$
	$$ h(\pi)(\Psi) = \Vol(\gp{K}_0[\vp^n])^{-1} \Pairing{v(\rpR_{n(t)} \phi_0 \mid \pi)}{v(\rpR_{n(t)}\phi_0^{\vee} \mid \pi^{\vee})}, $$
	where for any $\phi \in \Cont_c^{\infty}(\GL_2(\F))$ we have the following analogue of \eqref{eq: LocMotVec}
	$$ v(\phi \mid \pi) = \sideset{}{_{e \in \Bas(\pi)}} \sum \Zeta \left( \tfrac{1}{2}, W_{e^{\vee}} \right) \cdot \pi(\phi).e. $$
	
\begin{lemma} \label{CubeWtLocBdNA}
	Suppose $\pi \in \widehat{\PGL_2(\F)}$ is $\theta$-tempered. We have $\pi(\phi_0) \neq 0$ resp. $\pi^{\vee}(\phi_0^{\vee}) \neq 0$ only if $\cond(\pi \otimes \chi_0^{-1}) \leq n$. Under this condition, we have the lower bounds
	$$ \Pairing{v(\rpR_{n(t)} \phi_0 \mid \pi)}{v(\rpR_{n(t)}\phi_0^{\vee} \mid \pi^{\vee})} \gg_{\theta} q^{-3n}, \quad h(\pi)(\Psi) \gg_{\theta} q^{-2n}. $$
\end{lemma}
\begin{proof}
	$\pi(\phi_0) \neq 0$ only if $V_{\pi}$ contains a vector $e$ satisfying
	$$ \pi \begin{pmatrix} a & b \\ c & d \end{pmatrix}.e = \chi_0 \left( \frac{a}{d} \right) e, \quad \forall \begin{pmatrix} a & b \\ c & d \end{pmatrix} \in \gp{K}_0[\vp^n]. $$
	This is equivalent to
\begin{equation} \label{NewFormCond}
	(\pi \otimes \chi_0^{-1}) \begin{pmatrix} a & b \\ c & d \end{pmatrix}.(e \otimes \chi_0^{-1}) = \chi_0^{-2} (d) (e \otimes \chi_0^{-1}), \quad \forall \begin{pmatrix} a & b \\ c & d \end{pmatrix} \in \gp{K}_0[\vp^n]. 
\end{equation}
	The last condition is obviously equivalent to $\cond(\pi \otimes \chi_0^{-1}) \leq n$ by definition since the central character of $\pi \otimes \chi_0^{-1}$ is $\chi_0^{-2}$. The case for $\pi^{\vee}$ is similar. Moreover, the space of vectors satisfying (\ref{NewFormCond}) is a $\gp{K}_0[\vp^n]$-isotypic subspace. We can find an orthogonal basis $\Bas_0$, containing a new vector $e_0$ of $\pi \otimes \chi_0^{-1}$, of this space and extend it to an orthogonal basis of $\pi$. We obtain
	$$ v(\rpR_{n(t)}\phi_0 \mid \pi) = \Vol(\gp{K}_0[\vp^n]) \sideset{}{_{e \in \Bas_0}} \sum \Zeta \left( \tfrac{1}{2}, \chi_0^{-1}, n(t).W_{e^{\vee}} \right) \cdot n(t).(e \otimes \chi_0). $$
	Similarly, we have
	$$ v(\rpR_{n(t)}\phi_0^{\vee} \mid \pi^{\vee}) = \Vol(\gp{K}_0[\vp^n]) \sideset{}{_{e \in \Bas_0}} \sum \Zeta \left( \tfrac{1}{2}, \chi_0, n(t).W_{e} \right) \cdot n(t).(e^{\vee} \otimes \chi_0^{-1}). $$
	Consequently, we get
\begin{align*}
	\Pairing{v(\rpR_{n(t)} \phi_0 \mid \pi)}{v(\rpR_{n(t)}\phi_0^{\vee} \mid \pi^{\vee})} &= \Vol(\gp{K}_0[\vp^n])^2 \cdot \sideset{}{_{e \in \Bas_0}} \sum \Zeta \left( \tfrac{1}{2}, \chi_0^{-1}, n(t).W_{e^{\vee}} \right) \cdot \Zeta \left( \tfrac{1}{2}, \chi_0, n(t).W_{e} \right) \\
	&\geq \Vol(\gp{K}_0[\vp^n])^2 \cdot \Zeta \left( \tfrac{1}{2}, \chi_0^{-1}, n(t).W_0^{\vee} \right) \cdot \Zeta \left( \tfrac{1}{2}, \chi_0, n(t).W_0 \right),
\end{align*}
	where we have written $W_0 = W_{e_0}$, and have used the positivity of each summand. For an unramified character $\chi$, define $f_{\chi}(x)=\chi(x) \id_{\vo}(x)$. By the new vector theory, we can take
	$$ W_0(x) = \left\{ \begin{matrix} f_{\chi_1}(x) \norm[x]_{\F}^{\frac{1}{2}} & \text{if } \pi \simeq \pi(\chi,\chi^{-1}) \text{ with } \chi_1:=\chi\chi_0^{-1} \text{ or } \chi_1:=\chi^{-1}\chi_0^{-1} \text{ unramified} \\ (f_{\chi_1}*f_{\chi_2})(x) \norm[x]_{\F}^{\frac{1}{2}} & \text{if } \pi \simeq \pi(\chi,\chi^{-1}) \text{ with } \chi_1:=\chi\chi_0^{-1} \text{ and } \chi_2:=\chi^{-1}\chi_0^{-1} \text{ unramified} \\ \id_{\vo^{\times}}(x) & \text{otherwise} \end{matrix} \right., $$
	where the convolution is taken for $(\F^{\times},\ud^{\times} x)$. It follows that (see \cite[Proposition 4.6]{Wu14} for example)
	$$ \extnorm{\Zeta \left( \tfrac{1}{2}, \chi_0, n(t).W_0 \right)} = q^{-\frac{n}{2}} \zeta_{\vp}(1), \quad \Norm[W_0]^2 \ll_{\theta} 1. $$
	The desired bound follows readily.
\end{proof}

	\subsection{Dual Weight Estimation}
	
		\subsubsection{Preliminary Reductions}
	
	Recall the formula \eqref{eq: LocDWtDisDef} of $\widetilde{h}(\chi)(\Psi)$. A simple change of variables gives
\begin{multline*}
	\frac{\widetilde{h}(\chi)(\Psi)}{\zeta_{\vp}(1)^4} = \int_{\F^4} \phi_0 \begin{pmatrix} x_1 & x_2 \\ x_3 & x_4 \end{pmatrix} \chi \left( \frac{(x_1 + tx_3)(x_4 - tx_3)}{(x_2 - t(x_1 - x_4) - t^2 x_3)x_3} \right) \cdot \\
	\frac{\prod \ud x_i}{\norm[(x_1 + tx_3)(x_4 - tx_3)(x_2 - t(x_1 - x_4) - t^2 x_3)x_3]_{\F}^{\frac{1}{2}}}.
\end{multline*}

\noindent Under the change of variables 
	$$ x_3 \mapsto x_3(1+u), \quad x_1 \mapsto x_1-tux_3, \quad x_2 \mapsto x_2-t^2ux_3, \quad x_4 \mapsto x_4+tux_3, $$
	for an arbitrary $u \in \vp^n$, we have $ \widetilde{h}(\chi)(\Psi) = \chi(1+u)^{-1} \widetilde{h}(\chi)(\Psi) $. Hence $\widetilde{h}(\chi)(\Psi) \neq 0$ only if $m:=\cond(\chi) \leq n$, which we assume from now on.

	We notice that $\phi_0$ is given in the coordinates $(z,u,x,y) \in (\F^{\times})^2 \times \F^2$ of an open dense subset of $\Mat_2(\F)$ (actually of $\GL_2(\F) - w\gp{B}(\F)$) by a product of functions in each variable as
	$$ \phi_0 \left( \begin{pmatrix} z & \\ & z \end{pmatrix} \begin{pmatrix} 1 & \\ x & 1 \end{pmatrix} \begin{pmatrix} 1 & y \\ & 1 \end{pmatrix} \begin{pmatrix} u & \\ & 1 \end{pmatrix} \right) = \mathbbm{1}_{(\vo^{\times})^2}(z,u) \mathbbm{1}_{\vo^2}(\varpi^{-n}x,y) \chi_0(u)^{-1}. $$
	Making the change of variables
	$$ \begin{pmatrix} x_1 & x_2 \\ x_3 & x_4 \end{pmatrix} = \begin{pmatrix} z & \\ & z \end{pmatrix} \begin{pmatrix} 1 & \\ x & 1 \end{pmatrix} \begin{pmatrix} 1 & y \\ & 1 \end{pmatrix} \begin{pmatrix} u & \\ & 1 \end{pmatrix} = \begin{pmatrix} zu & zy \\ zux & z(1+xy) \end{pmatrix}, $$
	whose Jacobian is equal to $\norm[z^3u]$, we get
\begin{multline*}
	\frac{\widetilde{h}(\chi)(\Psi)}{\zeta_{\vp}(1)^4} = q^{-n} \int_{(\vo^{\times})^2} \int_{\vo^2} \chi \left( \frac{(1+x)(1-x(u-\varpi^ny))}{x(1-(1+x)(u-\varpi^ny))} \right) \chi_0(u)^{-1} \cdot \\
	\frac{\norm[z] \ud z \ud u \ud x \ud y}{\extnorm{x(1+x)(1-x(u-\varpi^ny))(1-(1+x)(u-\varpi^ny))}_{\F}^{\frac{1}{2}}} \\
	= \frac{q^{-n}}{\zeta_{\vp}(1)} \int_{\vo^{\times}} \int_{\vo} \chi \left( \frac{(1+x)(1-xu)}{x(1-(1+x)u)} \right) \chi_0(u)^{-1} \cdot \frac{\ud u \ud x}{\extnorm{x(1+x)(1-xu)(1-(1+x)u)}_{\F}^{\frac{1}{2}}},
\end{multline*}
	where we have used $\chi_0(u)=\chi_0(u-\varpi^n y)$ for $u \in \vo^{\times}, y \in \vo$ in the last line. The consecutive changes of variables $u \mapsto u^{-1}$ and $u \mapsto 1+x+y$ then gives
\begin{align*}
	\frac{\widetilde{h}(\chi)(\Psi)}{\zeta_{\vp}(1)^4} &= \frac{q^{-n}}{\zeta_{\vp}(1)} \int_{\vo^{\times}} \int_{\vo} \chi \left( \frac{(1+x)(u-x)}{x(u-1-x)} \right) \chi_0(u) \cdot \frac{\ud u \ud x}{\extnorm{x(1+x)(u-x)(u-1-x)}_{\F}^{\frac{1}{2}}} \\
	&= \frac{q^{-n}}{\zeta_{\vp}(1)} \int_{\vo^2} \mathbbm{1}_{\vo^{\times}}(1+x+y) \chi_0(1+x+y) \chi \left( \frac{(1+x)(1+y)}{xy} \right) \frac{\ud x \ud y}{\extnorm{xy(1+x)(1+y)}_{\F}^{\frac{1}{2}}}.
\end{align*}
	Introducing the function (note the symmetry $f(x,y) = f(y,x)$)
	$$ f(x,y) = f(x,y \mid \chi) := \mathbbm{1}_{\vo^{\times}}(1+x+y)\chi_0(1+x+y) \chi \left( \frac{(1 + x)(1 + y)}{xy} \right) \frac{1}{ \norm[(1 + x)(1+y)xy]_{\F}^{\frac{1}{2}} }, $$
	we rewrite the above equation as
\begin{equation} \label{Decomp}
	\frac{\widetilde{h}(\chi)(\Psi)}{\zeta_{\vp}(1)^3} = q^{-n} \int_{\vo^2} f(x,y) \ud x \ud y.
\end{equation}

\noindent Note that if $x y \notin \vo^{\times}$ (hence $\in \vp$) then $1+x+y \in \vo^{\times}$ implies $1+x, 1+y \in \vo^{\times}$. Hence
	$$ \int_{\vo^2 - (\vo^{\times})^2} f(x,y) \ud x\ud y = \int_{\vo^2 - (\vo^{\times})^2} \id_{\vo^{\times}}(1+x+y) \chi_0(1+x+y) \chi \left( \frac{(1 + x)(1 + y)}{xy} \right) \frac{\ud x\ud y}{\norm[xy]_{\F}^{\frac{1}{2}}}. $$
	On the other hand, we have (by the change of variables $x \mapsto x-1, y \mapsto y-1$ in the second integral below)
\begin{multline*}
	\int_{(\vo^{\times})^2} f(x,y) \ud x \ud y = \\
	\int_{(\vo^{\times})^2} \mathbbm{1}_{\vo^{\times}}(1+x+y) \mathbbm{1}_{\vo^{\times}}(1+x) \mathbbm{1}_{\vo^{\times}}(1+y) \chi_0(1+x+y) \chi \left( \frac{(1 + x)(1 + y)}{xy} \right) \ud x\ud y \\
	+ \int_{\vo^2 - (\vo^{\times})^2} \mathbbm{1}_{\vo^{\times}}(-1+x+y) \chi_0(-1+x+y) \chi \left( \frac{xy}{(-1 + x)(-1 + y)} \right) \frac{\ud x \ud y}{\norm[xy]_{\F}^{\frac{1}{2}}} \\
	= \int_{(\vo^{\times})^2} \mathbbm{1}_{\vo^{\times}}(1+x+y) \mathbbm{1}_{\vo^{\times}}(1+x) \mathbbm{1}_{\vo^{\times}}(1+y) \chi_0(1+x+y) \chi \left( \frac{(1 + x)(1 + y)}{xy} \right) \ud x\ud y \\
	+ \int_{\vo^2 - (\vo^{\times})^2} \mathbbm{1}_{\vo^{\times}}(1+x+y) \chi_0(1+x+y) \chi^{-1} \left( \frac{(1 + x)(1 + y)}{xy} \right) \frac{\ud x \ud y}{\norm[xy]_{\F}^{\frac{1}{2}}}.
\end{multline*}
	Let
	$$ g(x,y) = g(x,y \mid \chi) :=  \mathbbm{1}_{\vo^{\times}}(1+x) \mathbbm{1}_{\vo^{\times}}(1+y) f(x,y \mid \chi), $$
	and introduce for integers $k,\ell \geq 0$
\begin{multline*} 
	\widetilde{h}^{k,\ell}(\chi_0 \mid \chi) := \int_{\varpi^k \vo^{\times} \times \varpi^{\ell} \vo^{\times}} g(x,y \mid \chi) \ud x\ud y \\
	= \frac{q^{-\frac{k+\ell}{2}}}{\chi(\varpi)^{k+\ell}} \int_{(\vo^{\times})^2} \mathbbm{1}_{\vo^{\times}}(1+\varpi^k u + \varpi^{\ell} v) \mathbbm{1}_{\vo^{\times}}(1+\varpi^k u) \mathbbm{1}_{\vo^{\times}}(1+\varpi^{\ell} v) \cdot \\
	\chi_0(1+\varpi^ku +\varpi^{\ell}v) \chi \left( \frac{(1+\varpi^k u) (1+\varpi^{\ell}v)}{uv} \right) \ud u \ud v.
\end{multline*}
	We can rewrite \eqref{Decomp} as
\begin{equation} \label{DecompBis}
	\frac{\widetilde{h}(\chi)(\Psi)}{\zeta_{\vp}(1)^3} = q^{-n} \cdot \left\{ \widetilde{h}^{0,0}(\chi_0 \mid \chi) + \sideset{}{_{\substack{k, \ell \geq 0 \\ (k,\ell) \neq (0,0)}}} \sum \left( \widetilde{h}^{k,\ell}(\chi_0 \mid \chi) + \widetilde{h}^{k,\ell}(\chi_0 \mid \chi^{-1}) \right) \right\}.
\end{equation}
	We are thus led to bounding $\widetilde{h}^{k,\ell}(\chi_0 \mid \chi^{\pm 1})$.
	
		\subsubsection{Case $1 \leq m = \cond(\chi) \leq n-1$}
	
	We claim that only the terms for $k=\ell=n-m$ is non-vanishing. In fact, if $\ell > n-m$, then for any $y \in \varpi^{\ell}\vo^{\times}, \delta \in (1+\vp^{m-1}) \cap \vo^{\times}$ and $1+x+y \in \vo^{\times}$ we have
	$$ 1+y\delta \in (1+y)(1+\vp^m), \quad 1+x+y\delta \in (1+x+y)(1+\vp^n). $$
	Hence $g(x,y\delta \mid \chi^{\pm 1}) = \chi^{\mp 1}(\delta)g(x,y \mid \chi^{\pm 1})$. Thus $\widetilde{h}^{k,\ell}(\chi_0 \mid \chi^{\pm 1}) = 0$. While if $\ell < n-m$, then for any $y \in \varpi^{\ell}\vo^{\times}, \delta \in \vp^m$ and $1+x+y, 1+x, 1+y \in \vo^{\times}$ we have
	$$ 1+y(1+\delta) \in (1+y)(1+\vp^m), \quad 1+x+y(1+\delta) \in (1+x+y)(1+\vp^{m+\ell}). $$
	Hence $g(x,y(1+\delta)) = \chi_0(1+y\delta(1+x+y)^{-1})g(x,y)$, implying
\begin{multline*}
	\int_{\varpi^k \vo^{\times} \times \varpi^{\ell} \vo^{\times}} g(x,y) \ud x\ud y = \frac{1}{\Vol(\vp^{m})} \int_{\vp^{m}} \int_{\varpi^k \vo^{\times} \times \varpi^{\ell} \vo^{\times}} g(x,y(1+\delta)) \ud x\ud yd\delta \\
	= \frac{1}{\Vol(\vp^{m})} \int_{\varpi^k \vo^{\times} \times \varpi^{\ell} \vo^{\times}} g(x,y) \left( \int_{\vp^{m}} \chi_0 \left( 1 + \frac{y\delta}{1+x+y} \right) d\delta \right) \ud x\ud y = 0,
\end{multline*}
	since as $\delta$ traverses $\vp^{m}$, $1+y\delta(1+x+y)^{-1}$ traverses $1+\vp^{m+\ell} \subseteq 1+\vp^{n-1}$ on which $\chi_0$ is non-trivial. The claim is proved.
	
\noindent A change of variables
	$$ x = \varpi^{n-m}u, \quad y = \varpi^{n-m}v, \qquad \text{for } u,v \in \vo^{\times} $$
	transforms the term(s) for $k=\ell=n-m \geq 1$ as
\begin{multline*}
	\widetilde{h}^{n-m,n-m}(\chi_0 \mid \chi) \\
	= \frac{q^{-(n-m)}}{\chi(\varpi)^{2(n-m)}} \int_{(\vo^{\times})^2} \chi_0(1+\varpi^{n-m}u+\varpi^{n-m}v) \chi \left( \frac{(1+\varpi^{n-m}u)(1+\varpi^{n-m}v)}{uv} \right) \ud u\ud v.
\end{multline*}

\begin{lemma} \label{ExpBd1}
	Recall $\cond(\chi_0)=n, \cond(\chi)=m$ and assume $1 \leq m \leq n-1$. Then we have
	$$ \widetilde{h}^{n-m,n-m}(\chi_0 \mid \chi^{\pm 1}) \ll q^{-n}, \quad \widetilde{h}(\chi)(\Psi) \ll q^{-2n}. $$
	where the implied constant either is absolute if $2 \in \vo^{\times}$, or depends only on $[\F:\Q_2]$ if $2 \in \vp$.
\end{lemma}
\begin{proof}
	By the following transformation
\begin{align*}
	\widetilde{h}^{n-m,n-m}(\chi_0 \mid \chi) &= \frac{q^{-(n-m)}}{\chi(\varpi)^{2(n-m)}} \int_{(\vo^{\times})^2} \chi_0 \left( \frac{uv-\varpi^{2(n-m)}}{(u-\varpi^{n-m})(v-\varpi^{n-m})} \right) \chi(uv) \ud u\ud v \\
	&= \frac{q^{-(n-m)}}{\chi(\varpi)^{2(n-m)}} \int_{(\vo^{\times})^2} \chi_0 \left( \frac{1-\varpi^{2(n-m)}uv}{(1-\varpi^{n-m}u)(1-\varpi^{n-m}v)} \right) \chi^{-1}(uv) \ud u\ud v,
\end{align*}
	the first bound corresponds precisely to \cite[Lemma 2.8 \& 2.11]{PY19_All} in the cases $\F = \Q_p$. The proof works also in the general case. We leave the details to the reader. The second follows from the first and \eqref{DecompBis}.
\end{proof}

		\subsubsection{Case $m=\cond(\chi)=0$}
		\label{DGAuxNA}
	
	In this case, we easily get
\begin{align*}
	\widetilde{h}^{k,\ell}(\chi_0 \mid \chi^{\pm 1}) &= \chi(\varpi)^{\mp (k+\ell)} q^{-\frac{k+\ell}{2}} \cdot \\
	&\quad \int_{(\vo^{\times})^2} \mathbbm{1}_{\vo^{\times}}(1+\varpi^ku+\varpi^{\ell}v) \mathbbm{1}_{\vo^{\times}}(1+\varpi^ku) \mathbbm{1}_{\vo^{\times}}(1+\varpi^{\ell}v) \chi_0(1+\varpi^ku+\varpi^{\ell}v) \ud u\ud v.
\end{align*}
	We distinguish cases according to the values of $\min(k,\ell)$ and $\max(k,\ell)$.

\noindent (1) $\min(k,\ell) < n-1$. Say $k < n-1$. For any $\delta \in \vp^{n-1-k} \subseteq \vp$, we have
	$$ \mathbbm{1}_{\vo^{\times}}(1+\varpi^ku+\varpi^{\ell}v) = \mathbbm{1}_{\vo^{\times}}(1+\varpi^ku(1+\delta)+\varpi^{\ell}v), \quad \mathbbm{1}_{\vo^{\times}}(1+\varpi^ku) = \mathbbm{1}_{\vo^{\times}}(1+\varpi^ku(1+\delta)). $$
	It follows that
\begin{multline*}
	\int_{\vo^{\times}} \mathbbm{1}_{\vo^{\times}}(1+\varpi^ku+\varpi^{\ell}v) \mathbbm{1}_{\vo^{\times}}(1+\varpi^ku) \chi_0(1+\varpi^ku+\varpi^{\ell}v) \ud u \\
	= \frac{1}{\Vol(\vp^{n-1-k})} \int_{\vp^{n-1-k}} \int_{\vo^{\times}} \mathbbm{1}_{\vo^{\times}}(1+\varpi^ku+\varpi^{\ell}v) \mathbbm{1}_{\vo^{\times}}(1+\varpi^ku) \chi_0(1+\varpi^ku(1+\delta)+\varpi^{\ell}v) \ud u d\delta \\
	= \frac{1}{\Vol(\vp^{n-1-k})} \int_{\vo^{\times}} \mathbbm{1}_{\vo^{\times}}(1+\varpi^ku+\varpi^{\ell}v) \mathbbm{1}_{\vo^{\times}}(1+\varpi^ku) \chi_0(1+\varpi^ku +\varpi^{\ell}v) \cdot \\
	\int_{\vp^{n-1-k}} \chi_0 \left( 1 + \frac{\varpi^ku\delta}{1+\varpi^ku +\varpi^{\ell}v} \right) d\delta \ud u = 0,
\end{multline*}
	since as $\delta$ traverses $\vp^{n-1-k}$, $1+\varpi^ku\delta(1+\varpi^ku +\varpi^{\ell}v)^{-1}$ traverses $1+\vp^{n-1}$, on which $\chi_0$ is non trivial. Hence $\widetilde{h}^{k,\ell}(\chi_0 \mid \chi^{\pm 1}) = 0$.
	
\noindent (2) $\min(k,\ell) \geq n$. The integrand is equal to $1$ and we simply have
	$$ \widetilde{h}^{k,\ell}(\chi_0 \mid \chi^{\pm 1}) = \chi(\varpi)^{\mp (k+\ell)} q^{-\frac{k+\ell}{2}} \cdot \zeta_{\vp}(1)^{-2}. $$
	Consequently, we get
	$$ \sideset{}{_{\substack{k,\ell \geq 0 \\ \min(k,\ell) \geq n}}} \sum \extnorm{\widetilde{h}^{k,\ell}(\chi_0 \mid \chi^{\pm 1})} \ll q^{-n}. $$
	
\noindent (3) $\min(k,\ell) = n-1, \max(k,\ell) \geq n$. Say $k=n-1, \ell \geq n$. Then we have
\begin{multline*}
	\widetilde{h}^{n-1,\ell}(\chi_0 \mid \chi^{\pm 1}) = \frac{\chi(\varpi)^{\mp (n-1+\ell)} q^{-\frac{n-1+\ell}{2}}}{\zeta_{\vp}(1)} \cdot \int_{\vo^{\times}} \mathbbm{1}_{\vo^{\times}}(1+\varpi^{n-1}u) \chi_0(1+\varpi^{n-1}u) \ud u \\
	= -\chi(\varpi)^{\mp (n-1+\ell)} \zeta_{\vp}(1)^{-1} q^{-\frac{n+1+\ell}{2}}.
\end{multline*}
	Consequently, we get
	$$ \sideset{}{_{\substack{\min(k,\ell) = n-1 \\ \max(k,\ell) \geq n}}} \sum \extnorm{\widetilde{h}^{k,\ell}(\chi_0 \mid \chi^{\pm 1})} \ll q^{-n-\frac{1}{2}}. $$
	
\noindent (4) $k=\ell=n-1$ and $n \geq 2$. We integrate step by step and get
\begin{multline*}
	\widetilde{h}^{n-1,n-1}(\chi_0 \mid \chi^{\pm 1}) = \chi(\varpi)^{\mp 2(n-1)} q^{-(n-1)} \int_{(\vo^{\times})^2} \chi_0(1+\varpi^{n-1}(u+v)) \ud u\ud v \\
	= - \chi(\varpi)^{\mp 2(n-1)} q^{-n} \int_{\vo^{\times}} \chi_0(1+\varpi^{n-1}u) \ud u = \chi(\varpi)^{\mp 2(n-1)} q^{-n-1}.
\end{multline*}

\noindent (5) $k=\ell=0$ and $n=1$. We integrate step by step and get
\begin{multline*}
	\widetilde{h}^{0,0}(\chi_0 \mid \chi^{\pm 1}) = \int_{(\vo^{\times})^2} \mathbbm{1}_{\vo^{\times}}(1+u+v) \mathbbm{1}_{\vo^{\times}}(1+u) \mathbbm{1}_{\vo^{\times}}(1+v) \chi_0(1+u+v) \ud u \ud v \\
	= - q^{-1} \int_{\vo^{\times}-(-1+\vp)} \left( \chi_0(u) + \chi_0(1+u) \right) \ud u = q^{-2} \left\{ \chi_0(-1) + \chi_0(1) \right\}.
\end{multline*}

	We conclude by (\ref{DecompBis}) that in the case $m=\cond(\chi)=0$
	$$ \widetilde{h}(\chi)(\Psi) \ll q^{-2n}. $$
	
		\subsubsection{Case $m=\cond(\chi)=n$}
	
	The same argument as in the beginning of the case for $1 \leq m \leq n-1$ shows that if $(k,\ell) \neq (0,0)$ then $\widetilde{h}^{k,\ell}(\chi_0 \mid \chi^{\pm 1}) = 0$. The case for $m=n\geq 2$ is treated by the method of $p$-adic stationary phase. In particular, one finds the relevant \emph{coset-problem} at the set of critical points of a phase function. To introduce it, we begin with the following definition.
	
\begin{definition} \label{rDef}
	For any $\xi \in \vo^{\times}$ and any integer $\alpha \geq 1$, we denote
	$$ r(\xi, \vp^{\alpha}) := \left\{ \delta \in \vo/\vp^{\alpha} \ \middle| \ \delta^2 + \delta - \xi^2 \in \vp^{\alpha} \right\}. $$
	For any $\Delta \in \vo$, we denote
	$$ \rho(\Delta, \vp^{\alpha}) := \left\{ \delta \in \vo/\vp^{\alpha} \ \middle| \ \delta^2 \equiv \Delta \pmod{\vp^{\alpha}} \right\}. $$
\end{definition}
\begin{remark}
	Note that $\delta \in r(\xi,\vp^{\alpha})$ implies $\delta+\xi, 1+\delta+\xi, 1+2\delta+2\xi \in \vo^{\times}$.
\end{remark}

\begin{lemma} \label{rSize}
	We have the following estimation of $\extnorm{r(\xi,\vp^{\alpha})}, \alpha \geq 1$ (see Definition \ref{rDef}).
\begin{itemize}
	\item[(1)] If $\vp \nmid 2$, then $\extnorm{r(\xi,\vp^{\alpha})} = \extnorm{ \rho(\Delta, \vp^{\alpha}) }$ for $\Delta := 1+4\xi^2$. We have
	$$ \extnorm{\rho(\Delta,\vp^{\alpha})} \left\{ \begin{matrix} =1+ \left( \frac{\Delta}{\vp} \right) & \text{if } \Delta \notin \vp \\ =q^{\lfloor \frac{\alpha}{2} \rfloor} & \text{if } \Delta \in \vp^{\alpha} \\ \leq 2 q^{\frac{\mathrm{ord}(\Delta)}{2}} \mathbbm{1}_{2 \mid \mathrm{ord}(\Delta)} & \text{if } \Delta \in \vp - \vp^{\alpha} \end{matrix} \right. . $$
	\item[(2)] If $\vp \mid 2$, then we always have $\extnorm{r(\xi,\vp^{\alpha})} \leq 2$.
\end{itemize}
\end{lemma}
\begin{proof}
	(1) is the same as \cite[Lemma 3.1]{PY19_All}. For (2), we note that the polynomial
	$$ x^2 + x - \xi^2 \pmod{\vp} $$
	is separable, since $\vo/\vp$ has characteristic $2$. By Hensel's lemma, its solvability mod $\vp^{\alpha}$ is the same as its solvability mod $\vp$, with the same number of solutions $0$ or $2$. Hence we have the stated bound.
\end{proof}

\begin{lemma} \label{ExpBd2}
	We associate to $\chi_0, \chi$ a number $\xi \in \vo^{\times} / (1+\vp^{\lceil \frac{n}{2}} \rceil)$ such that $\chi(\exp(x)) = \chi_0(\exp(\xi x))$ for $x \in \vp^{\lfloor \frac{n}{2} \rfloor}$ (with $n$ sufficiently large if $\vp \mid 2$, see Lemma \ref{GpsId}).
\begin{itemize}
	\item[(1)] If $n=m=2\alpha$ with $\alpha \geq 1$, then we have
	$$ \extnorm{\widetilde{h}^{0,0}(\chi_0 \mid \chi)} \leq \extnorm{r(\xi, \vp^{\alpha})} \cdot q^{-n}. $$
	\item[(2)] If $n=m=2\alpha+1$ with $\alpha \geq 1$, and $\vp \nmid 2$, writing $\Delta := 1+4\xi^2$, then we have
	$$ \extnorm{\widetilde{h}^{0,0}(\chi_0 \mid \chi)} \leq q^{-n} \cdot \left\{ \begin{matrix} 2 & \text{if } \Delta \notin \vp \\ q^{\frac{1}{2}} \mathbbm{1}_{\Delta \in \vp^2} \extnorm{\rho(\varpi^{-2}\Delta, \vp^{\alpha-1})} & \text{if } \Delta \in \vp \end{matrix} \right. . $$
	\item[(3)] If $n=m=2\alpha+1$ with $\alpha \geq 1$, and $\vp \mid 2$, then
	$$ \extnorm{\widetilde{h}^{0,0}(\chi_0 \mid \chi)} \leq \left\{ \begin{matrix} 2^{3[\F:\Q_2]+1} \cdot q^{-n} & \text{if } \alpha \geq \alpha_{\F} \\ 1 & \text{if } \alpha < \alpha_{\F} \end{matrix} \right.,  $$
	where $\alpha_{\F} := \max(\mathrm{ord}(2)+1, (3 \mathrm{ord}(2) + 1)/2)$.
	\item[(4)] If $n=m=1$, then
	$$ \widetilde{h}^{0,0}(\chi_0 \mid \chi) \ll \left\{ \begin{matrix} q^{-1} & \text{if } \vp \nmid 2 \\ 1 & \text{if } \vp \mid 2 \end{matrix} \right. . $$
\end{itemize}
\end{lemma}

\noindent We postpone the proof to the next subsection. For the moment, we draw the following consequence.
\begin{corollary} \label{DWtBdNA}
	Assume $n \leq 2$. Then we have $\widetilde{h}(\chi)(\Psi) \neq 0$ only if $\cond(\chi) \leq n$. Under this condition, we have the bound
	$$ \widetilde{h}(\chi)(\Psi) \ll q^{-2n}, $$
	where the implied constant is absolute if $\vp \nmid 2$, or depends only on the degree of the residue field extension if $\vp \mid 2$.
\end{corollary}
\begin{proof}
	If $\cond(\chi) < n$, then this is simply a summary of the discussion before Lemma \ref{ExpBd2}. If $\cond(\chi)=n$, then the required bound of the dual weight follows from Lemma \ref{ExpBd2} by noting that (2) \& (3) do not apply, and that $\norm[r(\xi,\vp)] \leq 2$ by Lemma \ref{rSize}.
\end{proof}

\begin{proof}[Proof of Theorem \ref{DualWtBd} (2)]
	It suffice to combine the lower bound of $h$ in Lemma \ref{CubeWtLocBdNA} and the upper bound of $\widetilde{h}$ is in Corollary \ref{DWtBdNA}.
\end{proof}

	\subsection{Auxiliary Bounds of Exponential Sums}
	\label{sec: BdExpSums}
	
	We prove Lemma \ref{ExpBd1} \& \ref{ExpBd2} in this subsection.	
	
\begin{lemma} \label{GpsId}
	Recall that for any integer $\alpha \geq 1$, $U_{\alpha} := 1+\vp^{\alpha}$ is a subgroup of $\vo^{\times}$.
\begin{itemize}
	\item[(1)] For any integers $\beta \geq \alpha \geq 1$, the following two maps
	$$ \log : U_{\beta}/U_{\beta+\alpha} \to \vp^{\beta}/\vp^{\beta+\alpha}, \quad 1+x \mapsto x, $$
	$$ \exp: \vp^{\beta}/\vp^{\beta+\alpha} \to U_{\beta}/U_{\beta+\alpha}, \quad y \mapsto 1+y $$
	are inverse to each other and establish group isomorphisms.
	\item[(2)] For integers $\alpha \geq \max(\mathrm{ord}(2)+1, (3 \mathrm{ord}(2) + 1)/2)$, the following two maps
	$$ \log : U_{\alpha}/U_{2\alpha+1} \to \vp^{\alpha}/\vp^{2\alpha+1}, \quad 1+x \mapsto x - \frac{x^2}{2}, $$
	$$ \exp: \vp^{\alpha}/\vp^{2\alpha+1} \to U_{\alpha}/U_{2\alpha+1}, \quad y \mapsto 1+y+\frac{y^2}{2} $$
	are inverse to each other and establish group isomorphisms.
\end{itemize}
\end{lemma}
\begin{proof}
	Elementary.
\end{proof}

\begin{proof}[Proof of Lemma \ref{ExpBd2}]
	We leave (1)-(3) to the reader, since they are either simple or not essentially needed in this paper. Assume $m=n=1$ from now on.

The case $\vp \mid 2$ is trivial. To prove the case $\vp \nmid 2$, we can assume $q:=\Nr(\vp)$ is larger than any fixed number. We follow \cite[\S 9.1]{PY19_CF}. By the change of variables $x=-(u+1)$ and $y=-(v+1)$, we have
	$$ h(u,v; \chi^{-1}) = \chi_0(xy-1) \chi^{-1} \left( \frac{xy}{(x+1)(y+1)} \right). $$
	Hence we are reduced to bounding
	$$ g'(\chi,\chi_0) := \sum_{x \in \mathbb{F}_q - \{ 0,-1 \}} \bar{\chi}(x)\chi(x+1) \left( \sum_{y \in \mathbb{F}_q - \{ x^{-1},0,-1 \}} \bar{\chi}(y)\chi(y+1)\chi_0(xy-1) \right).  $$
	Our exponential sum $g'(\chi,\chi_0)$ differs from the $g(\chi,\psi)$ of \cite[\S 9.1]{PY19_CF} only by a complex conjugate for the first summand of the outer sum over $x$. Hence all arguments in \cite[\S 9.1]{PY19_CF} go through, replacing $\psi$ by $\chi_0$. For example, \cite[(9.3)]{PY19_CF} becomes (using notation from \cite[\S 9.1]{PY19_CF})
	$$ g'(\chi,\chi_0) = - \sum_{x \in W(\mathbb{F}_q)} t_{\mathcal{F}_1}(x) t_{\widetilde{\mathcal{G}}}(x) + O(q). $$
	The sheaf $\overline{\mathcal{F}}_1$, whose trace function is equal to $\overline{t_{\mathcal{F}_1}(x)}$, is associated to $\chi^{-1}$ the same way as $\mathcal{F}_1$ to $\chi$. Hence it is also pure of weight $0$ and rank $1$, thus not isomorphic to $\widetilde{\mathcal{G}}$. Therefore, as the final part of \cite[\S 9.1]{PY19_CF}, we only need a slight variation of \cite[Lemma 13.2]{CI00} to conclude, which we state as\footnote{It seems that (the counterpart of) the last term of the right hand side was missing in Conrey-Iwaniec's paper. However, this does not affect the validity of \cite[Lemma 13.1]{CI00}.}
	$$ \frac{1}{q-1} \sum_{\psi \in \widehat{\mathbb{F}_q^{\times}}} \norm[g'(\chi,\psi)]^2 = q^2-2q-2- \extnorm{\sum_{x \in \mathbb{F}_q-\{ 0,-1 \}} \chi \left( \frac{(x+1)^2}{x} \right) }^2 \leq q^2-2q-2, $$
	where the sum is over all characters of $\mathbb{F}_q^{\times}$. We leave this elementary proof to the reader.
\end{proof}

\section{Degenerate Terms}
	
	Recall the two degenerate terms $DG(\Psi)$ \& $DS(\Psi)$ given by \eqref{DGF} \& \eqref{DSF} respectively. We first bound $DS(\Psi)$, which is easier. By (\ref{DSF}), this amounts to explicitly computing $h_v(\pi(\id,s))$.

\begin{lemma} \label{DSBd}
	For our choice of test function $\Psi_v$ determined by \eqref{ConcWtR} at $v \mid \infty$ and by \eqref{TestPhiNA} at $v=\vp < \infty$, the degenerate term $DS(\Psi) \neq 0$ only if $\F = \Q$ and $\chi_0$ is unramified at all finite places. In the latter case we have $DS(\Psi) \ll_A \Cond(\chi_0)^{-A}$ for any $A > 1$.
\end{lemma}
\begin{proof}
	If $\chi_0$ is ramified at some $\vp < \infty$, then for any $s \in \C$ we have $h_{\vp}(\pi(\id,s))(\Psi_{\vp}) = 0$ by Lemma \ref{CubeWtLocBdNA}, since $\cond(\pi(\norm_{\vp}^s, \norm_{\vp}^{-s}) \otimes \chi_{0,\vp}^{-1}) = 2\cond(\chi_{0,\vp}) > \cond(\chi_{0,\vp})$. Consequently, we have $DS(\Psi)=0$ by (\ref{DSF}). Assume $\chi_0$ is unramified at every finite place. Then we have
	$$ DS(\Psi) = \frac{1}{\zeta_{\F}^*} \Res_{s=\frac{1}{2}} \left( \Dis_{\F}^{-2} \frac{\zeta_{\F} \left( \frac{1}{2}+s \right)^3 \zeta_{\F} \left( \frac{1}{2}-s \right)^3}{\zeta_{\F}(1+2s) \zeta_{\F}(1-2s)} \prod_{v \mid \infty} h_v(s,0) \right) =: \Res_{s=\frac{1}{2}} DS(s, \Psi), $$
	where the factor $\Dis_{\F}^{-2}$ is the contribution of the places $\vp \in S-S_{\infty}$. Recall the formula \eqref{ConcWtR}
	$$ h_v(s,0) = \sqrt{\pi} \frac{\cos(\pi s)}{2 \Delta_v} \left\{ \exp \left( \frac{(s-iT_v)^2}{2\Delta_v^2} + i \frac{\pi}{2} s \right) + \exp \left( \frac{(s+iT_v)^2}{2\Delta_v^2} - i \frac{\pi}{2} s \right) \right\}^2, $$
	which vanishes at $s=1/2$ to order one. Hence the order of vanishing of $DS(s, \Psi)$ is $-3+2(r-1)+r = 3r-5$, where $r=[\F:\Q]$. This is $\geq 1$ if $\F \neq \Q$. Thus by (\ref{DSF}) $DS(\Psi) = 0$ if $\F \neq \Q$. Assume $\F=\Q$, then $r=1$. For any $A > 1$ and any integer $n \geq 0$, it is easy to see
	$$ \left. \frac{\ud^n}{\ud s^n} \right|_{s = \frac{1}{2}} h_v \left( s,0 \right) \ll_{n,A} (1+\norm[T_v])^{-A}. $$
	The desired bound of $\norm[DS(\Psi)]$ follows readily from the above bounds via (\ref{DSF}).
\end{proof}

	To bound $DG(\Psi)$, we recall its formula (\ref{DGF})
\begin{multline*} 
	DG(\Psi) = \sum_{\pm} \pm \Res_{s=\pm \frac{1}{2}} \left( \Dis_{\F}^{-2} \zeta_{\F} \left( \frac{1}{2}+s \right)^2 \zeta_{\F} \left( \frac{1}{2}-s \right)^2 \cdot \prod_{v \mid \infty} \widetilde{h}_v \left( s,0 \right) \prod_{\vp \in S-S_{\infty}} \widetilde{H}_{\vp}\left( s,0 \right) \right) \\
	=: \sum_{\pm} \pm \Res_{s=\pm \frac{1}{2}} DG(s, \Psi), 
\end{multline*}
	$$ \text{where} \quad \widetilde{H}_{\vp}\left( s,0 \right) := \widetilde{h}_{\vp}\left( s,0 \right) \zeta_{\vp} \left( \frac{1}{2}+s \right)^{-2} \zeta_{\vp} \left( \frac{1}{2}-s \right)^{-2}. $$
	
\noindent The computation of $\widetilde{h}_{\vp}$ is given in \S \ref{DGAuxNA}, replacing $\chi$ there by $\norm_{\vp}^s$. Writing $n_{\vp} = \cond(\chi_{0,\vp}), q_{\vp}=\Nr(\vp)$, and assuming $n_{\vp} \geq 1$ (otherwise $\widetilde{h}_{\vp}(\cdot)=1$ for $n_{\vp}=0$), we have by (\ref{DecompBis})
	$$ \widetilde{h}_{\vp}\left( s,0 \right) = \zeta_{\vp}(1) q_{\vp}^{-n_{\vp}} \left\{ \widetilde{h}^{0,0}(\chi_{0,\vp} \mid s) + \sum_{\substack{k,\ell \geq n_{\vp}-1 \\ (k,\ell) \neq (0,0)}} \left( \widetilde{h}^{k,\ell}(\chi_{0,\vp} \mid s) + \widetilde{h}^{k,\ell}(\chi_{0,\vp} \mid -s) \right) \right\}, $$
	where $\widetilde{h}^{k,\ell}$ are given by:
\begin{itemize}
	\item[(1)] If $\min(k,\ell) \geq n_{\vp}$, then $\widetilde{h}^{k,\ell}(\chi_{0,\vp} \mid s) = \zeta_{\vp}(1)^{-2} q_{\vp}^{(k+\ell)(s-\frac{1}{2})}$;
	\item[(2)] If $\min(k,\ell)=n_{\vp}-1, \max(k,\ell) \geq n_{\vp}$, then $\widetilde{h}^{k,\ell}(\chi_{0,\vp} \mid s) = -\zeta_{\vp}(1)^{-1} q_{\vp}^{-1} q_{\vp}^{(k+\ell)(s-\frac{1}{2})}$;
	\item[(3)] If $k=\ell=n_{\vp}-1$ and $n_{\vp} \geq 2$, then $\widetilde{h}^{k,\ell}(\chi_{0,\vp} \mid s) = q_{\vp}^{-2} q_{\vp}^{(k+\ell)(s-\frac{1}{2})}$;
	\item[(4)] If $k=\ell=0$ and $n_{\vp} = 1$, then $\widetilde{h}^{k,\ell}(\chi_{0,\vp} \mid s) = q_{\vp}^{-2}(\chi_{0,\vp}(1)+\chi_{0,\vp}(-1))$.
\end{itemize}
	It follows readily that
\begin{align*}
	\widetilde{H}_{\vp}\left( s,0 \right) &= \zeta_{\vp}(1) q_{\vp}^{-n_{\vp}} \left\{ q_{\vp}^{-2n_{\vp}(\frac{1}{2}-s)} \zeta_{\vp}\left( \frac{1}{2}+s \right)^{-2} + q_{\vp}^{-2n_{\vp}(\frac{1}{2}+s)} \zeta_{\vp}\left( \frac{1}{2}-s \right)^{-2} \right\} - \\
	&\quad \frac{2 \zeta_{\vp}(1)^2 q_{\vp}^{-1-n_{\vp}}}{\zeta_{\vp}(\frac{1}{2}-s)\zeta_{\vp}(\frac{1}{2}+s)} \left\{ q_{\vp}^{-2n_{\vp}(\frac{1}{2}-s)} \zeta_{\vp}\left( \frac{1}{2}+s \right)^{-1} + q_{\vp}^{-2n_{\vp}(\frac{1}{2}+s)} \zeta_{\vp}\left( \frac{1}{2}-s \right)^{-1} \right\} + \\
	&\quad \zeta_{\vp}(1)^3 \frac{q_{\vp}^{-2-n_{\vp}}}{\zeta_{\vp}(\frac{1}{2}-s)^2 \zeta_{\vp}(\frac{1}{2}+s)^2} \cdot \left\{ \begin{matrix} (q_{\vp}^{-(n_{\vp}-1)(1-2s)} + q_{\vp}^{-(n_{\vp}-1)(1+2s)}) & \text{if } n_{\vp} \geq 2 \\ (\chi_{0,\vp}(1)+\chi_{0,\vp}(-1)) & \text{if } n_{\vp}=1 \end{matrix} \right. .
\end{align*}

\noindent We deduce the bound for any integer $n \geq 0$
\begin{equation} \label{DGBdLocNA}
	\left. \frac{\partial^n}{\partial s^n} \right|_{s=\pm \frac{1}{2}} \widetilde{H}_{\vp}(s,0) \ll_n q_{\vp}^{-n_{\vp}} \left( \log q_{\vp} \right)^n. 
\end{equation}

\noindent For $v \mid \infty$, to analyze $\widetilde{h}_{v}(s,0)$ at $s = \pm 1/2$, we need to revisit \S \ref{DGAuxA}. With the notation in \S \ref{DGAuxA}, $\widetilde{h}_{v}(s,0)$ is a linear combination of $\widetilde{h}_{1}(\pm s)$, $\widetilde{h}_{2}(\pm s)$ and $\cos(\pi s)^{-1} \widetilde{h}_3(\pm s)$ by Theorem \ref{MTF}. For $h=h_j$ we have for any $0 < c < 1/2$ and $x \in \R$ (see (\ref{ContShift}))

\begin{align}
	\widetilde{h}(ix) &= \int_{(c)} h^*(s) \cdot \frac{\Gamma(s)^3 \Gamma \left( \frac{1}{2}-ix-s \right)}{\Gamma(2s) \Gamma \left( \frac{1}{2}-ix \right)} \frac{\ud s}{2\pi i} \nonumber \\
	&= h^*\left( \frac{1}{2}-ix \right) \frac{\Gamma \left( \frac{1}{2}-ix \right)^2}{\Gamma \left( 1-2ix \right)} + \int_{(c+1)} h^*(s) \frac{\Gamma(s)^3 \Gamma \left( \frac{1}{2}-ix-s \right)}{\Gamma(2s) \Gamma \left( \frac{1}{2}-ix \right)} \frac{\ud s}{2\pi i}, \label{ModDualWtF}
\end{align}
	where we recall the definition of
	$$ h^*(s) = \Gamma(2s) \int_{-\infty}^{\infty} \frac{\Gamma \left( \frac{1}{2}+i\tau-s \right)}{\Gamma \left( \frac{1}{2}+i\tau+s \right)} h(i\tau) \frac{\ud \tau}{2\pi}. $$
	$h^*(s)$ is holomorphic for $\Re s > 0$ (see Lemma \ref{AuxTrans}). Hence we are reduced to studying $h^*(s)$ near $s=0$ and the above integrals defining $\widetilde{h}(ix)$ near $ix=\pm 1/2$.
	
\begin{lemma}
	(1) $h^*(s)$ has a simple pole at $s=0$. Introducing the Laurent expansion 
	$$ h^*(s) = \frac{(h^*)^{(-1)}(0)}{s} + \sum_{n=0}^{\infty} \frac{(h^*)^{(n)}(0)}{n!} s^n, $$
	we have for any integer $n \geq 0$ and any $\epsilon > 0$ a bound
	$$ (h^*)^{(-1)}(0) \ll T \Delta, \quad (h^*)^{(n)}(0) \ll_{\epsilon,n} T^{1+\epsilon} \Delta. $$
	
\noindent (2) Write $\Gamma^{(n)}(s)$ for the $n$-th derivative of the Gamma function. For any integer $n \geq 0$, we have
	$$ \int_{(c)} h^*(s) \cdot \frac{\Gamma(s)^3 \Gamma^{(n)} \left( 1-s \right)}{\Gamma(2s)} \frac{\ud s}{2\pi i} \ll_{\epsilon,n} T^{1+\epsilon} \Delta^{-2c}, $$
	$$ \int_{(c+1)} h^*(s) \cdot \frac{\Gamma(s)^3 \Gamma^{(n)} \left( -s \right)}{\Gamma(2s)} \frac{\ud s}{2\pi i} \ll_{\epsilon,n} T^{1+\epsilon} \Delta^{-2-2c}. $$
\end{lemma}
\begin{proof}
	(1) The integral over $\tau$ in the defining formula of $h^*(s)$ is absolutely convergent near $s=0$, hence is regular at $0$. While $\Gamma(2s)$ has a simple pole at $s=0$, we get the first assertion. By induction on $n$, one easily shows
\begin{align*}
	\frac{\ud^n}{\ud s^n} \left( \frac{\Gamma \left( \frac{1}{2}+i\tau-s \right)}{\Gamma \left( \frac{1}{2}+i\tau+s \right)} \right) &= \frac{\Gamma \left( \frac{1}{2}+i\tau-s \right)}{\Gamma \left( \frac{1}{2}+i\tau+s \right)} \cdot \sum_{\substack{\sum_i a_i k_i + \sum_j b_j \ell_j = n \\ a_i,k_i,b_j,\ell_j \in \Z_{\geq 0}}} \omega(a_i,k_i;b_j,\ell_j) \cdot \\
	&\quad  \prod_{i,j} \left( \frac{\Gamma^{(a_i)} \left( \frac{1}{2}+i\tau-s \right)}{\Gamma \left( \frac{1}{2}+i\tau-s \right)} \right)^{k_i} \left( \frac{\Gamma^{(b_j)} \left( \frac{1}{2}+i\tau+s \right)}{\Gamma \left( \frac{1}{2}+i\tau+s \right)} \right)^{\ell_j},
\end{align*} 
	where $\omega(\cdots)$ are absolute constants depending only on $a_i,k_i,b_j,\ell_j$. Hence the desired bounds follows from the bounds
	$$ \int_{-\infty}^{\infty} h(i\tau) \ud \tau \ll T \Delta, $$
	$$ \int_{-\infty}^{\infty} \prod_i \left( \frac{\Gamma^{(a_i)} \left( \frac{1}{2}+i\tau \right)}{\Gamma \left( \frac{1}{2}+i\tau \right)} \right)^{k_i} h(i\tau) \ud \tau \ll_{\epsilon,n} T^{1+\epsilon} \Delta, $$
	where $\sideset{}{_i} \sum a_i k_i \leq n$. In fact, for $h=h_{1,2}$, $h(i\tau)$ is concentrated around $\tau = \pm T$ in an interval of size $O(\Delta)$, with absolute value bounded by $O(T)$, while by Stirling's estimation the factor of Gamma functions is bounded by $(1+\norm[\tau])^{\epsilon}$ for any $\epsilon > 0$. Hence the above bounds follow readily.
	
\noindent (2) These bounds are direct consequences of Lemma \ref{RefAuxTrans} (2).
\end{proof}

\begin{corollary} \label{DGBdLocA}
	$\widetilde{h}(s)$ is regular at $s=-1/2$, and has a double pole at $s=1/2$. Introducing the Laurent expansion for $a=\pm 1/2$
	$$ \widetilde{h}(s) = \frac{\widetilde{h}^{(-2)}(a)}{(s-a)^2} + \frac{\widetilde{h}^{(-1)}(a)}{s-a} + \sum_{n=0}^{\infty} \frac{\widetilde{h}^{(n)}(a)}{n!}(s-a)^n, $$
	we have the bounds for any $n \geq -2$
	$$ \widetilde{h}^{(n)} \left( \pm \frac{1}{2} \right) \ll_{\epsilon, n} T^{1+\epsilon}. $$
\end{corollary}
\begin{proof}
	This follows from the lemma directly via (\ref{ModDualWtF}).
\end{proof}

\begin{lemma} \label{DGBd}
	(1) For our choice of weight function given in \eqref{ConcWtR}, $\widetilde{h}_v(s,0)$ has possible poles at $s=\pm 1/2$ of order $4$. 
	
\noindent (2)We have the bound for any $\epsilon > 0$
	$$ DG(\Psi) \ll_{\epsilon} \Cond_{\infty}(\chi_0)^{1+\epsilon} \Cond_{\fin}(\chi_0)^{-1+\epsilon}. $$
\end{lemma}
\begin{proof}
	By Corollary \ref{DGBdLocA}, we see that $\widetilde{h}_{v}(s,0)$ has possible poles at $s=\pm 1/2$ of order $2$. Hence the global weight $DG(s, \Psi)$ has possible poles at $s=\pm 1/2$ of order $2-2(r-1)+2r = 4$, proving (1). The bound in (2) is deduced via (\ref{DGF}) from the local bounds of the coefficients in the Laurent expansion of local weights in (\ref{DGBdLocNA}) and Corollary \ref{DGBdLocA}.
\end{proof}

\section{Uniform Bound of A Generalized Hypergeometric Special Value}
\label{TechKerBd}

	\subsection{Easy Cases}

	In this section, we prove Proposition \ref{BasisTransProp} (1), which we reformulate as follows.

\begin{theorem}\label{Thm:K main estimate}
For any real $x,\tau$ we have
\begin{equation}\label{K main estimate}
K(x,\tau) = \Gamma \left( \frac{1}{2}-i x \right) \GenHyGI{3}{2}{\frac{1}{2}+i\tau, \frac{1}{2}+i\tau, \frac{1}{2}+i\tau}{1-i x + i \tau, 1+2i \tau}{1} \ll \frac{1}{1+|\tau|}.
\end{equation}
\end{theorem}

\noindent For the sake of simplicity, we assume that $x,\tau>0$ and consider all four cases: $\epsilon_1,\epsilon_2=\pm1$ with

\begin{equation}\label{K Mellin}
K(\epsilon_1x,\epsilon_2\tau)=\frac{1}{2\pi i}\int_{(c)}
\frac{\Gamma^3(s)\Gamma(1/2-i\epsilon_1x-s)\Gamma(1/2+i\epsilon_2\tau-s)}{\Gamma(1/2+i\epsilon_2\tau+s)\Gamma(1/2-i\epsilon_1x)}\ud s,
\end{equation}
	
\noindent where $0<c<\frac{1}{2}.$ Let $s=c+iy$. Using the Stirling formula for estimating the Gamma-functions, we can estimate \eqref{K Mellin} by the following integral
\begin{equation}\label{K Mellin est1}
K(\epsilon_1x,\epsilon_2\tau)\ll
\int_{-\infty}^{\infty}g_{\epsilon_1,\epsilon_2}(c,y)e^{\frac{\pi}{2}f_{\epsilon_1,\epsilon_2}(y)}\ud y, \quad \text{where}
\end{equation}
$$ f_{\epsilon_1,\epsilon_2}(y)=x-3|y|-|\epsilon_1x+y|-|\epsilon_2\tau-y|+|\epsilon_2\tau+y|. $$
Now we are interested in proving the inequality $f(y)\le0$  and in finding the regions where $f(y)$ is close to zero. Note that the contribution of those $y$ such that 
	$$ f_{\epsilon_1,\epsilon_2}(y)\ll -\max \left( \log^2(1+x), \log^2(1+\tau) \right) $$
is bounded by  $\ll_A (1+x)^{-A}(1+\tau)^{-A}$ for any $A \gg 1$. Moreover, by the relation of symmetry
$$ f_{-1,1}(y)=f_{1,-1}(-y),\quad f_{-1,-1}(y)=f_{1,1}(-y), $$
it is enough to consider the cases $\epsilon_1=-1,\epsilon_2=1$ and $\epsilon_1=-1,\epsilon_2=-1$.

\begin{lemma}\label{lemma K++K--}
For any positive $x,\tau$ we have
\begin{equation}\label{K++and--estimate}
|K(x,\tau)|+|K(-x,-\tau)|\ll\frac{1}{1+\tau}.
\end{equation}
\end{lemma}
\begin{proof}
Note that $f_{-1,-1}(y)<0$ for $y\neq0$, and that $f_{-1,-1}(y)$ is close to zero only if $|y|$ is small. Therefore, the  product of Gamma functions in \eqref{K Mellin} is exponentially small unless
$$y \ll y_0:=\max\left(\log^2(1+x),\log^2(1+\tau)\right).$$
The integrand in \eqref{K Mellin} has poles at
$$ s_{x}(j)=\frac{1}{2}-i\epsilon_1x+j,\quad s_{\tau}(j)=\frac{1}{2}+i\epsilon_2\tau+j,\quad j=0,1,2,\ldots $$
Note that if  $x,\tau<y_0$ then $x,\tau\ll1$ and estimating \eqref{K Mellin est1} trivially we obtain \eqref{K++and--estimate}.

\noindent Consider two cases: $\tau>x>y_0$ and $x>\tau>y_0.$  We change the contour of integration in  \eqref{K Mellin} to
\begin{equation}\label{contour gamma}
\gamma=\gamma_1\cup\gamma_2\cup\gamma_3\cup\gamma_4\cup\gamma_5, \quad \text{where}
\end{equation}
$$ \gamma_1=(c-i\infty,c-iy_0),\,\gamma_2=(c-iy_0,A-iy_0),\,\gamma_3=(A-iy_0,A+iy_0), $$
$$ \gamma_4=(A+iy_0,c+iy_0),\,\gamma_5=(c+iy_0,c+i\infty). $$
While doing this we do not cross any poles at $s_{x}(j)$, $s_{\tau}(j)$.
Since the  product of Gamma functions in \eqref{K Mellin} is exponentially small if
$|y|\gg y_0$,  the integrals over $\gamma_1,\gamma_2,\gamma_4,\gamma_5$ are negligible. Therefore,
\begin{equation}\label{K-1-1estimate0}
K(-x,-\tau) \ll \frac{1}{(x\tau)^B} + \int_{-y_0}^{y_0}
\left|\frac{\Gamma^3(A+iy)\Gamma(1/2-A+i(x-y))\Gamma(1/2-A-i(\tau+y))}{\Gamma(1/2+A-i(\tau-y))\Gamma(1/2+ix)}\right|\ud y.
\end{equation}
Using the Stirling formula to estimate the Gamma functions in  \eqref{K-1-1estimate0}, we obtain
\begin{equation}\label{K-1-1estimate1}
K(-x,-\tau)\ll\frac{1}{(x\tau^2)^A},
\end{equation}
where $A>1$ is an arbitrary constant.

\noindent Consider the case $\tau>y_0>x$. In this situation, we cross the poles at $s_{x}(j)$ while changing the  contour $\Re{s}=c$ in \eqref{K Mellin} to the contour \eqref{contour gamma}. The integral over the contour \eqref{contour gamma} can be estimated as before, hence
\begin{equation}\label{K-1-1estimate2}
K(-x,-\tau)=
\sum_{j=0}^{[A]}\frac{(-1)^j}{j!}
\frac{\Gamma^3(1/2-i\epsilon_1x+j)\Gamma(i\epsilon_2\tau+i\epsilon_1x-j)}
{\Gamma(1+i\epsilon_2\tau-i\epsilon_1x+j)\Gamma(1/2-i\epsilon_1x)}+O\left(\frac{1}{\tau^A}\right)\ll\frac{1}{\tau}.
\end{equation}

\noindent Consider the case $x>y_0>\tau$. In the same way (by computing residues at the poles $s_{\tau}(j)$), we obtain
\begin{equation}\label{K-1-1estimate3}
K(-x,-\tau)=
\sum_{j=0}^{[A]}\frac{(-1)^j}{j!}
\frac{\Gamma^3(1/2+i\epsilon_2\tau+j)\Gamma(-i\epsilon_1x-i\epsilon_2\tau-j)}
{\Gamma(1+2i\epsilon_2\tau+j)\Gamma(1/2-i\epsilon_1x)}+O\left(\frac{1}{x^A}\right)\ll\frac{1}{(\tau x)^{1/2}}.
\end{equation}
Finally combining \eqref{K-1-1estimate1}, \eqref{K-1-1estimate2} and \eqref{K-1-1estimate3}, we prove  \eqref{K++and--estimate}.
\end{proof}

Consider the case $\epsilon_1=-1,\epsilon_2=1$ in \eqref{K Mellin}, \eqref{K Mellin est1}. We have $f_{-1,1}(y)\le0$ and $f_{-1,1}(y)=0$ for $0\le y\le \min(x,\tau)$.
Note that arguing as before, it is not possible to obtain a uniform bound of the same strength as \eqref{K main estimate}.
For example, for $\tau\sim x$ one can prove only the estimate
$$ K(-x,\tau)\ll\frac{1}{\tau^{1/2}}.$$
Nevertheless, it is possible to prove an estimate of size $1/\tau$ in some ranges of $x$ and $\tau$.

\begin{lemma}\label{lemma big and small tau}
For $\tau\ll x^{2/3-\varepsilon}$ or $\tau\gg x^{3/2}$ we have
$$ |K(-x,\tau)|+|K(x,-\tau)|\ll\frac{1}{1+\tau}. $$
\end{lemma}
\begin{proof}
Suppose first that
$\tau\ll x^{2/3-\varepsilon}.$ Moving the contour of integration in  \eqref{K Mellin} to the line $\Re{s}=\sigma=1+\delta$, we cross two poles getting
\begin{multline}\label{K-1+1Mellin2}
K(-x,\tau)=
\frac{\Gamma^2(1/2+ix)\Gamma(i\tau-ix)}{\Gamma(1+i\tau+ix)}+
\frac{\Gamma^3(1/2+i\tau)\Gamma(ix-i\tau)}
{\Gamma(1+2i\tau)\Gamma(1/2+ix)}+\\
\frac{1}{2\pi i}\int_{(\sigma)}
\frac{\Gamma^3(s)\Gamma(1/2+ix-s)\Gamma(1/2+i\tau-s)}{\Gamma(1/2+i\tau+s)\Gamma(1/2+ix)}\ud s.
\end{multline}
Using the Stirling formula to estimate the Gamma functions  in \eqref{K-1+1Mellin2} and taking into account that
$\tau\ll x^{2/3-\varepsilon}$, we infer
\begin{multline*}
K(-x,\tau)\ll\frac{1}{\sqrt{x\tau}}
+\int_{0}^{\tau}\frac{(1+y)^{3\sigma-3/2}}{(x\tau)^{\sigma}(1+\tau-y)^{\sigma}}\ud y\ll\\
\frac{1}{\sqrt{x\tau}}+\int_{0}^{\tau/2}\frac{y^{3\sigma-3/2}}{(x\tau^2)^{\sigma}}\ud y
+\int_{\tau/2}^{\tau}\frac{\tau^{3\sigma-3/2}}{(x\tau)^{\sigma}(1+\tau-y)^{\sigma}}\ud y\ll\\
\frac{1}{\sqrt{x\tau}}+\frac{\tau^{\sigma-1/2}}{x^{\sigma}}+\frac{\tau^{2\sigma-3/2}}{x^{\sigma}}\ll
\frac{1}{\sqrt{x\tau}}+\frac{\tau^{1/2+2\delta}}{x^{1+\delta}}\ll\frac{1}{\tau}.
\end{multline*}

\noindent Suppose then that $\tau\gg x^{3/2}$. Using the Stirling formula similarly, we obtain
\begin{multline*}
K(-x,\tau)\ll\frac{1}{\tau}
+\int_{0}^{x}\frac{(1+y)^{3\sigma-3/2}}{\tau^{2\sigma}(1+x-y)^{\sigma}}\ud y\ll\\
\frac{1}{\tau}+\int_{0}^{x/2}\frac{y^{3\sigma-3/2}}{(x\tau^2)^{\sigma}}\ud y
+\int_{x/2}^{x}\frac{x^{3\sigma-3/2}}{\tau^{2\sigma}(1+x-y)^{\sigma}}\ud y\ll\\
\frac{1}{\tau}+\frac{x^{2\sigma-1/2}}{\tau^{2\sigma}}+\frac{x^{3\sigma-3/2}}{\tau^{2\sigma}}
\ll
\frac{1}{\tau}+\frac{x^{3/2+3\delta}}{\tau^{2+2\delta}}
\ll\frac{1}{\tau}.
\end{multline*}
\end{proof}

	\subsection{Difficult Case: Approximation}
	
	We are left to consider, for $K(-x,\tau)$, the interval
\begin{equation}\label{x23<tau<x32}
x^{2/3-\varepsilon}\ll\tau\ll x^{3/2}.
\end{equation}
In this range, we apply the following integral representation
\begin{equation}\label{K to 2F1}
K(-x,\tau)=
\int_{0}^{1}y^{-1/2+i\tau}(1-y)^{-1/2+ix}
\HyGI \left(\frac{1}{2}+i\tau,\frac{1}{2}+i\tau;1+2i\tau; y\right)\ud y.
\end{equation}

\begin{remark}
Using \cite[(15.6.1)]{OLBC10} we obtain
\begin{align} 
\left| \HyGI \left(\frac{1}{2}+i\tau,\frac{1}{2}+i\tau;1+2i\tau; y\right)\right| &= \left|
\int_{0}^{1}\frac{t^{-1/2+i\tau}(1-t)^{-1/2-i\tau}}{(1-yt)^{1/2+i\tau}}\ud t
\right| \label{2F1 abs estimate} \\
&\le \int_{0}^{1}\frac{t^{-1/2}(1-t)^{-1/2}}{(1-yt)^{1/2}}\ud t=
\HyGI \left(\frac{1}{2},\frac{1}{2};1; y\right). \nonumber
\end{align}
Estimating \eqref{K to 2F1} by absolute value using \eqref{2F1 abs estimate}, we infer that
\begin{equation*}
\left|K(-x,\tau)\right|\le
\int_{0}^{1}y^{-1/2}(1-y)^{-1/2}
\HyGI \left(\frac{1}{2},\frac{1}{2};1; y\right)\ud y=\\
\Gamma(1/2){}_3\mathrm{I}_{2}\left(\frac{1}{2},\frac{1}{2},\frac{1}{2};1, 1; 1\right)\ll1.
\end{equation*}
\end{remark}

\noindent We will apply the following lemma to perform integration by parts in \eqref{K to 2F1}.
\begin{lemma}
The function
\begin{equation}\label{T def}
T_{\tau}(y)=
y^{1/2+i\tau}(1-y)^{1/2}
\HyGI \left(\frac{1}{2}+i\tau,\frac{1}{2}+i\tau;1+2i\tau; y\right)
\end{equation}
satisfies the differential equation
\begin{equation}\label{T difeq}
T''_{\tau}(y)-\left(\tau^2a(y)+b(y)\right)T_{\tau}(y)=0, \quad \text{where}
\end{equation}
\begin{equation}\label{a,b def}
a(y)=\frac{-1}{y^2(1-y)},\quad b(y)=-\frac{1}{4y^2(1-y)^2}+\frac{1}{4y(1-y)}.
\end{equation}
\end{lemma}
\begin{proof}
This follows from \cite[Eqs. (8)-(9),  p. 96]{Er53}.
\end{proof}

	Using \eqref{T def} one can rewrite \eqref{K to 2F1} as
\begin{equation}\label{K to T}
K(-x,\tau)=
\int_{0}^{1}T_{\tau}(y)(1-y)^{-1+ix}\frac{\ud y}{y}.
\end{equation}
The main idea of estimating the integral in \eqref{K to 2F1}  is to approximate  the hypergeometric function under the integral by some elementary function. Unfortunately, we do not know a uniform asymptotic formula for \eqref{T def} valid in the range $0<y<1.$
To overcome this difficulty, we split the integral \eqref{K to T} into two parts  using the following smooth partition of unity:
\begin{equation*}
\chf_0(y)+\chf_1(y)=1, \quad\hbox{for}\quad 0<y<1,
\end{equation*}
\begin{equation*}
\chf_0(y)=1, \quad\hbox{for}\quad 0<y<1-2\delta,\quad
\chf_0(y)=0, \quad\hbox{for}\quad 1-\delta<y<1,
\end{equation*}
\begin{equation*}
\chf_1(y)=0, \quad\hbox{for}\quad 0<y<1-2\delta,\quad
\chf_1(y)=1, \quad\hbox{for}\quad 1-\delta<y<1,
\end{equation*}
\begin{equation*}
\frac{\partial^n}{\partial y^n}\chf_1(y),
\frac{\partial^n}{\partial y^n}\chf_0(y)\ll\frac{1}{\delta^n}, \quad\hbox{for}\quad 1-2\delta<y<1-\delta,
\end{equation*}
where  $\delta=\delta(\tau)\gg\tau^{-2+\epsilon}$ will be chosen later.
For $j=0,1$ let
\begin{equation}\label{Kj def}
K_{j}(-x,\tau)=
\int_{0}^{1}\frac{\chf_j(y)}{y}T_{\tau}(y)(1-y)^{-1+ix}\ud y.
\end{equation}
Therefore, we get a decomposition
\begin{equation}\label{K=K0+K1}
K(-x,\tau)=K_0(-x,\tau)+K_1(-x,\tau).
\end{equation}

	First, we estimate $K_{1}(-x,\tau)$.
\begin{lemma}
The following estimate holds
\begin{equation*}
K_{1}(-x,\tau)\ll\frac{(1+\tau^2\delta)\sqrt{\delta}}{x^2}.
\end{equation*}
\end{lemma}
\begin{proof}
Integrating by parts three times in  \eqref{Kj def} we obtain
\begin{equation*}
K_{1}(-x,\tau)\ll\frac{1}{x^3}
\left|\int_{0}^{1}\left(\frac{\chf_1(y)}{y}T_{\tau}(y)\right)'''(1-y)^{2+ix}\ud y\right|.
\end{equation*}
For the sake of simplicity, we denote $\frac{\chf_1(y)}{y}=:h(y)$. Applying \eqref{T difeq} we find that
\begin{multline*}
\left(h(y)T_{\tau}(y)\right)'''=
\Bigl(3h''(y)+h(y)(\tau^2a(y)+b(y))\Bigr)T'_{\tau}(y)+\\
\Bigl(h'''(y)+3h'(y)(\tau^2a(y)+b(y))+h(y)(\tau^2a'(y)+b'(y))\Bigr)T_{\tau}(y).
\end{multline*}
Therefore,
\begin{align}
K_{1}(-x,\tau) &\ll \frac{1}{x^3}
\left|\int_{0}^{1}\left(3h''(y)+h(y)(\tau^2a(y)+b(y))\right)T'_{\tau}(y)(1-y)^{2+ix}\ud y\right|+ \label{K1 est2} \\
&\frac{1}{x^3}\int_{0}^{1}|T_{\tau}(y)|(1-y)^{2} \times \left|h'''(y)+h'(y)(\tau^2a(y)+b(y))+h(y)(\tau^2a'(y)+b'(y))\right|\ud y. \nonumber
\end{align}
In the first integral on the right-hand side of \eqref{K1 est2} we integrate by parts once again showing that
\begin{align*}
K_{1}(-x,\tau) &\ll \frac{1}{x^2}
\int_{0}^{1}\Bigl|h''(y)+h(y)(\tau^2a(y)+b(y))\Bigr|(1-y)\Bigl|T_{\tau}(y)\Bigr|\ud y+ \\
&\quad \frac{1}{x^3}\int_{0}^{1}|T_{\tau}(y)|(1-y)^{2} \times \left|h'''(y)+h'(y)(\tau^2a(y)+b(y))+h(y)(\tau^2a'(y)+b'(y))\right|\ud y. \nonumber
\end{align*}
It follows from \eqref{T def}, \eqref{2F1 abs estimate} and \cite[15.4.21]{OLBC10} that
$$T_{\tau}(y)\ll\sqrt{y(1-y)}\log(1-y).$$
Applying this estimate and using \eqref{a,b def} and $h^{(k)}(y)\ll\delta^{-k}$, we obtain
\begin{equation*}
K_{1}(-x,\tau)\ll\frac{(1+\tau^2\delta)\sqrt{\delta}}{x^2}+\frac{(1+\tau^2\delta)\sqrt{\delta}}{x^3},
\end{equation*}
thus proving the lemma.
\end{proof}

\begin{corollary}
For $x^{2/3-\varepsilon}\ll\tau\ll x^{3/2}$ and $\delta>\tau^{-2}$ the following estimate holds
\begin{equation}\label{K1 estimate1}
K_{1}(-x,\tau)\ll\frac{(\tau^{10/9}\delta)^{3/2}}{\tau}.
\end{equation}
\end{corollary}

	To estimate $K_{0}(-x,\tau)$ we will first replace the hypergeometric function in \eqref{K to 2F1} by an asymptotic formula. To derive it we will use the following result of Farid Khwaja and  Olde Daalhuis \cite{KD13}.
	
\begin{lemma}\label{lem:dalh}
For $z>(e^{1/4}-1)^{-1}$ and $\tau\rightarrow \infty$ we have
\begin{multline}\label{hy:f1/2}
\frac{\Gamma(1/2+2i\tau)}{\Gamma(1+2i\tau)} \HyG \left( \frac{1}{2},\frac{1}{2},1+2i\tau; -z\right)=
\xi^{1/2}U(1/2,1,\lambda \xi)\sum_{j=0}^{n-1}\frac{P_{j}}{\lambda^j}\\+\xi^{3/2}U(1/2,2,\lambda \xi)\sum_{j=0}^{n-1}\frac{Q_{j}}{\lambda^j}+
O\left(\frac{\xi^{1/2}U(1/2,1,\lambda \xi)+\xi^{3/2}U(1/2,2,\lambda \xi)}{\lambda^n}\right),
\end{multline}
where $U(a,b,z)$ is a confluent hypergeometric function (see \cite[Section 13]{OLBC10}) and
\begin{equation*}
\lambda=1/2+2i\tau, \quad \xi=\log(1+1/z), \quad P_0=e^{-\xi/2}, \quad Q_0=\frac{1-e^{-\xi/2}}{\xi}.
\end{equation*}
\end{lemma}
\begin{proof}
See \cite[Eqs. 4.3, 4.5, 5.8, 5.9]{KD13}.
\end{proof}

\begin{corollary}\label{cor:confl}
Assume that $z>(e^{1/4}-1)^{-1}$. For $\tau\rightarrow +\infty$ uniformly for $z$ such that $\tau\xi \gg \tau^{\epsilon}$ the following asymptotic formula holds
\begin{equation}\label{KD result}
\frac{\Gamma(1/2+2i\tau)}{\Gamma(1+2i\tau)}{}_2F_{1}\left( \frac{1}{2},\frac{1}{2},1+2i\tau; -z\right)=
\frac{1}{\sqrt{\lambda}}+
O\left(\frac{1}{\lambda^{3/2}\xi}\right).
\end{equation}
\end{corollary}
\begin{proof}
Applying \eqref{hy:f1/2} with $n=1$ and using the asymptotic formula for the confluent hypergeometric function of large argument (see \cite[Eq. 13.7.3]{OLBC10}) we prove the lemma.
\end{proof}

\noindent We will also use the following result first obtained by Zavorotny \cite{Zav89}.

\begin{lemma} 
There is an $y_1>0$ such that for $\tau \rightarrow \infty$ uniformly for  all $y>y_1$ we have
\begin{multline}\label{Zav lemma}
\HyG \left( \frac{1}{2}+i\tau,\frac{1}{2}+i\tau,1+2i\tau; \frac{-1}{y^2}\right)=(2y)^{2i\tau}e^{-2i\tau\log(y+\sqrt{1+y^2})}\\ \times \left( \frac{y^2}{1+y^2}\right)^{1/4}
\left(1-\frac{1}{8i\tau }\left( 1-\frac{1+2y^2}{2y\sqrt{1+y^2}}\right) \right)+O\left(\frac{1}{y^4\tau^2} \right).
\end{multline}
\end{lemma}
\begin{proof}
This formula is proved in \cite[Lemma 2.4]{Zav89}.
\end{proof}

\noindent Combining \eqref{KD result} and \eqref{Zav lemma} we obtain the following asymptotic formula.

\begin{lemma} 
For $\tau \rightarrow \infty$ uniformly for  all $y$ such that $0<y<1-\tau^{-2+\epsilon}$ we have
\begin{equation}\label{2F1asymptotic}
\HyG \left( \frac{1}{2}+i\tau,\frac{1}{2}+i\tau,1+2i\tau; y\right) = \frac{2^{2i\tau}(1-y)^{-1/4}}{(1+\sqrt{1-y})^{2i\tau}}+O\left(\frac{1}{\tau(1-y)^{3/4}}\right).
\end{equation}
\end{lemma}
\begin{proof}
The case $0<y<(1+y_1^{2})^{-1}$ follows from \eqref{Zav lemma} together with \cite[15.8.1]{OLBC10}.   Consider the case $(1+y_1^{2})^{-1}<y<1-\tau^{-2+\epsilon}$.
According to \cite[(25), p.112]{Er53}
\begin{equation*}
{}_2F_{1}\left(\frac{1}{2},\frac{1}{2};1+i\tau; -z\right)=
\frac{(1+z)^{i\tau}}{(\sqrt{1+z}+\sqrt{z})^{1+2i\tau}}\times\\
{}_2F_{1}\left(\frac{1}{2}+i\tau,\frac{1}{2}+i\tau;1+2i\tau; \frac{4\sqrt{z}\sqrt{1+z}}{(\sqrt{1+z}+\sqrt{z})^2}\right).
\end{equation*}
This can be rewritten as
\begin{equation}\label{2F1transform2}
\GenHyG{2}{1}{\frac{1}{2}+i\tau,\frac{1}{2}+i\tau}{1+2i\tau}{y}=
\frac{2^{2i\tau}}{(1+\sqrt{1-y})^{2i\tau}(1-y)^{1/4}}\times \GenHyG{2}{1}{\frac{1}{2},\frac{1}{2}}{1+i\tau}{ -\frac{(1-\sqrt{1-y})^2}{4\sqrt{1-y}}}.
\end{equation}
For the hypergeometric function on the right-hand side of \eqref{2F1transform2} we apply \eqref{KD result} to prove the lemma.
\end{proof}

	Recall that we are left to estimate \eqref{Kj def} with $j=0$. To do this we substitute \eqref{T def} and \eqref{2F1asymptotic}
to \eqref{Kj def} showing that

\begin{equation}\label{K0 estimate0}
K_{0}(-x,\tau)=\frac{\Gamma^2(1/2+i\tau)2^{2i\tau}}{\Gamma(1+i\tau)}
\int_{0}^{1}\frac{\chf_0(y)(1-y)^{-3/4+ix}y^{-1/2+i\tau}}{(1+\sqrt{1-y})^{2i\tau}}\ud y
+O\left(\frac{1}{\tau^{3/2}\delta^{1/4}}\right).
\end{equation}
Let
\begin{equation}\label{alpha q def}
\alpha=\frac{x}{\tau},\quad q_0(y)=\frac{\chf_0(y)}{y^{1/2}(1-y)^{3/4}},
\end{equation}
\begin{equation}\label{alpha p def}
p_0(\alpha,y)=\alpha\log(1-y)-2\log(1+\sqrt{1-y})+\log y.
\end{equation}
Then \eqref{K0 estimate0} can be rewritten as
\begin{equation}\label{K0 estimate1}
K_{0}(-x,\tau)=\frac{\Gamma^2(1/2+i\tau)2^{2i\tau}}{\Gamma(1+i\tau)}
\int_{0}^{1}q_0(y)e^{i\tau p_0(\alpha,y)}\ud y
+O\left(\frac{1}{\tau^{3/2}\delta^{1/4}}\right).
\end{equation}

\begin{lemma}
For $\tau\sim x$ and $\tau^{-2+\epsilon}\ll\delta\ll\tau^{-2/3-\epsilon}$
\begin{equation}\label{K0 estimate2}
K_{0}(-x,\tau)=\frac{-2\pi i}{\tau(1+4\alpha^2)^{1/4}}e^{i\tau p_0(\alpha,y_{\alpha})}
+O\left(\frac{1}{\tau^{3/2}\delta^{1/4}}\right),
\end{equation}
where
\begin{equation}\label{p0(alpha,yalpha) def}
p_0(\alpha,y_{\alpha})=2\alpha\log\frac{2\alpha}{1+\sqrt{1+4\alpha^2}}-\log\left(2\alpha+\sqrt{1+4\alpha^2}\right).
\end{equation}
\end{lemma}
\begin{proof}
We have
\begin{equation*}
\frac{\partial}{\partial y}p_0(\alpha,y)=\frac{1}{y\sqrt{1-y}}-\frac{\alpha}{1-y}.
\end{equation*}
Therefore, the saddle point (which is a solution of $\frac{\partial}{\partial y}p_0(\alpha,y)=0$) is equal to
\begin{equation}\label{sadpoint}
y_{\alpha}=\frac{-1+\sqrt{1+4\alpha^2}}{2\alpha^2}=\frac{2}{1+\sqrt{1+4\alpha^2}}.
\end{equation}
Note that $\alpha y_{\alpha}=\sqrt{1-y_{\alpha}}$ and that there exists the second saddle point
\begin{equation*}
y_{\alpha}^{(-)}=\frac{-1-\sqrt{1+4\alpha^2}}{2\alpha^2}.
\end{equation*}
Since  $\alpha\sim1$ the saddle point $y_{\alpha}^{(-)}$ is bounded away from the interval of integration in \eqref{K0 estimate1} and the saddle point $y_{\alpha}$ is bounded  away from zero and one. Therefore, we can apply a standard version of the saddle point method getting
\begin{equation}\label{sadpoint approx}
\int_{0}^{1}q_0(y)e^{i\tau p_0(\alpha,y)}\ud y=
\frac{\sqrt{2\pi}e^{\pi i/4}q_0(y_{\alpha})}{\sqrt{\tau p_0''(\alpha,y_{\alpha})}}e^{i\tau p_0(\alpha,y_{\alpha})}
+O\left(\frac{1}{\tau^{3/2}}\right),
\end{equation}
where
\begin{equation}\label{alpha p derivative2}
p_0''(\alpha,y_{\alpha})=\frac{\partial^2}{\partial y^2}p_0(\alpha,y)\Biggl|_{y=y_{\alpha}}=\\
\frac{1}{2y_{\alpha}(1-y_{\alpha})^{3/2}}-\frac{1}{y^2_{\alpha}\sqrt{1-y_{\alpha}}}-\frac{\alpha}{(1-y_{\alpha})^2}=
-\frac{\sqrt{1+4\alpha^2}}{2\alpha^3y_{\alpha}^4}.
\end{equation}
For $\delta\ll \tau^{-2/3-\epsilon}$ in the range \eqref{x23<tau<x32} we have $\chf_0(y_{\alpha})=1$.  Substituting \eqref{alpha q def}, \eqref{alpha p derivative2} to \eqref{sadpoint approx} we have
\begin{equation}\label{sadpoint approx2}
\int_{0}^{1}q_0(y)e^{i\tau p_0(\alpha,y)}\ud y=
\frac{2\sqrt{\pi}e^{-\pi i/4}}{(1+4\alpha^2)^{1/4}\sqrt{\tau}}e^{i\tau p_0(\alpha,y_{\alpha})}
+O\left(\frac{1}{\tau^{3/2}}\right).
\end{equation}
It follows from  the Stirling formula that
\begin{equation}\label{gamma2/gamma}
\frac{\Gamma^2(1/2+i\tau)2^{2i\tau}}{\Gamma(1+i\tau)}=\frac{\sqrt{\pi}}{\sqrt{\tau}}e^{-\pi i/4}
+O\left(\frac{1}{\tau^{3}}\right).
\end{equation}
Substituting \eqref{gamma2/gamma} to  \eqref{K0 estimate1} and applying \eqref{sadpoint approx2} we prove
\eqref{K0 estimate2}.

\noindent Using \eqref{sadpoint}, \eqref{alpha p def}  and the relation $\alpha y_{\alpha}=\sqrt{1-y_{\alpha}}$ we obtain
\begin{equation*}
p_0(\alpha,y_{\alpha})=
2\alpha\log(\alpha y_{\alpha})-\log\frac{(1+\alpha y_{\alpha})^2}{y_{\alpha}}=\\
2\alpha\log\frac{2\alpha}{1+\sqrt{1+4\alpha^2}}-\log\left(2\alpha+\sqrt{1+4\alpha^2}\right),
\end{equation*}
thus proving \eqref{p0(alpha,yalpha) def}.
\end{proof}

\begin{corollary}
For $\tau\sim x$ we have
\begin{equation}\label{K-1+1estimate main1}
K(-x,\tau)=\frac{-2\pi i}{\tau(1+4\alpha^2)^{1/4}}e^{i\tau p_0(\alpha,y_{\alpha})}
+O\left(\frac{1}{\tau^{25/21}}\right).
\end{equation}
\end{corollary}
\begin{proof}
Substituting \eqref{K0 estimate2} and \eqref{K1 estimate1} to \eqref{K=K0+K1} we prove that  for
$\tau\sim x$ and $\tau^{-2}\ll\delta\ll \tau^{-2/3-\epsilon}$ one has
\begin{equation}\label{K-1+1estimate main2}
K(-x,\tau)=\frac{-2\pi i}{\tau(1+4\alpha^2)^{1/4}}e^{i\tau p_0(\alpha,y_{\alpha})}
+O\left(\frac{1}{\tau^{3/2}\delta^{1/4}}+\frac{\tau^{5/3}\delta^{3/2}}{\tau}\right).
\end{equation}
The optimal choice for $\delta$ is $\delta=\tau^{-26/21}$.
\end{proof}

	Numerical computations show that \eqref{K-1+1estimate main1} is valid in the whole range
$x^{2/3-\varepsilon}\ll\tau\ll x^{3/2}$. In order to prove this fact one should use more involved arguments to confirm
\eqref{sadpoint approx}.

\noindent First, we split the integral in  \eqref{K0 estimate0} into two parts.  Let
\begin{equation*}
\chf_{0,0}(y)+\chf_{0,1}(y)=\chf_0(y),
 \quad\hbox{for}\quad 0<y<1,
\end{equation*}
\begin{equation*}
\chf_{0,0}(y)=1 \quad\hbox{for}\quad 0<y<\delta_0,\quad
\chf_{0,0}(y)=0 \quad\hbox{for}\quad 2\delta_0<y<1,
\end{equation*}
\begin{equation}\label{chf01 def1}
\chf_{0,1}(y)=0 \quad\hbox{for}\quad 0<y<\delta_0,\quad 1-\delta<y<1,
\end{equation}
\begin{equation}\label{chf01 def2}
\chf_{0,1}(y)=1 \quad\hbox{for}\quad 2\delta_0<y<1-2\delta.
\end{equation}
Therefore, we can rewrite \eqref{K0 estimate0}  as
\begin{equation}\label{K0 to K01}
K_{0}(-x,\tau)=\frac{\Gamma^2(1/2+i\tau)2^{2i\tau}}{\Gamma(1+i\tau)}\left(K_{0,0}(x,\tau)+K_{0,1}(x,\tau)
\right)
+O\left(\frac{1}{\tau^{3/2}\delta^{1/4}}\right),
\end{equation}
where for $j=0,1$
\begin{equation}\label{K01 def}
K_{0,j}(x,\tau)=
\int_{0}^{1}\frac{\chf_{0,j}(y)(1-y)^{-3/4+ix}y^{-1/2+i\tau}}{(1+\sqrt{1-y})^{2i\tau}}\ud y.
\end{equation}

\begin{lemma}
For
\begin{equation}\label{delta0 estimate}
\delta_0\ll\min\left(\frac{\tau}{x\tau^{\epsilon}},\frac{1}{\tau^{\epsilon}}\right)
\end{equation}
we have
\begin{equation}\label{K00 estimate}
K_{0,0}(x,\tau)\ll\frac{1}{\tau^A}.
\end{equation}
\end{lemma}
\begin{proof}
Integrating by parts we estimate the integral \eqref{K01 def} with $\chf_{0,0}(y)$ in the following way
\begin{multline}\label{K00 estimate0}
K_{0,0}(x,\tau)\ll
\frac{1}{(1+\tau)^a}
\int_{0}^{1}\left|\frac{\partial^a}{\partial y^a}\left(\frac{\chf_{0,0}(y)(1-y)^{-3/4+ix}}{(1+\sqrt{1-y})^{2i\tau}}\right)\right|y^{a-1/2}\ud y\\
\ll\frac{\sqrt{\delta_0}}{(1+\tau)^a}\left(1+(x\delta_0)^a+(\tau\delta_0)^a\right).
\end{multline}
The right-hand side of \eqref{K00 estimate0} is negligible if \eqref{delta0 estimate} is satisfied.
\end{proof}

	\subsection{Difficult Case: Variant of Temme's Method}
	
	Analysis of $K_{0,1}(x,\tau)$ in the ranges $\tau\ll x$ and $\tau\gg x$ are slightly different. If $\tau\gg x$ we make the change of variables $y=(\cosh v)^{-2}$ so that with $\alpha=x/\tau$
\begin{multline}\label{K01 tau>x1}
K_{0,1}(x,\tau)=
2^{3/2}\int_{0}^{\infty}\frac{\chf_{0,1}((\cosh v)^{-2})}{\sqrt{\sinh(2v)}}\frac{(\tanh v)^{2ix}}{e^{2i\tau v}}\ud v=\\
2^{3/2}\int_{0}^{\infty}\frac{\chf_{0,1}((\cosh v)^{-2})}{\sqrt{\sinh(2v)}}
e^{-2i\tau(v-\alpha\log(\tanh v))}\ud v.
\end{multline}

\noindent If $\tau\ll x$ we make the change of variables $y=1-e^{-2v}$ so that with $\beta=\frac{\tau}{4x}$
\begin{multline}\label{K01 tau<x1}
K_{0,1}(x,\tau)=
2\int_{0}^{\infty}\frac{\chf_{0,1}(1-e^{-2v})e^{-v/2}}{\sqrt{1-e^{-2v}}}
\frac{e^{-2ixv}(1-e^{-v})^{i\tau}}{(1+e^{-v})^{i\tau}}\ud v=\\
2^{1/2}\int_{0}^{\infty}\frac{\chf_{0,1}(1-e^{-2v})}{\sqrt{\sinh(v)}}
e^{ix(-2v+x^{-1}\tau\log(\tanh v/2))}\ud v=\\
2^{3/2}\int_{0}^{\infty}\frac{\chf_{0,1}(1-e^{-4v})}{\sqrt{\sinh(2v)}}
e^{-4ix(v-\beta\log(\tanh v))}\ud v.
\end{multline}

\noindent We remark that both integrals in \eqref{K01 tau>x1} and \eqref{K01 tau<x1} have the following shape
\begin{equation}\label{Temme1}
\int_{0}^{a}q(v)^{\lambda-1}e^{-zp(v)}h(v)\ud v.
\end{equation}

\noindent The integrals of such type were intensively investigated by  Temme, see \cite{Te83}, \cite{Te85}, \cite[Ch. 25]{Te15}, \cite{Te21}.  Assuming that $p(v)=v$ and functions $q(v)$ and $h(v)$ satisfy some reasonable assumptions,  Temme
\cite[Theorem 3.2]{Te85} obtained an asymptotic expansion of \eqref{Temme1}. Unfortunately, our choice of $h(v)$ does not allow us to apply the result of Temme directly (one of the reasons for this is that our $h(v)$ is not independent of $\lambda,z$ due to the presence of $\chf_{0,1}(\cdot)$). Therefore, we reproduce the arguments in \cite{Te85}. Also since we are not searching for a full  asymptotic expansion of $K_{0,1}(x,\tau)$ we will estimate the remainder in a slightly different way than Temme. Consider the integral
\begin{equation}\label{I def}
I(q,h,\mu;z):=\int_{0}^{\infty}q(v)^{\lambda-1}e^{-zv}h(v)\ud v=
\int_{0}^{\infty}e^{-z(v-\mu\log q(v))}\frac{h(v)}{q(v)}\ud v,
\end{equation}
where $\mu=\lambda/z.$
To investigate \eqref{I def}, Temme suggested (see \cite[(2.4)]{Te85}) to make the following change of variables
\begin{equation}\label{change var}
v-\mu\log q(v)=t-\mu\log t+A(\mu).
\end{equation}
Properties  of this transformation are discussed in \cite[Sec. 2.3]{Te85}.
Accordingly,
\begin{equation}\label{Ieq2}
I(q,h,\mu;z)=e^{-zA(\mu)}\int_{0}^{\infty}e^{-z(t-\mu\log t)}f(t)\frac{\ud t}{t},\quad
f(t):=t\frac{h(v)}{q(v)}\frac{\ud v}{\ud t}.
\end{equation}
The point $t_0=\mu$ is the saddle-point of the integral \eqref{Ieq2}. If $\mu\gg1$ then one can apply the Laplace method and  when $\mu\rightarrow0$ one can use the Wattson lemma (see discussions in \cite[Sec.1]{Te83}).  But in our case we need a uniform  asymptotic formula valid for all $0<\mu<1$. Such  formula can be obtained using different methods, see \cite[Sec. 3,4,5]{Te83}. We choose the method of \cite[Sec. 5]{Te83}, \cite[Sec. 3.6]{Te85}.  The idea is to express $f(t)$ in \eqref{Ieq2} as
$f(t)=f(\mu)+(f(t)-f(\mu))$ and  then integrate  the second summand by parts.

\begin{lemma}
For $\Re{z}>0$ and $\mu>0$
\begin{equation}\label{Ieq0}
I(q,h,\mu;z)=
e^{-zA(\mu)}\frac{f(\mu)\Gamma(\lambda)}{z^{\lambda}}+\\
\frac{e^{-zA(\mu)}}{z}\int_{0}^{\infty}e^{-z(t-\mu\log t)}\left(\frac{f(t)-f(\mu)}{t-\mu}\right)'\ud t.
\end{equation}
\end{lemma}
\begin{proof}
This follows from the elementary computation
\begin{align*}
I(q,h,\mu;z) &= e^{-zA(\mu)}f(\mu)\int_{0}^{\infty}t^{\lambda-1}e^{-zt}\ud t+
e^{-zA(\mu)}\int_{0}^{\infty}(f(t)-f(\mu))\frac{de^{-z(t-\mu\log t)}}{-z(t-\mu)} \\
&= e^{-zA(\mu)}\frac{f(\mu)\Gamma(\lambda)}{z^{\lambda}}+
\frac{e^{-zA(\mu)}}{z}\int_{0}^{\infty}e^{-z(t-\mu\log t)}\left(\frac{f(t)-f(\mu)}{t-\mu}\right)'\ud t. \nonumber
\end{align*}
\end{proof}

	We turn to the concrete case $q(v)=\tanh v$. The transformation \eqref{change var} becomes
\begin{equation}\label{change var2}
v-\mu\log \tanh v=t-\mu\log t+A(\mu).
\end{equation}
The saddle-point $v_{\mu}$ of the left-hand side of \eqref{change var2} is defined by
\begin{equation}
\sinh(2v_{\mu})=2\mu.
\end{equation}
The transformation \eqref{change var2} maps  (see  \cite[(2.6)]{Te85})
\begin{equation}\label{change var3}
v=0\leftrightarrow t=0,\, v=v_{\mu}\leftrightarrow t=\mu,\, v=+\infty\leftrightarrow t=+\infty,
\end{equation}
and according to \cite[(2.7)]{Te85}
\begin{multline}\label{A(mu) def}
A(\mu)=v_{\mu}-\mu\log \tanh v_{\mu}-\mu+\mu\log\mu=\\
\frac{1}{2}\log(2\mu+\sqrt{1+4\mu^2})-\mu\log2+\mu\log(1+\sqrt{1+4\mu^2})-\mu=\frac{\mu^3}{3}+O(\mu^5).
\end{multline}
Note that the Taylor expansion in \eqref{A(mu) def} coincides with  \cite[(2.12)]{Te85}.
According to  \cite[(2.15)]{Te85}
\begin{equation}\label{dt/dv}
\frac{\ud t}{\ud v}\Biggl|_{v=v_{\mu}}=\sqrt{1-\frac{\mu^2q''(v_{\mu})}{q(v_{\mu})}}=(1+4\mu^2)^{1/4}.
\end{equation}

\noindent Let us choose
\begin{equation*}
q(v)=\tanh v, \,
h_1(v)=\frac{\chf_{0,1}((\cosh v)^{-2})\tanh v}{\sqrt{\sinh(2v)}},
\end{equation*}
\begin{equation*}
h_2(v)=\frac{\chf_{0,1}(1-e^{-4v})\tanh v}{\sqrt{\sinh(2v)}}.
\end{equation*}
Then \eqref{K01 tau>x1} can be rewritten as
\begin{equation}\label{K01 tau>x to I}
K_{0,1}(x,\tau)=2^{3/2}I(q,h_1,\alpha;2i\tau):=2^{3/2}I_1(\alpha;2i\tau)
\end{equation}
and \eqref{K01 tau<x1} can be rewritten as
\begin{equation}\label{K01 tau<x to I}
K_{0,1}(x,\tau)=2^{3/2}I(q,h_2,\beta;4ix):=2^{3/2}I_2(\beta;4ix).
\end{equation}
Applying \eqref{Ieq0} we obtain the following results.

\begin{lemma}
Consider the transformation \eqref{change var2} with $\mu=\alpha$. Let
\begin{equation}\label{f1 def}
f_{1}(t_{\alpha}):=\frac{\chf_{0,1}((\cosh v)^{-2})}{\sqrt{\sinh(2v)}}t_{\alpha}\frac{\ud v}{\ud t_{\alpha}}.
\end{equation}
Then for $\tau\gg x$
\begin{equation}\label{K01 tau>x Temme}
K_{0,1}(x,\tau)=
e^{-2i\tau A(\alpha)}\frac{2^{3/2}f_1(\alpha)\Gamma(2ix)}{(2i\tau)^{2ix}}+\\
O\left(\frac{1}{\tau}\int_{0}^{\infty}\left|\left(\frac{f_1(t_{\alpha})-f_1(\alpha)}{t_{\alpha}-\alpha}\right)'\right|\ud t_{\alpha}\right),
\end{equation}
where for $\delta=o(\alpha^{2})$
\begin{equation}\label{f1 alpha}
f_{1}(\alpha)=\frac{\sqrt{\alpha}}{(1+4\alpha^2)^{1/4}\sqrt{2}}.
\end{equation}
\end{lemma}
\begin{proof}
Formula \eqref{K01 tau>x Temme} follows from \eqref{K01 tau>x to I}, \eqref{Ieq0}.
Since $\sinh(2v_{\alpha})=2\alpha$ one has
\begin{equation*}
\chf_{0,1}\left(\frac{1}{\cosh^{2} v_{\alpha}}\right)=\chf_{0,1}\left(\frac{2}{1+\sqrt{1+4\alpha^2}}\right).
\end{equation*}
Therefore, using \eqref{chf01 def2} and the condition on $\delta$ we find that $\chf_{0,1}((\cosh v_{\alpha})^{-2})=1$.
Now \eqref{f1 alpha} follows from \eqref{f1 def} and \eqref{dt/dv} with $\mu=\alpha$.
\end{proof}

\begin{lemma}\label{lemma K01 tau<x}
Consider the transformation \eqref{change var2} with $\mu=\beta$. Let
\begin{equation}\label{f2 def}
f_{2}(t_{\beta}):=\frac{\chf_{0,1}(1-e^{-4v})}{\sqrt{\sinh(2v)}}t_{\beta}\frac{\ud v}{\ud t_{\beta}}.
\end{equation}
Then for $\tau\ll x$
\begin{equation}\label{K01 tau<x Temme}
K_{0,1}(x,\tau)=
e^{-4ix A(\beta)}\frac{2^{3/2}f_2(\beta)\Gamma(i\tau)}{(4ix)^{i\tau}}+\\
O\left(\frac{1}{x}\int_{0}^{\infty}\left|\left(\frac{f_2(t_{\beta})-f_2(\beta)}{t_{\beta}-\beta}\right)'\right|\ud t_{\beta}\right),
\end{equation}
where for $\delta_0=o(\beta)$
\begin{equation}\label{f2 beta}
f_{2}(\beta)=\frac{\sqrt{\beta}}{(1+4\beta^2)^{1/4}\sqrt{2}}.
\end{equation}
\end{lemma}
\begin{proof}
Formula \eqref{K01 tau<x Temme} follows from \eqref{K01 tau<x to I}, \eqref{Ieq0}.
Since $\sinh(2v_{\beta})=2\beta$ one has
\begin{equation*}
\chf_{0,1}\left(1-e^{-4v_{\beta}}\right)=\chf_{0,1}\left(\frac{4\beta}{2\beta+\sqrt{1+4\beta^2}}\right).
\end{equation*}
Therefore, using \eqref{chf01 def2} and the condition on $\delta_0$ we have $\chf_{0,1}\left(1-e^{-4v_{\beta}}\right)=1$.
Now \eqref{f2 beta} follows from \eqref{f2 def} and \eqref{dt/dv} with $\mu=\beta$.
\end{proof}

	Substituting \eqref{K01 tau>x Temme}  and \eqref{K00 estimate} into \eqref{K0 to K01}, evaluating the main term and estimating the error term in  \eqref{K01 tau>x Temme}, we arrive at the following fine asymptotic formula.

\begin{lemma}\label{lemma K0 tau>x}
For $x\ll\tau\ll x^{3/2}$  and $\delta\ll\tau^{-1}$ we have
\begin{equation}\label{K0 asympt tau>x 0}
K_{0}(-x,\tau)=\frac{-2\pi i}{\tau(1+4\alpha^2)^{1/4}}e^{i\tau p_0(\alpha,y_{\alpha})}+
O\left(\frac{1}{\tau^{3/2}\delta^{1/4}}+\frac{|\log\alpha|+1}{\tau x^{1/2}}\right).
\end{equation}
\end{lemma}
\begin{proof}
	\emph{Main Term}: From Stirling's formula for $x>0$ we obtain
\begin{equation}\label{Stirling2}
\Gamma(2iy)=\frac{\sqrt{\pi}}{\sqrt{y}}e^{-\pi y}e^{i(2y\log(2y)-2y-\pi/4)}+O(y^{-3/2}).
\end{equation}
Therefore, using  \eqref{dt/dv}, \eqref{f1 alpha} (note that the condition $\delta=o(\alpha^{2})$ is satisfied) and  \eqref{Stirling2} with $y=x$ we show that
\begin{equation}\label{I1MT1}
e^{-2i\tau A(\alpha)}\frac{f_1(\alpha)\Gamma(2ix)}{(2i\tau)^{2ix}}=
\frac{\sqrt{\pi}}{\sqrt{2\tau}(1+4\alpha^2)^{1/4}}
e^{-2i\tau A(\alpha)}\\\times
e^{i(2x\log(2x)-2x-\pi/4)}e^{-2ix\log(2\tau)}.
\end{equation}
It follows from  \eqref{I1MT1}, \eqref{K01 tau>x Temme}, \eqref{K0 to K01}, \eqref{gamma2/gamma} and \eqref{K00 estimate} that
\begin{multline*}
K_{0}(-x,\tau)=
\frac{-2\pi i}{\tau(1+4\alpha^2)^{1/4}}
e^{i(-2\tau A(\alpha)+2x\log(2x)-2x-2x\log(2\tau))}+\\+
O\left(\frac{1}{\tau^{3/2}}\int_{0}^{\infty}\left|\left(\frac{f_1(t_{\alpha})-f_1(\alpha)}{t_{\alpha}-\alpha}\right)'\right|\ud t_{\alpha}\right)
+O\left(\frac{1}{\tau^{3/2}\delta^{1/4}}\right).
\end{multline*}
After some transformations we can prove that
\begin{equation*}
-2\tau A(\alpha)+2x\log(2x)-2x-2x\log(2\tau)=\tau p_0(\alpha,y_{\alpha}),
\end{equation*}
where $p_0(\alpha,y)$ is defined by \eqref{p0(alpha,yalpha) def}. Therefore, we obtain the same main term as in \eqref{K0 estimate2}. As a result,
\begin{multline}\label{K0 to K01 3}
K_{0}(-x,\tau)=\frac{-2\pi i}{\tau(1+4\alpha^2)^{1/4}}e^{i\tau p_0(\alpha,y_{\alpha})}+
O\left(\frac{1}{\tau^{3/2}\delta^{1/4}}\right)+\\+
O\left(\frac{1}{\tau^{3/2}}\int_{0}^{\infty}\left|\left(\frac{f_1(t_{\alpha})-f_1(\alpha)}{t_{\alpha}-\alpha}\right)'\right|\ud t_{\alpha}\right).
\end{multline}
We are left to estimate the integral in \eqref{K0 to K01 3}. To this end, we split it into two parts: $|t_{\alpha}-\alpha|< \delta_2$ and $|t_{\alpha}-\alpha|\geq \delta_2$, where the parameter $0<\delta_2\ll\alpha$ will be chosen later.

\noindent \emph{Remainder 1:} Consider first the integral over $|t_{\alpha}-\alpha|\geq \delta_2$. By \eqref{f1 def}, we obviously have
$$ \int_{0}^{\infty}\left|\left(\frac{f_1(t_{\alpha})-f_1(\alpha)}{t_{\alpha}-\alpha}\right)'\right|\ud t_{\alpha} \ll \int_{0}^{\infty}\left|\left(\frac{f_1(t_{\alpha})}{t_{\alpha}-\alpha}\right)'\right|\ud t_{\alpha} + \frac{\sqrt{\alpha}}{\delta_2}. $$
Applying \eqref{change var2} (see \cite[(2.5)]{Te85}) we show that
\begin{equation}\label{dv/dt}
\frac{\ud v}{\ud t_{\alpha}}=\frac{(t_{\alpha}-\alpha)\sinh(2v)}{t_{\alpha}(\sinh(2v)-2\alpha)},
\end{equation}
which is uniformly bounded by \eqref{change var2}. Substituting \eqref{dv/dt} to \eqref{f1 def} we have
\begin{equation*}
\frac{f_{1}(t_{\alpha})}{t_{\alpha}-\alpha}=\frac{\chf_{0,1}((\cosh v)^{-2})\sqrt{\sinh(2v)}}{\sinh(2v)-2\alpha},
\end{equation*}
which is a function of $v$. Changing the variable of integration to $v$ and applying \eqref{dv/dt}, we obtain
\begin{multline}\label{f1-f1 int}
\int_{|t_{\alpha}-\alpha|\geq \delta_2}\left|\left(\frac{f_1(t_{\alpha})}{t_{\alpha}-\alpha}\right)'\right|
\ud t_{\alpha} \ll
\int_{|v-v_{\alpha}|\gg \delta_2}
\frac{|\chf'_{0,1}((\cosh v)^{-2})|\sinh^{3/2} v}{\cosh^{5/2} v|\sinh(2v)-2\alpha|}\ud v+\\
\int_{|v-v_{\alpha}|\gg \delta_2}
\frac{\chf_{0,1}((\cosh v)^{-2})\cosh(2v)}{|\sinh(2v)-2\alpha|\sqrt{\sinh(2v)}}\ud v + \int_{|v-v_{\alpha}|\gg \delta_2}
\frac{\chf_{0,1}((\cosh v)^{-2})\cosh(2v)\sinh^{1/2}(2v)}{(\sinh(2v)-2\alpha)^2}\ud v,
\end{multline}
where $v_{\alpha}$ is defined as $\sinh(2v_{\alpha})=2\alpha$.  Note that  due to the property  \eqref{change var3}  the change of variable moves the point $t_{\alpha}=\alpha$  to the point $\sinh(2v_{\alpha})=2\alpha$. Furthermore, due to the regularity of the transformation the interval  $|t_{\alpha}-\alpha|\geq \delta_2$ becomes $|v-v_{\alpha}|\gg \delta_2$.  It follows from \eqref{chf01 def1}, \eqref{chf01 def2} that
$|\chf'_{0,1}((\cosh v)^{-2})|=0$
unless
\begin{equation*}
v_1:=\arccosh\frac{1}{\sqrt{1-\delta}}<v<\arccosh\frac{1}{\sqrt{1-2\delta}}:=v_2, \text{ or } v_3:=\arccosh\frac{1}{\sqrt{2\delta_0}}<v<\arccosh\frac{1}{\sqrt{\delta_0}}:=v_4.
\end{equation*}
On these intervals $\chf'_{0,1}((\cosh v)^{-2})\ll\delta^{-1}$ or $\chf'_{0,1}((\cosh v)^{-2})\ll\delta_0^{-1}$, respectively. Note that the size of $v_1$  is approximately $\sqrt{\delta}$ and the size of $v_4$ is close to $-\log\sqrt{\delta_0}$. Since
\begin{equation}\label{delta<alpha}
\sqrt{\delta}\ll\frac{1}{\sqrt{\tau}}<\frac{1}{\tau^{1/3}}\ll\frac{x}{\tau}=\alpha \quad \Rightarrow \quad \sinh(2v) \ll \alpha, \forall v_1 < v < v_2,
\end{equation}
$$ \text{and} \quad \sinh(2v) \gg 1 > \alpha, \forall v_3 < v < v_4, $$
we obtain \textcolor{blue}{}
\begin{equation}\label{f1-f1 int1}
\int_{|v-v_{\alpha}|\gg \delta_2}
\frac{|\chf'_{0,1}((\cosh v)^{-2})|\sinh^{3/2} v}{\cosh^{5/2} v|\sinh(2v)-2\alpha|}\ud v\\\ll
\int_{v_1}^{v_2}\frac{\delta^{3/4}\ud v}{\delta\alpha}+
\int_{v_3}^{v_4}\frac{\sinh^{1/2} v}{\cosh^{7/2}v}\frac{\ud v}{\delta_0}
\ll
\frac{\delta^{1/4}}{\alpha}+\delta_0^{1/2}.
\end{equation}
It follows from \eqref{chf01 def1}, \eqref{chf01 def2} that the second integral in \eqref{f1-f1 int} can be estimated as
\begin{multline}\label{f1-f1 int2.1}
\int_{|v-v_{\alpha}|\gg \delta_2}
\frac{\chf_{0,1}((\cosh v)^{-2})\cosh(2v)}{|\sinh(2v)-2\alpha|\sqrt{\sinh(2v)}}\ud v\\\ll
\left(\int_{v_{\alpha}+c_1\delta_2}^{v_4}+\int_{v_1}^{v_{\alpha}-c_2\delta_2}
\right)
\frac{\cosh(2v)\ud v}{|\sinh(2v)-\sinh(2v_{\alpha})|\sqrt{\sinh(2v)}},
\end{multline}
where $c_j$ are some absolute constants.  Splitting the first integral on the right-hand side of \eqref{f1-f1 int2.1} at some fixed point $c_3$, we obtain
\begin{equation}\label{f1-f1 int2.2}
\int_{v_{\alpha}+c_1\delta_2}^{v_4}
\frac{\cosh(2v)\ud v}{|\sinh(2v)-\sinh(2v_{\alpha})|\sqrt{\sinh(2v)}}\\\ll
\int_{v_{\alpha}+c_1\delta_2}^{c_3}\frac{\ud v}{(v-v_{\alpha})\sqrt{v}}+
\int_{c_3}^{v_4}e^{-v}\ud v\ll
\frac{|\log\delta_2|}{\sqrt{\alpha}}+1.
\end{equation}
Splitting the second integral on the right-hand side of \eqref{f1-f1 int2.1} at some point $c_4\alpha$ ($c_4$ is some small fixed absolute constant), we infer
\begin{multline}\label{f1-f1 int2.3}
\int_{v_1}^{v_{\alpha}-c_2\delta_2}\frac{\cosh(2v)\ud v}{|\sinh(2v)-\sinh(2v_{\alpha})|\sqrt{\sinh(2v)}} \ll
\int_{v_1}^{c_4\alpha}\frac{\ud v}{\sinh(2v_{\alpha})\sqrt{v}}+
\int_{c_4\alpha}^{v_{\alpha}-c_2\delta_2}\frac{\ud v}{|v-v_{\alpha}|\sqrt{\alpha}}\\\ll
\frac{1}{\sqrt{\alpha}}+\frac{|\log\delta_2|}{\sqrt{\alpha}}.
\end{multline}
Substituting  \eqref{f1-f1 int2.2} and \eqref{f1-f1 int2.3} to \eqref{f1-f1 int2.1} we obtain
\begin{equation}\label{f1-f1 int2}
\int_{|v-v_{\alpha}|\gg \delta_2}
\frac{\chf_{0,1}((\cosh v)^{-2})\cosh(2v)}{|\sinh(2v)-2\alpha|\sqrt{\sinh(2v)}}\ud v\ll
\frac{|\log\delta_2|}{\sqrt{\alpha}}.
\end{equation}
We argue in the same way to estimate the third integral in \eqref{f1-f1 int}
\begin{multline}\label{f1-f1 int3.1}
\int_{|v-v_{\alpha}|\gg \delta_2}
\frac{\chf_{0,1}((\cosh v)^{-2})\cosh(2v)\sinh^{1/2}(2v)}{(\sinh(2v)-2\alpha)^2}\ud v
\\\ll
\left(\int_{v_{\alpha}+c_1\delta_2}^{v_4}+\int_{v_1}^{v_{\alpha}-c_2\delta_2}
\right)
\frac{\cosh(2v)\sinh^{1/2}(2v)}{(\sinh(2v)-2\alpha)^2}\ud v.
\end{multline}
Splitting the first integral on the right-hand side of \eqref{f1-f1 int3.1} at some fixed point $c_5$, we have, on noting $v_{\alpha} = \log(2\alpha+\sqrt{1+4\alpha^2})/2 \sim \alpha$ for $0 < \alpha < 1$ 
\begin{multline}\label{f1-f1 int3.2}
\int_{v_{\alpha}+c_1\delta_2}^{v_4}
\frac{\cosh(2v)\sinh^{1/2}(2v)}{(\sinh(2v)-2\alpha)^2}\ud v \ll
\int_{v_{\alpha}+c_1\delta_2}^{c_5}\frac{\sqrt{v}\ud v}{(v-v_{\alpha})^2}+
\int_{c_5}^{v_4}e^{-v}\ud v \ll \\
1+\int_{v_{\alpha}+c_1\delta_2}^{c_6v_{\alpha}}\frac{\sqrt{v_{\alpha}}\ud v}{(v-v_{\alpha})^2}+
\int_{c_6v_{\alpha}}^{c_5}\frac{\sqrt{v}\ud v}{v^2}\ll
\frac{\sqrt{\alpha}}{\delta_2}+\frac{1}{\sqrt{\alpha}},
\end{multline}
where $c_6>10$ is some fixed  constant.  Splitting the second integral on the right-hand side of \eqref{f1-f1 int3.1} at some point $c_7\alpha$ ($c_7$ is some small fixed absolute constant), we show that
\begin{equation}\label{f1-f1 int3.3}
\int_{v_1}^{v_{\alpha}-c_2\delta_2}\frac{\cosh(2v)\sinh^{1/2}(2v)}{(\sinh(2v)-2\alpha)^2}\ud v\ll
\int_{v_1}^{c_7\alpha}\frac{\sqrt{v}\ud v}{\alpha^2}+\\
\int_{c_7\alpha}^{v_{\alpha}-c_2\delta_2}\frac{\sqrt{\alpha}\ud v}{|v-v_{\alpha}|^2}\ll
\frac{1}{\sqrt{\alpha}}+\frac{\sqrt{\alpha}}{\delta_2}.
\end{equation}
Substituting  \eqref{f1-f1 int3.2} and \eqref{f1-f1 int3.3} to \eqref{f1-f1 int3.1} yields
\begin{equation}\label{f1-f1 int3}
\int_{|v-v_{\alpha}|\gg \delta_2}
\frac{\chf_{0,1}((\cosh v)^{-2})\cosh(2v)\sinh^{1/2}(2v)}{(\sinh(2v)-2\alpha)^2}\ud v\\\ll
\frac{1}{\sqrt{\alpha}}+\frac{\sqrt{\alpha}}{\delta_2}.
\end{equation}
Substituting  \eqref{f1-f1 int1}, \eqref{f1-f1 int2} and \eqref{f1-f1 int3} to \eqref{f1-f1 int} and using \eqref{delta<alpha}, we infer
\begin{equation}\label{f1-f1 int4}
\int_{|t_{\alpha}-\alpha|\geq \delta_2}\left|\left(\frac{f_1(t_{\alpha})-f_1(\alpha)}{t-\alpha}\right)'\right|\ud t_{\alpha}\ll
\frac{|\log\delta_2|}{\sqrt{\alpha}}+\frac{\sqrt{\alpha}}{\delta_2}.
\end{equation}

\noindent \emph{Remainder 2:} Let us consider the integral over $|t_{\alpha}-\alpha|< \delta_2$. In this case, we represent $f(t)$ using the Taylor expansion at the point $t=\alpha$ showing that
\begin{equation}\label{f1-f1 int5}
\int_{|t_{\alpha}-\alpha|< \delta_2}\left|\left(\frac{f_1(t_{\alpha})-f_1(\alpha)}{t-\alpha}\right)'\right|\ud t_{\alpha}\ll\\
f_1''(\alpha)\delta_2+\delta^2_2\max_{|t_{\alpha}-\alpha|< \delta_2}f'''_{1}(t_{\alpha})\ll
\frac{\delta_2}{\alpha^{3/2}}+\frac{\delta^2_2}{\alpha^{5/2}}.
\end{equation}
To prove the last estimate we evaluate the derivatives of $f_{1}(t_{\alpha})$ using \eqref{f1 def}. Note that in the neighborhood of the point $t_{\alpha}=\alpha$  the derivatives of $\chf_{0,1}((\cosh v)^{-2})$ are equal to zero.

\noindent \emph{Remainder Term:} Combining \eqref{f1-f1 int4}, \eqref{f1-f1 int5} and choosing $\delta_2=\alpha/100$ we obtain
\begin{equation}\label{f1-f1 int6}
\int_{0}^{\infty}\left|\left(\frac{f_1(t_{\alpha})-f_1(\alpha)}{t_{\alpha}-\alpha}\right)'\right|\ud t_{\alpha}\ll
\frac{|\log\alpha|+1}{\sqrt{\alpha}}.
\end{equation}
Substituting \eqref{f1-f1 int6} to \eqref{K0 to K01 3} we prove \eqref{K0 asympt tau>x 0}.
\end{proof}

\noindent Similarly  using Lemma \ref{lemma K01 tau<x} we obtain the next asymptotic formula.
\begin{lemma}\label{lemma K0 tau<x}
For $x^{2/3-\epsilon}\ll\tau\ll x$ and $\delta\ll\tau^{-1}$ we have
\begin{equation}\label{K0 asympt tau<x 0}
K_{0}(-x,\tau)=\frac{-2\pi i}{\tau(1+4\alpha^2)^{1/4}}e^{i\tau p_0(\alpha,y_{\alpha})}+
O\left(\frac{1}{\tau^{3/2}\delta^{1/4}}+\frac{|\log\beta|+1}{\tau x^{1/2}}\right).
\end{equation}
\end{lemma}
\begin{proof}
Let $\delta_0=\frac{\beta}{\tau^{\epsilon}}$. Therefore, \eqref{delta0 estimate} and the condition $\delta_0=o(\beta)$ in Lemma \ref{lemma K01 tau<x} are satisfied. Using Stirling's formula \eqref{Stirling2} with $y=\tau/2$, \eqref{dt/dv} with $\mu=\beta$, \eqref{f2 beta} we show that
\begin{equation}\label{I2MT1}
e^{-4ix A(\beta)}\frac{2^{3/2}f_2(\beta)\Gamma(i\tau)}{(4ix)^{i\tau}}=
\frac{2^{3/2}\sqrt{\pi\beta}}{\sqrt{\tau}(1+4\beta^2)^{1/4}}\\\times
e^{i(-4xA(\beta)+\tau\log(\tau)-\tau-\tau\log(4x)-\pi/4)}.
\end{equation}
Recall that $A(\beta)$ is defined by \eqref{A(mu) def}.  One can prove that 
\begin{equation}\label{Abeta to pyalpha}
-4x A(\beta)+\tau\log(\tau)-\tau-\tau\log(4x)=\tau p_0(\alpha,y_{\alpha}),
\end{equation}
where $p_0(\alpha,y)$ is defined by \eqref{p0(alpha,yalpha) def}. Note that $\beta=\frac{1}{4\alpha}.$
It follows from  \eqref{I2MT1}, \eqref{Abeta to pyalpha}, \eqref{K01 tau<x Temme}, \eqref{K0 to K01}, \eqref{gamma2/gamma} and \eqref{K00 estimate} that
\begin{multline}\label{K0 to K01 4}
K_{0}(-x,\tau)=\frac{-2\pi i}{\tau(1+4\alpha^2)^{1/4}}e^{i\tau p_0(\alpha,y_{\alpha})}+
O\left(\frac{1}{\tau^{3/2}\delta^{1/4}}\right)+\\+
O\left(\frac{1}{\tau^{1/2}x}\int_{0}^{\infty}\left|\left(\frac{f_2(t_{\beta})-f_2(\beta)}{t_{\beta}-\beta}\right)'\right|\ud t_{\beta}\right).
\end{multline}
The integral in \eqref{K0 to K01 4} can be estimated in the same way as the integral in \eqref{K0 to K01 3}. The only difference between them is the argument of the function $\chf_{0,1}(\cdot)$ (compare \eqref{f1 def}, \eqref{f2 def}).  We again split the integral into two parts: $|t_{\beta}-\beta|< \delta_2$ and $|t_{\beta}-\beta|\geq \delta_2$, where $0<\delta_2\ll\beta$ will be chosen later. The analogue of \eqref{f1-f1 int} is
\begin{multline}\label{f2-f2 int}
\int_{|t_{\beta}-\beta|\geq \delta_2}\left|\left(\frac{f_2(t_{\beta})}{t_{\beta}-\beta}\right)'\right|
\ud t_{\beta} \ll
\int_{|v-v_{\beta}|\gg \delta_2}
\frac{|\chf'_{0,1}(1-e^{-4v})|e^{-4v}\sqrt{\sinh 2v}}{|\sinh(2v)-2\beta|}\ud v+ \\
\int_{|v-v_{\beta}|\gg \delta_2}
\frac{\chf_{0,1}(1-e^{-4v})\cosh(2v)}{|\sinh(2v)-2\beta|\sqrt{\sinh(2v)}}\ud v+
\int_{|v-v_{\beta}|\gg \delta_2}
\frac{\chf_{0,1}(1-e^{-4v})\cosh(2v)\sinh^{1/2}(2v)}{(\sinh(2v)-2\beta)^2}\ud v,
\end{multline}
where $v_{\beta}$ is defined as $\sinh(2v_{\beta})=2\beta$. It follows from \eqref{chf01 def1}, \eqref{chf01 def2} that
$|\chf'_{0,1}(1-e^{-4v})|=0$
unless
\begin{equation*}
v_5:=\frac{-\log(1-\delta_0)}{4}<v<\frac{-\log(1-2\delta_0)}{4}:=v_6, \quad \text{or} \quad v_7:=\frac{-\log(2\delta)}{4}<v<\frac{-\log\delta}{4}:=v_8.
\end{equation*}
On these intervals  $\chf'_{0,1}(1-e^{-4v})\ll\delta_0^{-1}$ or $\chf'_{0,1}(1-e^{-4v})\ll\delta^{-1}$, respectively. Note that the size of $v_5$  is approximately $\delta_0/4$.   Since $\delta_0=\frac{\beta}{\tau^{\epsilon}}$ we have
\begin{equation*}
\int_{|v-v_{\beta}|\gg \delta_2}
\frac{|\chf'_{0,1}(1-e^{-4v})|e^{-4v}\sqrt{\sinh 2v}}{|\sinh(2v)-2\beta|}\ud v\ll
\int_{v_5}^{v_6}\frac{\sqrt{\delta_0}}{\delta_0\beta}\ud v+
\int_{v_7}^{v_8}\frac{e^{-v}}{\delta}\ud v\\\ll\frac{\sqrt{\delta_0}}{\beta}+\delta^{1/4}\ll\frac{1}{\tau^{\epsilon}\sqrt{\beta}}
\end{equation*}
The second and the third integrals in \eqref{f2-f2 int} are integrals over the interval $(v_5,v_8)$ and can be estimated in the same way as the second the third integrals in \eqref{f1-f1 int}. Therefore, we obtain
\begin{equation}\label{f2-f2 int4}
\int_{|t_{\beta}-\beta|\geq \delta_2}\left|\left(\frac{f_2(t_{\beta})-f_2(\beta)}{t_{\beta}-\beta}\right)'\right|
\ud t_{\beta}\ll\frac{|\log\delta_2|}{\sqrt{\beta}}+\frac{\sqrt{\beta}}{\delta_2}.
\end{equation}
Let us consider the integral over $|t_{\beta}-\beta|< \delta_2$. Arguing in the same way as in the proof of the estimate 
\eqref{f1-f1 int5} (again in the neighborhood of the point $t_{\beta}=\beta$  the derivatives of $\chf_{0,1}(1-e^{-4v})$ are equal to zero due to the choice of $\delta_0$), we conclude that
\begin{equation}\label{f2-f2 int5}
\int_{|t_{\beta}-\beta|< \delta_2}\left|\left(\frac{f_1(t_{\beta})-f_1(\beta)}{t-\beta}\right)'\right|\ud t_{\beta}\ll
\frac{\delta_2}{\beta^{3/2}}+\frac{\delta^2_2}{\beta^{5/2}}.
\end{equation}
Combining \eqref{f2-f2 int4}, \eqref{f2-f2 int5} and choosing $\delta_2=\beta/100$ we obtain
\begin{equation}\label{f2-f2 int6}
\int_{0}^{\infty}\left|\left(\frac{f_1(t_{\beta})-f_1(\beta)}{t_{\beta}-\beta}\right)'\right|\ud t_{\beta}\ll
\frac{|\log\beta|+1}{\sqrt{\beta}}.
\end{equation}
Substituting \eqref{f2-f2 int6} to \eqref{K0 to K01 4} we prove \eqref{K0 asympt tau<x 0}.
\end{proof}

	Finally, we prove the following result.
	
\begin{corollary}\label{main corollary}
For $x^{2/3-\epsilon} \ll\tau\ll x^{3/2}$ the following asymptotic formula holds
\begin{equation*}
K(-x,\tau)=\frac{-2\pi i}{\tau(1+4\alpha^2)^{1/4}}e^{i\tau p_0(\alpha,y_{\alpha})}
+O\left(\frac{1}{\tau^{25/21}}\right).
\end{equation*}
\end{corollary}
\begin{proof}
Using \eqref{K=K0+K1}, \eqref{K1 estimate1}, \eqref{K0 asympt tau>x 0}, \eqref{K0 asympt tau<x 0} and \eqref{K-1+1estimate main2},   we prove that  for
$x^{2/3-\epsilon} \ll\tau\ll x^{3/2}$ and $\tau^{-2}\ll\delta\ll \tau^{-1}$
\begin{equation*}
K(-x,\tau)=\frac{-2\pi i}{\tau(1+4\alpha^2)^{1/4}}e^{i\tau p_0(\alpha,y_{\alpha})}
+O\left(\frac{1}{\tau^{3/2}\delta^{1/4}}+\frac{\log\tau}{\tau x^{1/2}}
+\tau^{2/3}\delta^{3/2}\right).
\end{equation*}
The optimal choice for $\delta$ is $\delta=\tau^{-26/21}$.
\end{proof}

\noindent Now Theorem \ref{Thm:K main estimate} follows from Lemmas \ref{lemma K++K--} and \ref{lemma big and small tau} and Corollary \ref{main corollary}.

\section{Fourth Moment Bound}

	\subsection{Period Version of Motohashi's Formula}
	
	We need a period version of our Motohashi formula \eqref{MTF}, which corresponds to considering test functions $\Psi \in \Sch(\Mat_2(\A))$ of the special form
	$$ \Psi \begin{pmatrix} x_1 & x_2 \\ x_3 & x_4 \end{pmatrix} = \Phi_1(x_1,x_2) \Phi_2(x_3,x_4), \quad \Phi_j \in \Sch(\A^2). $$
	To $\Phi_j$, we associate the Godement section and the Eisenstein series
	$$ f_j(s,g) := \norm[\det g]_{\A}^{s+\frac{1}{2}} \int_{\A^{\times}} \Phi_j((0,t)g) \norm[t]_{\A}^{2s+1} \ud^{\times}t, \quad \eis_j(s,g) := \sum_{\gamma \in \gp{B}(\F) \backslash \GL_2(\F)} f_j(s,\gamma g). $$
	We also write $\eis_j(g) = \eis_j(0,g)$. We easily identify the zeta functional as
	$$ \Zeta(s,\chi,\eis_j) = \int_{\A^{\times}} \OFour_2 \Phi_j(x_1,x_2) \chi(x_1x_2) \norm[x_1x_2]_{\A}^s \ud^{\times} x_1 \ud^{\times} x_2. $$
	We introduce the \emph{global} weight functions as
\begin{equation} \label{eq: GlobWt}
	h(\pi)(\Psi) := \frac{L(1/2, \pi)^3}{2 \Lambda_{\F}(2) L(1,\pi,\mathrm{Ad})} \cdot \sideset{}{_{v \mid \infty}} \prod h_v(\pi_v) \cdot \sideset{}{_{\vp \in S}} \prod H_{\vp}(\pi_{\vp}),
\end{equation}
\begin{equation} \label{eq: GlobWtBis}
	h(\chi,i\tau)(\Psi) := \frac{\extnorm{L(1/2+i\tau, \chi)}^6}{2\Lambda_{\F}(2) \extnorm{L(1+2i\tau,\chi^2)}^2} \cdot \sideset{}{_{v \mid \infty}} \prod h_v(\pi(\chi_v,i\tau)) \cdot \sideset{}{_{\vp \in S}} \prod H_{\vp}(\pi(\chi_{\vp},i\tau)),
\end{equation}
\begin{equation} \label{eq: GlobWtMain}
	h(\Psi) := \sideset{}{_{\pi}} \sum h(\pi)(\Psi) + \sideset{}{_{\chi \in \widehat{\R_+ \F^{\times} \backslash \A^{\times}}}} \sum \int_{-\infty}^{\infty} h(\chi,i\tau)(\Psi) \frac{\ud \tau}{2\pi};
\end{equation}
\begin{equation} \label{eq: GlobDWt}
	\widetilde{h}(\chi \norm_{\A}^{i\tau})(\Psi) := \extnorm{L(1/2+i\tau, \chi)}^4 \cdot \sideset{}{_{v \mid \infty}} \prod \widetilde{h}_v(\chi_v \norm_v^{i\tau}) \cdot \sideset{}{_{\vp \in S}} \prod \widetilde{H}_{\vp}(\chi_{\vp} \norm_{\vp}^{i\tau}),
\end{equation}
\begin{equation} \label{eq: GlobDWtMain}
	\widetilde{h}(\Psi) := \frac{1}{\zeta_{\F}^*} \sideset{}{_{\chi \in \widehat{\R_+ \F^{\times} \backslash \A^{\times}}}} \sum \int_{-\infty}^{\infty} \widetilde{h}(\chi \norm_{\A}^{i\tau})(\Psi) \frac{\ud \tau}{2\pi}.
\end{equation}
	
\noindent We get by the global functional equation for Tate's integrals
\begin{align*}
	\widetilde{h}(\chi \norm_{\A}^s)(\Psi) &= \int_{(\A^{\times})^4} \Psi \begin{pmatrix} x_1 & x_2 \\ x_3 & x_4 \end{pmatrix} \chi \left( \frac{x_1x_4}{x_2x_3} \right) \norm[x_1x_4]_{\A}^s \norm[x_2x_3]_{\A}^{1-s} \prod_{i=1}^4 \ud^{\times} x_i \\
	&= \Zeta(s,\chi,\eis_1) \Zeta(1-s,\chi^{-1},\eis_2).
\end{align*}

\begin{lemma}
	Let $\Reis$ be the (linear combination of non-unitary) Eisenstein series such that $\eis_1 \eis_2 - \Reis \in \intL^2([\PGL_2])$. Then we rewrite \eqref{eq: GlobWt} \& \eqref{eq: GlobWtBis} as
	$$ h(\pi)(\Psi) = \sum_{\phi \in \Bas(\pi)} \int_{[\PGL_2]} \left( \eis_1 \eis_2 - \Reis \right)(g) \phi(g) \ud g \cdot \Zeta \left( \frac{1}{2},\phi^{\vee} \right), $$
	$$ h(\chi,i\tau)(\Psi) = \sum_{e \in \Bas(\chi)} \int_{[\PGL_2]} \left( \eis_1 \eis_2 - \Reis \right)(g) \eis(i\tau,e)(g) \ud g \cdot \Zeta \left( \frac{1}{2}, \eis \left( -i\tau, e^{\vee} \right) \right). $$
\end{lemma}
\begin{remark}
	For the existence and uniqueness of $\Reis$, see \cite[Proposition 2.25]{Wu19_TF}.
\end{remark}
\begin{proof}
	Since $\phi \in \Bas(\pi)$ is cuspidal, hence orthogonal to any Eisenstein series and of rapid decay, we have
\begin{align}
	\int_{[\PGL_2]} \left( \eis_1 \eis_2 - \Reis \right)(g) \overline{\phi(g)} \ud g &= \int_{[\PGL_2]} \eis_1(g) \eis_2(g) \overline{\phi(g)} \ud g \nonumber \\
	&= \left. \int_{\gp{Z}(\A)\gp{N}(\A) \backslash \GL_2(\A)} \overline{W_{\phi}(g)} W_1(g) f_2(s,g) \ud g \right|_{s=0}, \label{RSUnfold}
\end{align}
	where $W_j$ resp. $W_{\phi}$ is the Whittaker function of $\eis_j$ resp. $\phi$ with respect to the additive character $\psi$, and we have applied the standard Rankin-Selberg unfolding. The subsequent transformation of integrals is a special case of the reduction of the Rankin-Selberg integrals from the pair $\GL_n \times \GL_n$ to $\GL_n \times \GL_{n-1}$ in \cite[\S 8.2]{J09}. Namely, we have for $\Re s_1, \Re s_2 \gg 1$ and $\Re s_2 - \Re s_1 \gg 1$
\begin{align*}
	&\quad \int_{\gp{Z}(\A)\gp{N}(\A) \backslash \GL_2(\A)} \overline{W_{\phi}(g)} W_1(s_1,g) f_2(s_2,g) \ud g = \int_{\gp{Z}(\A) \backslash \GL_2(\A)} \overline{W_{\phi}(g)} f_1(s_1,wg) f_2(s_2,g) \ud g \\
	&= \int_{\gp{Z}(\A) \backslash \GL_2(\A)} \overline{W_{\phi}(g)} \left( \int_{(\A^{\times})^2} \Psi \left( \begin{pmatrix} t_1 & \\ & t_2 \end{pmatrix} g \right) \norm[t_1]_{\A}^{2s_1+1} \norm[t_2]_{\A}^{2s_2+1} \ud^{\times}t_1 \ud^{\times}t_2 \right) \norm[\det g]_{\A}^{s_1+s_2+1} \ud g \\
	&= \int_{\GL_2(\A)} \overline{W_{\phi}(g)} \left( \int_{\A^{\times}} \Psi \left( \begin{pmatrix} t & \\ & 1 \end{pmatrix} g \right) \norm[t]_{\A}^{2s_1+1} \ud^{\times}t \right) \norm[\det g]_{\A}^{s_1+s_2+1} \ud g \\
	&= \int_{\A^{\times}} \left( \int_{\GL_2(\A)} \Psi(g) \overline{W_{\phi} \left( \begin{pmatrix} t & \\ & 1 \end{pmatrix} g \right)} \norm[\det g]_{\A}^{s_1+s_2+1} \ud g \right) \norm[t]_{\A}^{s_2-s_1} \ud^{\times}t.
\end{align*}
	But we have
	$$ g.\overline{W_{\phi}} = \sum_{e \in \Bas(\pi)} \Pairing{\pi^{\vee}(g).\overline{\phi}}{e} W_{e^{\vee}}. $$
	Hence we identify the last integral as
	$$ \sum_{e \in \Bas(\pi)} \Zeta \left( s_1+s_2+\frac{1}{2}, \Psi, \beta(\overline{\phi},e) \right) \Zeta \left( s_2-s_1+\frac{1}{2}, e^{\vee} \right). $$
	Successively taking $s_1=0$ and $s_2=0$, we obtain
	$$ \int_{[\PGL_2]} \left( \eis_1 \eis_2 - \Reis \right)(g) \overline{\phi(g)} \ud g = \sum_{e \in \Bas(\pi)} \Zeta \left( \frac{1}{2}, \Psi, \beta(\overline{\phi},e) \right) \Zeta \left( \frac{1}{2}, e^{\vee} \right). $$
	The first equality, with $\pi$ replaced by $\pi^{\vee}$, follows readily. To prove the second equality, we apply the theory of regularized integrals \cite[\S 2]{Wu19_TF} to get
\begin{align*}
	\int_{[\PGL_2]} \left( \eis_1 \eis_2 - \Reis \right)(g) \overline{\eis \left( i\tau, e \right)(g)} \ud g &= \int_{[\PGL_2]} \left( \overline{\eis \left( i\tau, e \right)} \eis_1 - \Reis' \right)(g) \eis_2(g) \ud g \\
	&\quad - \int_{[\PGL_2]}^{\reg} \Reis(g) \overline{\eis \left( i\tau, e \right)(g)} \ud g + \int_{[\PGL_2]}^{\reg} \Reis'(g) \eis_2(g) \ud g,
\end{align*} 
	where $\Reis'$ is the linear combination of Eisenstein series such that $\overline{\eis \left( i\tau, e \right)} \eis_1 - \Reis' \in \intL^2([\PGL_2])$. The two regularized integrals in the last line are zero. In fact, we are in a situation where the regularized integrals are $\GL_2(\A)$-invariant (see \cite[Proposition 2.27 (2)]{Wu19_TF}). Since $\Reis$ and $\overline{\eis(i\tau,e)}$ resp. $\Reis'$ and $\eis_2$ have different eigenvalues with respect to a Hecke operator at a finite place $\vp$ unramified for both functions, the regularized integrals are vanishing. We then unfold the first term on the right hand side by \cite[Proposition 2.5]{Wu19_S}, arriving at (\ref{RSUnfold}) with $\phi$ replaced by $\eis(i\tau,e)$. Arguing as in the cuspidal case, we obtain the second equality.
\end{proof}

	\subsection{Reduction to Local Estimation}

	We turn to the proof of Theorem \ref{4thMBd}. Let $\Phi_0 \in \Sch(\A^2)$ be the standard spherical function given by
	$$ \Phi_{0,v}(x,y) = \left\{ \begin{matrix} e^{-2\pi(\norm[x]^2 + \norm[y]^2)} & \text{if } \F_v = \C \\ e^{-\pi(x^2+y^2)} & \text{if } \F_v = \R \\ \id_{\vo_{\vp} \times \vo_{\vp}}(x,y) & \text{if } v=\vp < \infty \end{matrix} \right.  \quad \Rightarrow \Psi_0 \begin{pmatrix} x_1 & x_2 \\ x_3 & x_4 \end{pmatrix} = \Phi_0(x_1,x_2) \Phi_0(x_3,x_4). $$
	Write $\eis_0(s,g)$ for the spherical Eisenstein series associated with $\Phi_0$. Note that $\eis_0(g) := \eis_0(0,g)$ is real valued. The Whittaker function $W_0(s,g)$ of $\eis_0(s,g)$ is given by
\begin{equation} \label{WhiE0}
	W_0(s,g) = \int_{\A} f_0(s,wn(x)g) \psi(-x) \ud x = \int_{\A^{\times}} \OFour_2 \rpR(g) \Phi_0 (t, t^{-1}) \norm[t]_{\A}^{2s} \ud^{\times}t,
\end{equation}	
	which is regular at $s=0$ and has obvious decomposition into product of local components. Let $r = [\F : \Q]$ be as in the totally real case. We are going to choose $B^{(j)} = (B_v^{(j)})_v \in \A, 1 \leq j \leq 2^{r}$ according to $\vec{T}, \idl{N}$, and consider the choice of test functions
	$$ \Phi_1^{(j)} = \Phi_2^{(j)} = n(B^{(j)}).\Phi_0 \quad \Leftrightarrow \quad \eis_1^{(j)} = \eis_2^{(j)} = n(B^{(j)}).\eis_0 \quad \Leftrightarrow \quad \Psi^{(j)} = \rpR_{n(B^{(j)})}.\Psi_0 . $$

\begin{lemma} \label{4thMGLowerBd}
	For any $\chi \norm_{\A}^{i\tau} \in B(\vec{T},\idl{N})$, we have
	$$ \sum_{j=1}^{2^{r}} \extnorm{\Zeta \left( \frac{1}{2}+i\tau, \chi, \eis_1^{(j)} \right)}^2 \gg \extnorm{L \left( \frac{1}{2}+i\tau,\chi \right)}^4 \left( \prod_{v \mid \infty} T_v \cdot \Nr(\idl{N}) \right)^{-1}. $$
	Consequently, we get
	$$ \sum_{j=1}^{2^{r}} \widetilde{h}(\Psi^{(j)}) \gg \left( \prod_{v \mid \infty} T_v \cdot \Nr(\idl{N}) \right)^{-1} \sum_{\chi \in \widehat{\F^{\times} \R_{>0} \backslash \A^{\times}}} \int_{\R} \extnorm{L \left( \frac{1}{2}+i\tau,\chi \right)}^4 \id_{B(\vec{T},\idl{N})}(\chi \norm_{\A}^{i\tau}) \ud \tau . $$
\end{lemma}

\noindent If we write $\Reis_0$ for the linear combination of Eisenstein series such that $\eis_0^2 - \Reis_0 \in \intL^2([\PGL_2])$, then we have $\eis_1^{(j)} \eis_2^{(j)} - n(B^{(j)}).\Reis_0 = n(B^{(j)}).(\eis_0^2 - \Reis_0) \in \intL^2([\PGL_2])$. Hence $\Reis_j := n(B^{(j)}).\Reis_0$ is the $\Reis$ in the above general discussion for the pair of Eisenstein series $\eis_1^{(j)}$ and $\eis_2^{(j)}$. Now that $\eis_0^2 - \Reis_0$ is $\gp{K}$-invariant (and $\mathrm{U}_2(\C)$ \& $\mathrm{O}_2(\R)$-invariant), we see that $h(\pi)(\Psi^{(j)})$ resp. $h(\chi,s)(\Psi^{(j)})$ is non-vanishing only if $\pi$ resp. $\chi$ is spherical resp. unramified at every place. Under this condition, we have
	$$ h(\pi)(\Psi^{(j)}) = \Pairing{\eis_0^2-\Reis_0}{\overline{e_{\pi}}} \cdot \Zeta \left( \frac{1}{2}, \pi^{\vee}(n(B^{(j)})).e_{\pi}^{\vee} \right), $$
	$$ h(\chi,i\tau)(\Psi^{(j)}) = \Pairing{\eis_0^2-\Reis_0}{\overline{\eis(i\tau, e_{\chi})}} \cdot \Zeta \left( \frac{1}{2}, n(B^{(j)}).\eis(-i\tau, e_{\chi}^{\vee}) \right), $$
	where $e_{\pi}$ resp. $e_{\chi}$ is the unique spherical vector of $\pi$ resp. $\pi(\chi,\chi^{-1})$.

\begin{lemma} \label{4thMGUpperBd}
	For any $\epsilon > 0$, any constant $\theta$ towards the Ramanujan-Petersson conjecture ($\theta=7/64$ is currently the best known \cite{KS02, BB11}) and each $j$, we have
	$$ \Norm[e_{\pi}^{\vee}]^{-1} \cdot \Zeta \left( \frac{1}{2}, \pi^{\vee}(n(B^{(j)})).e_{\pi}^{\vee} \right) \ll_{\epsilon} \Cond(\pi)^{\frac{1}{4}+\epsilon} \left( \prod_{v \mid \infty} T_v \cdot \Nr(\idl{N}) \right)^{-\frac{1}{2}+\theta+\epsilon}, $$
	$$ \Norm[e_{\chi}^{\vee}]^{-1} \cdot \Zeta \left( \frac{1}{2}, n(B^{(j)}).\eis(-i\tau, e_{\chi}^{\vee}) \right) \ll_{\epsilon} \Cond(\chi \norm_{\A}^{i\tau})^{\frac{1}{2}+\epsilon} \left( \prod_{v \mid \infty} T_v \cdot \Nr(\idl{N}) \right)^{-\frac{1}{2}+\epsilon}, $$
	where the norm of $e_{\pi}$ resp. $e_{\chi}$ is the Petersson norm resp. norm on the induced model of $\pi(\chi,\chi^{-1})$. Consequently, we get for any $j$ and any $\epsilon > 0$
	$$ h(\Psi^{(j)}) \ll_{\epsilon} \left( \prod_{v \mid \infty} T_v \cdot \Nr(\idl{N}) \right)^{-\frac{1}{2}+\theta+\epsilon}. $$
\end{lemma}

\begin{lemma} \label{4thMGDegBd}
	For any $j$ and any $\epsilon > 0$, we have
	$$ \norm[DS(\Psi^{(j)})] + \norm[DG(\Psi^{(j)})] \ll_{\epsilon} \left( \prod_{v \mid \infty} T_v \cdot \Nr(\idl{N}) \right)^{\epsilon}. $$
\end{lemma}

\noindent Theorem \ref{4thMBd} obviously follows from the above Lemma \ref{4thMGLowerBd}, \ref{4thMGUpperBd} and \ref{4thMGDegBd}, each of which is of local nature.

	\subsection{Local Estimation}
	
	We first prove Lemma \ref{4thMGLowerBd}, along the way we make the choice of $B^{(j)}$ explicit. For every finite place $\vp$, we choose
	$$ B_{\vp} = \varpi_{\vp}^{-n_{\vp}}, \quad n_{\vp} := \mathrm{ord}_{\vp}(\idl{N}). $$
\begin{lemma} \label{4thMLLowerBdNA}
	(1) Let $W_{0,\vp}$ be the local component at $\vp$ of the Whittaker function $W_0$ of $\eis_0$ (see (\ref{WhiE0})). For any $\chi_{\vp}$ unitary character of $\F_{\vp}^{\times}$ with conductor exponent $\cond(\chi_{\vp}) \leq n_{\vp}$, we have
	$$ \extnorm{L_{\vp} \left( \frac{1}{2}+i\tau, \chi_{\vp} \right)}^{-2} \extnorm{ \Zeta_{\vp} \left( \frac{1}{2}+i\tau, \chi_{\vp}, n(\varpi_{\vp}^{-n_{\vp}}).W_{0,\vp} \right) } \left\{ \begin{matrix} \gg \Nr(\vp)^{-\frac{n_{\vp}}{2}} & \text{if } n_{\vp} > 0 \\ = 1 & \text{if } n_{\vp}=0 \end{matrix} \right. . $$
	
\noindent (2) In the case $n_{\vp} > 0$ and $\chi_{\vp} = \id$ is the trivial character, we have for any integer $n \geq 0$
	$$ \left. \frac{\partial^n}{\partial s^n} \right|_{s=\frac{1}{2}} \zeta_{\vp} \left( \frac{1}{2}+s \right)^{-2} \Zeta_{\vp} \left( \frac{1}{2}+s, n(\varpi_{\vp}^{n_{\vp}}).W_{0,\vp} \right) \ll_n n_{\vp}^{n+1} \Nr(\vp)^{-n_{\vp}} \left( \log \Nr(\vp) \right)^n, $$
	$$ \left. \frac{\partial^n}{\partial s^n} \right|_{s=-\frac{1}{2}} \zeta_{\vp} \left( \frac{1}{2}+s \right)^{-2}  \Zeta_{\vp} \left( \frac{1}{2}+s, n(\varpi_{\vp}^{n_{\vp}}).W_{0,\vp} \right) \ll_n n_{\vp}^{n+1} \left( \log \Nr(\vp) \right)^n. $$
\end{lemma}
\begin{proof}
	(1) Let $v_{\vp}(\cdot)$ be the normalized additive valuation in $\F_{\vp}$. By (\ref{WhiE0}), we easily deduce
	$$ W_{0,\vp}(a(y)) = \int_{\F_{\vp}^{\times}} \id_{\vo_{\vp}^2}(ty,t^{-1}) \ud^{\times}t = (1+v_{\vp}(y)) \norm[y]_{\vp}^{\frac{1}{2}} \id_{\vo}(y). $$
	Write $m_{\vp}=\cond(\chi_{\vp}) \leq n_{\vp}$. We apply \cite[Lemma 4.7]{Wu14}: In the case $m_{\vp} > 0$, we get
	$$ \extnorm{ \Zeta_{\vp} \left( \frac{1}{2}+i\tau, \chi_{\vp}, n(\varpi_{\vp}^{-n_{\vp}}).W_{0,\vp} \right) } = \zeta_{\vp}(1) (1+n_{\vp}-m_{\vp}) \Nr(\vp)^{-\frac{n_{\vp}}{2}}; $$
	in the case $m_{\vp}=0 < n_{\vp}$, writing $X:=\chi_{\vp}(\varpi_{\vp}) \Nr(\vp)^{-\frac{1}{2}-i\tau}$, we get
\begin{align}
	&\quad L_{\vp} \left( \frac{1}{2}+i\tau, \chi_{\vp} \right)^{-2} \Zeta_{\vp} \left( \frac{1}{2}+i\tau, \chi_{\vp}, n(\varpi_{\vp}^{-n_{\vp}}).W_{0,\vp} \right) \nonumber \\
	&= (n_{\vp}+1) X^{n_{\vp}} - n_{\vp} X^{n_{\vp}+1} - \zeta_{\vp}(1) n_{\vp} \Nr(\vp)^{-1} X^{n_{\vp}-1}(1-X)^2; \label{DGLUpperBdNA}
\end{align} 
	while in the case $n_{\vp}=m_{\vp}=0$, we get
	$$ L_{\vp} \left( \frac{1}{2}+i\tau, \chi_{\vp} \right)^{-2} \Zeta_{\vp} \left( \frac{1}{2}+i\tau, \chi_{\vp}, n(\varpi_{\vp}^{-n_{\vp}}).W_{0,\vp} \right) = 1. $$
	The desired bound follows readily from the above explicit formulas.
	
\noindent (2) Putting $\chi_{\vp} = \id$ in the above equation (\ref{DGLUpperBdNA}), we get
\begin{align*}
	\zeta_{\vp} \left( \frac{1}{2}+s \right)^{-2} \Zeta_{\vp} \left( \frac{1}{2}+s, n(\varpi_{\vp}^{n_{\vp}}).W_{0,\vp} \right) &= (n_{\vp}+1) \Nr(\vp)^{- \left( \frac{1}{2}+s \right) n_{\vp}} - n_{\vp} \Nr(\vp)^{- \left( \frac{1}{2}+s \right) (n_{\vp}+1)} \\
	&\quad - \zeta_{\vp}(1) n_{\vp} \Nr(\vp)^{-1-\left( \frac{1}{2}+s \right)(n_{\vp}-1)} (1-\Nr(\vp)^{-\frac{1}{2}-s})^2.
\end{align*} 
	The desired bounds follow readily.
\end{proof}

\begin{lemma} \label{4thMLLowerBdA}
	(1) Let $W_{0,v}$ be the local component at a place $v \mid \infty$ of the Whittaker function $W_0$ of $\eis_0$. There exist absolute constants $C_+, C_- > 0$ such that for $B_v^+ = C_+ T_v$ and $B_v^- = C_- T_v$, we have uniformly for unitary characters $\chi_v$ of $\F_v^{\times}$ with $\Cond(\chi_v) \leq T_v$
	$$ \extnorm{ \Zeta_v \left( \frac{1}{2}, \chi_v, n(B_v^+) W_{0,v} \right) }^2 + \extnorm{ \Zeta_v \left( \frac{1}{2}, \chi_v, n(B_v^-) W_{0,v} \right) }^2 \gg T_v^{-1}. $$
	
\noindent (2) The function $s \mapsto \Zeta_v(s, n(B_v^{\pm}).W_{0,v})$ is regular at $s=1$, and admits a pole at $s=0$ of order $2$. If we denote by $\Zeta_v^{(n)}(s_0, n(B_v^{\pm}).W_{0,v})$ the $n$-th coefficient in the Laurent expansion at $s=s_0$, then we have for any integer $n \geq 0$ resp. $\geq -2$
	$$ \Zeta_v^{(n)}(1, n(B_v^{\pm}).W_{0,v}) \ll_n T_v^{-1} (\log T_v)^n, \quad \Zeta_v^{(n)}(0, n(B_v^{\pm}).W_{0,v}) \ll_n (\log T_v)^{n+2}. $$
\end{lemma}
\begin{proof}
	(1) The complex case is the content of \cite[Lemma 1]{Sar85}. We only need to consider the real case. We omit the subscript $v$ for simplicity of notation. Assume $\chi(t) = t^{i\mu}$ for $t>0$ and some $\mu \in \R$ with $1+\norm[\mu] \leq T$. We have
	$$ W_0(a(y)) = \norm[y]^{\frac{1}{2}} \int_{\R^{\times}} e^{-\pi (t^2 y^2 + t^{-2})} \ud^{\times}t = \norm[y]^{\frac{1}{2}} \BesselK_0(2\pi \norm[y]). $$
	Consequently, we get
\begin{align*}
	\Zeta \left( \frac{1}{2},\chi,n(B).W_0 \right) &= \int_{\R^{\times}} W_0(a(y)) e^{2\pi i By} \chi(y) \ud^{\times}y \\
	&= (2\pi)^{-\frac{1}{2}-i\mu} \left\{ \int_0^{\infty} \BesselK_0(y) e^{iBy} y^{i\mu-\frac{1}{2}} \ud y + \chi(-1) \int_0^{\infty} \BesselK_0(y) e^{-iBy} y^{i\mu-\frac{1}{2}} \ud y \right\}.
\end{align*}
	By the formulas \cite[6.6213]{GR07} and \cite[15.8.1 \& 15.8.25]{OLBC10}, we have
\begin{align*}
	\int_0^{\infty} &\BesselK_0(y) e^{iBy} y^{i\mu-\frac{1}{2}} \ud y = \frac{\sqrt{\pi}}{(1-iB)^{\frac{1}{2}+i\mu}} \frac{\Gamma \left( \frac{1}{2}+i\mu \right)^2}{\Gamma(1+i\mu)} \GenHyG{2}{1}{\frac{1}{2}+i\mu, \frac{1}{2}}{1+i\mu}{- \frac{1+iB}{1-iB}} \\
	&= \frac{\sqrt{\pi}}{2^{\frac{1}{2}+i\mu}} \frac{\Gamma \left( \frac{1}{2}+i\mu \right)^2}{\Gamma(1+i\mu)} \GenHyG{2}{1}{\frac{1}{2}+i\mu, \frac{1}{2}+i\mu}{1+i\mu}{ \frac{1+iB}{2}} \\
	&= \frac{\sqrt{\pi}}{2^{\frac{1}{2}+i\mu}} \left\{ \frac{\Gamma \left( \frac{1}{2} \right) \Gamma \left( \frac{1}{2}+i\mu \right)^2}{\Gamma \left( \frac{3}{4}+\frac{i\mu}{2} \right)^2} \GenHyG{2}{1}{\frac{1}{4}+\frac{i\mu}{2}, \frac{1}{4}+\frac{i\mu}{2}}{\frac{1}{2}}{-B^2} \right. \\
	&\quad \left. -iB \frac{\Gamma \left( -\frac{1}{2} \right) \Gamma \left( \frac{1}{2}+i\mu \right)^2}{\Gamma \left( \frac{1}{4}+\frac{i\mu}{2} \right)^2} \GenHyG{2}{1}{\frac{3}{4}+\frac{i\mu}{2}, \frac{3}{4}+\frac{i\mu}{2}}{\frac{3}{2}}{-B^2} \right\}.
\end{align*}
Hence we obtain 
	$$ \frac{\Zeta \left( \frac{1}{2},\chi,n(B).W_0 \right)}{\pi^{\frac{1}{2}-i\mu} 2^{-2 i\mu}} = \left\{ \begin{matrix} \frac{\Gamma \left( \frac{1}{2}+i\mu \right)^2}{\Gamma \left( \frac{3}{4}+\frac{i\mu}{2} \right)^2}  \GenHyG{2}{1}{\frac{1}{4}+\frac{i\mu}{2}, \frac{1}{4}+\frac{i\mu}{2}}{\frac{1}{2}}{-B^2} & \text{if } \chi(-1)=1 \\ \frac{2iB\Gamma \left( \frac{1}{2}+i\mu \right)^2}{\Gamma \left( \frac{1}{4}+\frac{i\mu}{2} \right)^2}  \GenHyG{2}{1}{\frac{3}{4}+\frac{i\mu}{2}, \frac{3}{4}+\frac{i\mu}{2}}{\frac{3}{2}}{-B^2} & \text{if } \chi(-1)=-1 \end{matrix} \right. . $$
	We treat the first case $\chi(-1)=1$ in detail, the other one being similar. By the formulas \cite[15.6.6]{OLBC10}, we have
	$$ \frac{\Gamma \left( \frac{1}{2}+i\mu \right)^2}{\Gamma \left( \frac{3}{4}+\frac{i\mu}{2} \right)^2}   \GenHyG{2}{1}{\frac{1}{4}+\frac{i\mu}{2}, \frac{1}{4}+\frac{i\mu}{2}}{\frac{1}{2}}{-B^2}=\frac{2^{2i\mu-1}\pi^{-1/2}}{2\pi i}\int_{(c)}\frac{\Gamma(1/4+i\mu/2+z)^2\Gamma(-z)}{\Gamma(1/2+z)}B^{2z}\ud z, $$
	where $-1/4<c<0$.
	Moving the contour of integration to $\Re{z}=-1$, estimating the resulting integral  by
	$$B^{-2}\int_{-\infty}^{\infty}e^{-\pi |y+\mu/2|}\frac{(1+|y|)^{3/2}}{(1+|y+\mu/2|)^{5/2}}\ud y\ll \frac{(1+|\mu|)^{3/2}}{B^2}\ll B^{-1/2},$$
	and evaluating the residue at the point $z=-1/4-i\mu/2$, we get
	\begin{align*} \Zeta \left( \frac{1}{2},\chi,n(B).W_0 \right) &=\frac{\pi^{-i\mu}}{2B^{1/2+i\mu}}\frac{\Gamma(1/4+i\mu/2)}{\Gamma(1/4-i\mu/2)}\\&\times \left(2\log{B}+2\psi(1)-\psi(1/4+i\mu/2)-\psi(1/4-i\mu/2) \right)  +O(B^{-1/2}) \gg B^{-1/2}. 
	\end{align*}
	We conclude the proof.
	
\noindent (2) In the real case, putting $\chi=\norm^{s-1/2}$ in the above computation with simplification, we get
	$$ \Zeta(s, n(B).W_0)=\frac{\pi^{1-s}2^{1-2s}\Gamma^2(s)}{\Gamma^2((1+s)/2)} \GenHyG{2}{1}{\frac{s}{2}, \frac{s}{2}}{\frac{1}{2}}{-B^2}=\frac{\Gamma^2(s/2)}{2\pi^s}\GenHyG{2}{1}{\frac{s}{2}, \frac{s}{2}}{\frac{1}{2}}{-B^2}. $$
As in the previous case, using the Mellin integral for the hypergeometric function (see \cite[15.6.6]{OLBC10}), moving the contour of integration to the left and evaluating the residue at the point $-s/2$, we obtain
	$$ \Zeta(s, n(B).W_0)=\frac{\pi^{1/2-s}\Gamma(s/2)}{2B^s\Gamma(1/2-s/2)}\left( 2\log{B}+2\psi(1)-\psi(s/2)-\psi(1/2-s/2)\right)+O(B^{-3/2}). $$
The desired assertions follow readily from the above formula. In the complex case, we have
	$$ W_0(a(y)) = \norm[y] \int_{\C^{\times}} e^{-2\pi(\norm[ty]^2 + \norm[t]^{-2})} \ud^{\times}t = 2\pi \norm[y] \BesselK_0(4\pi \norm[y]). $$
	Applying the formula \cite[6.5763]{GR07}, we get
\begin{align*}
	\Zeta(s, n(B).W_0) &= 2^{-1} (4\pi)^{2-2s} \int_0^{\infty} \BesselK_0(\rho) J_0(B \rho) \rho^{2s-1} d\rho \\
	&= 2^{-1} (2\pi)^{2-2s} \Gamma \left( s \right)^2 \GenHyG{2}{1}{s, s}{1}{-B^2},
\end{align*}
	where $J_0$ is the usual Bessel-J function of order $0$. 
	We conclude as in the real case.
\end{proof}

\noindent It is clear that Lemma \ref{4thMGLowerBd} follows from Lemma \ref{4thMLLowerBdNA} (1) and Lemma \ref{4thMLLowerBdA} (1). We have also specified $B^{(j)}$ by renumbering $(B_v^{\epsilon_v})_{v \mid \infty} (B_{\vp})_{\vp < \infty}$ with $\epsilon_v \in \{ \pm \}$. Moreover, we have also proved the part of Lemma \ref{4thMGDegBd} for $DG(\Psi)$ by Lemma \ref{4thMLLowerBdNA} (2) and Lemma \ref{4thMLLowerBdA} (2) via (\ref{DGF}).

	We turn to the local upper bounds in Lemma \ref{4thMGUpperBd}.
	
\begin{lemma} \label{4thMLUpperBdNA}
	(1) Let $W_{\vp}$ be the Whittaker function of a spherical element in $\pi_{\vp} \simeq \pi(\norm_{\vp}^{\nu}, \norm_{\vp}^{-\nu})$ of $\PGL_2(\F_{\vp})$ with parameter $\nu \in i\R \cup (-\theta,\theta)$. Then we have for any $\epsilon > 0$
	$$ L_{\vp} \left( \frac{1}{2},\pi_{\vp} \right)^{-1} \sqrt{\frac{L_{\vp}(1,\pi_{\vp} \times \pi_{\vp}^{\vee})}{\zeta_{\vp}(2)}} \Norm[W_{\vp}]^{-1} \Zeta_{\vp} \left( \frac{1}{2}, n(\varpi_{\vp}^{-n_{\vp}}).W_{\vp} \right) \left\{ \begin{matrix} \ll_{\epsilon} \Nr(\vp)^{(-\frac{1}{2}+\theta+\epsilon) n_{\vp}} & \text{if } n_{\vp} > 0 \\ = 1 & \text{if } n_{\vp}=0 \end{matrix} \right. . $$
	
\noindent (2) Assume $n_{\vp} > 0$. Then we have for any integer $n \geq 0$
	$$ \left. \frac{\partial^n}{\partial \nu^n} \right|_{\nu=\frac{1}{2}} \Zeta_{\vp} \left( \frac{1}{2}, W_{\vp} \right)^{-1} \Zeta_{\vp} \left( \frac{1}{2}, n(\varpi_{\vp}^{-n_{\vp}}).W_{\vp} \right) \ll_n \left( n_{\vp} \log \Nr(\vp) \right)^n. $$
\end{lemma}
\begin{proof}
	(1) The bound for the term on the left hand side \emph{without} $L$-factor is treated in \cite[Lemma 6.8]{Wu14}. Adding the $L$-factor does not alter the size, but making the unramified case equal to $1$.
	
\noindent (2) We can take $W_{\vp}(a(y))$ invariant by $y \mapsto y\delta$ for $\delta \in \vo_{\vp}^{\times}$ with (see \cite[Theorem 4.6.5]{Bu98})
	$$ W_{\vp}(\varpi_{\vp}^m) = q^{-\frac{m}{2}} \frac{q^{\nu (m+1)} - q^{-\nu(m+1)}}{q^{\nu}-q^{-\nu}} \id_{m \geq 0}, \quad q := \Nr(\vp). $$
	By \cite[Lemma 4.7]{Wu14}, we get
\begin{align}
	\frac{\Zeta_{\vp} \left( \frac{1}{2}, n(\varpi_{\vp}^{-n_{\vp}}).W_{\vp} \right)}{\Zeta_{\vp} \left( \frac{1}{2}, W_{\vp} \right)} &= q^{- \left( \frac{1}{2}-\nu \right) n_{\vp}} \frac{1-q^{-2(n_{\vp}+1)\nu}}{1-q^{-2\nu}} - q^{-\frac{1}{2}-\nu- \left( \frac{1}{2}-\nu \right) n_{\vp}} \frac{1-q^{-2n_{\vp}\nu}}{1-q^{-2\nu}} \nonumber \\
	&\quad - \frac{q^{-\left( \frac{1}{2}-\nu \right)(n_{\vp}-1)}}{q-1} \frac{1-q^{-2n_{\vp}\nu}}{1-q^{-2\nu}} (1-q^{-\frac{1}{2}-\nu}) (1-q^{-\frac{1}{2}+\nu}). \label{DSLUpperBdNA}
\end{align}
	The desired bound follows (\ref{DSLUpperBdNA}) at once.
\end{proof}

\begin{lemma} \label{4thMLUpperBdA}
	(1) Let $W_v$ be the Whittaker function of a spherical element in $\pi_v \simeq \pi(\norm_v^{\nu}, \norm_v^{-\nu})$ of $\PGL_2(\F_v)$ with parameter $\nu \in i\R \cup (-\theta,\theta)$. Then we have for any $\epsilon > 0$
	$$ \Norm[W_v]^{-1} \cdot \Zeta_v \left( \frac{1}{2},n(B_v^{\pm}).W_v \right) \ll_{\epsilon} T_v^{-\frac{1}{2}+\theta+\epsilon}. $$
	
\noindent (2) We have for any integer $n \geq 0$
	$$ \left. \frac{\partial^n}{\partial \nu^n} \right|_{\nu = \frac{1}{2}} \Zeta_v \left( \frac{1}{2},W_v \right)^{-1} \Zeta_v \left( \frac{1}{2},n(B_v^{\pm}).W_v \right) \ll_n (\log T_v)^n. $$
\end{lemma}
\begin{proof}
	(1) This is contained in \cite[Lemma 6.8]{Wu14}.
	
\noindent (2) We treat the real case, leaving the similar complex case as an exercise. We omit the subscript $v$ for simplicity. We can take (see (\ref{KirToKBessel}))
	$$ W(a(y)) = \norm[y]^{\frac{1}{2}} \BesselK_{\nu}(2\pi \norm[y]). $$
	For $B = B_v^{\pm}$, we have by the formulas \cite[6.5763]{GR07} and \cite[15.8.1 \& 15.8.25]{OLBC10}
\begin{align*}
	\Zeta \left( \frac{1}{2}, n(B).W \right) &= \frac{1}{\sqrt{2\pi}} \int_0^{\infty} \BesselK_{\nu}(y) \left( e^{iBy}+e^{-iBy} \right) y^{-\frac{1}{2}} \ud y \\
	&= 2^{\nu-\frac{1}{2}} \Gamma \left( \frac{1}{2}+\nu \right) \Gamma \left( \frac{1}{2}-\nu \right) \sum_{\pm} \frac{1}{(1 \mp iB)^{\frac{1}{2}+\nu}} \GenHyG{2}{1}{\frac{1}{2}+\nu, \frac{1}{2}+\nu}{1}{-\frac{1 \pm iB}{1 \mp iB}} \\
	&= 2^{-1} \Gamma \left( \frac{1}{2}+\nu \right) \Gamma \left( \frac{1}{2}-\nu \right) \sum_{\pm} \GenHyG{2}{1}{\frac{1}{2}+\nu, \frac{1}{2}-\nu}{1}{\frac{1 \pm iB}{2}} \\
	&= \frac{\sqrt{\pi} \Gamma \left( \frac{1}{2}+\nu \right) \Gamma \left( \frac{1}{2}-\nu \right)}{\Gamma \left( \frac{3}{4}+\frac{\nu}{2} \right) \Gamma \left( \frac{3}{4}-\frac{\nu}{2} \right)} \GenHyG{2}{1}{\frac{1}{4}+\frac{\nu}{2}, \frac{1}{4}-\frac{\nu}{2}}{\frac{1}{2}}{-B^2}.
\end{align*} 
	By the connection formula \cite[2.10 (3)]{Er53}, we deduce
\begin{align*}
	\frac{\Zeta \left( \frac{1}{2}, n(B).W \right)}{\Zeta \left( \frac{1}{2}, W \right)} &= \frac{\sqrt{\pi} \Gamma(-\nu)}{\Gamma \left( \frac{1}{4}-\frac{\nu}{2} \right)^2} (1+B^2)^{-\frac{1}{4}-\frac{\nu}{2}} \GenHyG{2}{1}{\frac{1}{4}+\frac{\nu}{2}, \frac{1}{4}+\frac{\nu}{2}}{1+\nu}{\frac{1}{1+B^2}} + \\
	&\quad \frac{\sqrt{\pi} \Gamma(\nu)}{\Gamma \left( \frac{1}{4}+\frac{\nu}{2} \right)^2} (1+B^2)^{-\frac{1}{4}+\frac{\nu}{2}} \GenHyG{2}{1}{\frac{1}{4}-\frac{\nu}{2}, \frac{1}{4}-\frac{\nu}{2}}{1-\nu}{\frac{1}{1+B^2}}.
\end{align*}
	Replacing the above hypergeometric functions by their Taylor expansions at $0$, we easily deduce the desired bounds.
\end{proof}

\begin{proof}[Proof of Lemma \ref{4thMGUpperBd}]
	Let $W_{\pi}^{\vee}=\otimes_v W_{\pi,v}^{\vee}$ be the Whittaker function of $e_{\pi}^{\vee}$. By (\ref{NormIdCusp}), we have
\begin{align*}
	\frac{\Zeta \left( \frac{1}{2}, n(B^{(j)}.e_{\pi}^{\vee}) \right)}{\Norm[e_{\pi}^{\vee}]} &= \frac{L \left( \frac{1}{2}, \pi^{\vee} \right)}{\sqrt{2 L(1, \pi, \mathrm{Ad})}} \cdot \prod_{v \mid \infty} \frac{\Zeta_v \left( \frac{1}{2}, n(B_v^{(j)}).W_{\pi,v}^{\vee} \right)}{\sqrt{\zeta_v(2)} \Norm[W_{\pi,v}^{\vee}]} \cdot \\
	&\quad \prod_{\vp < \infty} L_{\vp} \left( \frac{1}{2},\pi_{\vp} \right)^{-1} \sqrt{\frac{L_{\vp}(1,\pi_{\vp} \times \pi_{\vp}^{\vee})}{\zeta_{\vp}(2)}} \Norm[W_{\vp}]^{-1} \Zeta_{\vp} \left( \frac{1}{2}, n(\varpi_{\vp}^{-n_{\vp}}).W_{\vp} \right),
\end{align*}
	where those terms over $\vp < \infty$ with $n_{\vp} > 0$ is not $1$. The desired bound follows from Lemma \ref{4thMLUpperBdNA} (1), Lemma \ref{4thMLUpperBdA} (1), the convex bound of $L(1/2,\pi^{\vee})$ and the lower bound of $L(1,\pi,\mathrm{Ad})$ (see \cite{HL94} and \cite[Lemma 3]{BH10}). Similarly, let $W_{\chi}^{\vee} = \otimes_v W_{\chi,v}^{\vee}$ be the Whittaker function of $e_{\chi}^{\vee}$, we have by (\ref{NormIdEis})
\begin{align*}
	\frac{\Zeta \left( \frac{1}{2}, n(B^{(j)}).\eis(-i\tau, e_{\chi}^{\vee}) \right)}{\Norm[e_{\chi}^{\vee}]} &= \frac{\extnorm{L\left( \frac{1}{2}+i\tau,\chi \right)}^2}{\extnorm{L(1+2i\tau,\chi)}} \cdot \prod_{v \mid \infty} \frac{\zeta_v(1)}{\sqrt{\zeta_v(2)}} \frac{\Zeta_v \left( \frac{1}{2}, n(B_v^{(j)}).W_{\chi,v}^{\vee}(-i\tau) \right)}{\Norm[W_{\chi,v}^{\vee}]} \\
	&\cdot \prod_{\vp < \infty} \frac{\extnorm{L_{\vp}(1+2i\tau,\chi_{\vp})}}{\extnorm{L_{\vp}\left( \frac{1}{2}+i\tau,\chi_{\vp} \right)}^2} \frac{\zeta_{\vp}(1)}{\sqrt{\zeta_{\vp}(2)}} \frac{\Zeta_{\vp} \left( \frac{1}{2}, n(\varpi_{\vp}^{-n_{\vp}}).W_{\chi,\vp}^{\vee}(-i\tau) \right)}{\Norm[W_{\chi,\vp}^{\vee}]},
\end{align*}
	The local terms for $v=\vp<\infty$ are special cases of those in the cuspidal case. We argue similarly, using the convex bound of $L(1/2+i\tau,\chi)$ and Siegel's lower bound of $L(1+2i\tau,\chi)$ instead. To derive the bound of $M_3(\Psi^{(j)})$, it suffices to notice that $\Cond(\pi)$ is bounded above by the Laplacian eigenvalue of $e_{\pi}$. Hence the sum resp. integral over $\pi$ resp. $\chi$ and $i\tau$ is bounded by some Sobolev norm of $\eis_0^2-\Reis_0$.
\end{proof}

	Finally, to treat the $DS(\Psi^{(j)})$ part of Lemma \ref{4thMGDegBd}, we note that
\begin{align*}
	h(\id,s)(\Psi^{(j)}) &= h(\id,s)(\Psi_0) \cdot \prod_v \frac{\Zeta_v \left( \frac{1}{2}, n(B_v^{(j)}).W_{\id,v}^{\vee}(-s) \right)}{\Zeta_v \left( \frac{1}{2}, W_{\id,v}^{\vee}(-s) \right)} \\
	&= \frac{\Lambda_{\F} \left( \frac{1}{2}+s \right)^3 \Lambda_{\F} \left( \frac{1}{2}-s \right)^3}{\Lambda_{\F}(1+2s) \Lambda_{\F}(1-2s)} \cdot \prod_v \frac{\Zeta_v \left( \frac{1}{2}, n(B_v^{(j)}).W_{\id,v}^{\vee}(-s) \right)}{\Zeta_v \left( \frac{1}{2}, W_{\id,v}^{\vee}(-s) \right)}.
\end{align*}
	The desired bound of $DS(\Psi^{(j)})$ then follows from Lemma \ref{4thMLUpperBdNA} (2) and Lemma \ref{4thMLUpperBdA} (2) via (\ref{DSF}).

\section{Proof of Central Result}

	\subsection{Hecke Characters of Given Module of Definition}

	We have already proved Theorem \ref{ExpInvMF} \& \ref{DualWtBd}. It remains only to prove Proposition \ref{CubicMBd}. By the bound of the value at $1$ of the adjoint $L$-functions (see \cite{HL94} and \cite[Lemma 3]{BH10}), Lemma \ref{CubeWtLocBdNA} and the obvious bounds $h_{v}(i\tau,\varepsilon) \gg \Delta_v^{-1}$ for $\norm[\tau \pm T_v] \leq \Delta_v$, we see that the left hand side of the desired inequality is bounded by $\Cond(\chi_0)^{\epsilon} \Cond_{\fin}(\chi_0)^2 h(\Psi)$. The bound of degenerate terms in Lemma \ref{DSBd} and \ref{DGBd} shows that $\norm[DS(\Psi)] + \norm[DG(\Psi)] \ll_{\epsilon} \Cond_{\infty}(\chi_0)^{1+\epsilon} \Cond_{\fin}(\chi_0)^{-1+\epsilon}$, which gives the first term on the right hand side of the desired inequality. We are left to bounding

\begin{align}
	\Cond_{\fin}(\chi_0)^2 \widetilde{h}(\Psi) &= \frac{\Cond_{\fin}(\chi_0)^2}{\zeta_{\F}^*} \sum_{\chi \in \widehat{\R_+ \F^{\times} \backslash \A^{\times}}} \int_{\R} \widetilde{h}(\chi \norm_{\A}^{i\tau})(\Psi) \frac{\ud \tau}{2\pi} \nonumber \\
	&\ll \sum_{\substack{\chi \in \widehat{\R_+ \F^{\times} \backslash \A^{\times}} \\ \cond(\chi_{\vp}) \leq \cond(\chi_{0,\vp}), \forall \vp < \infty}} \int_{\R} \extnorm{L \left( \frac{1}{2}+i\tau,\chi \right)}^4 \prod_{v \mid \infty} \extnorm{\widetilde{h}_{v}(s_v,\varepsilon_v)} \ud \tau =: S, \label{1stBdM4}
\end{align}
	where we have applied Corollary \ref{DWtBdNA}. 
	
	Let $\idl{m}$ be the usual conductor ideal of the Hecke character $\chi_0$, i.e.,
	$$ \idl{m} := \sideset{}{_{\vp < \infty}} \prod \vp^{n_{\vp}}, \quad n_{\vp} := \cond(\chi_{0,\vp}). $$
	For notational convenience, let us write $I=\A^{\times}$ for the group of ideles of $\F$. Introduce
	$$ I_{\fin}^{\idl{m}} := \sideset{}{_{\vp < \infty}} \prod U_{\vp}^{(n_{\vp})}, \quad \text{where } U_{\vp}^{(n)} := \left\{ \begin{matrix} \vo_{\vp}^{\times} & \text{if } n=0 \\ 1+\vp^n \vo_{\vp} & \text{if } n \geq 1 \end{matrix} \right. . $$ 

\begin{lemma} \label{IdeleES}
	Let $\gCl(\F)$ be the class group of $\F$. Let $\vo_+$ be the group of totally positive units. Introduce some subgroups of $\vo^{\times}$ of finite index
	$$ \vo^{\idl{m}} := \left\{ \varepsilon \in \vo^{\times} \ \middle| \ \varepsilon \equiv 1 \pmod{\idl{m}} \right\}, \quad \vo_+^{\idl{m}} := \left\{ \varepsilon \in \vo_+ \ \middle| \ \varepsilon \equiv 1 \pmod{\idl{m}} \right\}. $$
	Recall $r=[\F:\Q]$. We have the exact sequences
	$$ 1 \to \vo^{\times} \backslash \left( I_{\infty} \times (\vo/\idl{m})^{\times} \right) \to \F^{\times} \backslash I / I_{\fin}^{\idl{m}} \to \gCl(\F) \to 1, $$
	$$ 1 \to \vo^{\idl{m}} \backslash I_{\infty} \to \vo^{\times} \backslash \left( I_{\infty} \times (\vo/\idl{m})^{\times} \right) \to (\vo^{\times} / \vo^{\idl{m}} ) \backslash (\vo / \idl{m})^{\times} \to 1, $$
	$$ 1 \to \vo_+^{\idl{m}} \backslash I_{\infty}^+ \to \vo^{\idl{m}} \backslash I_{\infty} \to (\vo^{\idl{m}}/\vo_+^{\idl{m}}) \backslash (\Z/2\Z)^r \to 1, $$
	where $I_{\infty}^+ \simeq \R_{>0}^r$ is the connected component subgroup of $I_{\infty}$.
\end{lemma}
\begin{proof}
	These are easy consequences of definitions. We leave the details to the reader.
\end{proof}

\begin{corollary} \label{HeckeCharViaExt}
	Recall that we have identified $\R_+ = \R_{>0}$ with a subgroup of $I_{\infty}^+$ via a section $s_{\F}$ of the adelic norm map. We have the exact sequences
	$$ 1 \to C_1(\idl{m}):=\widehat{(\vo^{\idl{m}}/\vo_+^{\idl{m}}) \backslash (\Z/2\Z)^r} \to \widehat{\R_+ \vo^{\idl{m}} \backslash I_{\infty}} \to \widehat{\R_+ \vo_+^{\idl{m}} \backslash I_{\infty}^+} \to 1, $$
	$$ 1 \to C_2(\idl{m}):=\widehat{(\vo^{\times} / \vo^{\idl{m}} ) \backslash (\vo / \idl{m})^{\times}} \to \widehat{\R_+ \vo^{\times} \backslash \left( I_{\infty} \times (\vo/\idl{m})^{\times} \right)} \to \widehat{\R_+ \vo^{\idl{m}} \backslash I_{\infty}} \to 1, $$
	$$ 1 \to \widehat{\gCl(\F)} \to \widehat{\R_+ \F^{\times} \backslash I / I_{\fin}^{\idl{m}}} \to \widehat{\R_+ \vo^{\times} \backslash \left( I_{\infty} \times (\vo/\idl{m})^{\times} \right)} \to 1. $$
\end{corollary}
\begin{proof}
	These are simple consequences of the Pontryagin duality theorem and the obvious variant of Lemma \ref{IdeleES} by quotients by $\R_+$.
\end{proof}

\noindent The Hecke characters $\chi$ appearing on the right hand side of (\ref{1stBdM4}) are those of a module of definition $\idl{m}$ (see \cite[Chapter \Rmnum{7} (6.11)]{Ne99}, i.e., characters of $\R_+ \F^{\times} \backslash I / I_{\fin}^{\idl{m}}$. For every character $\chi_{\infty}^+$ of $\R_+ \vo_+^{\idl{m}} \backslash I_{\infty}^+$, we fix an extension to $\R_+ \F^{\times} \backslash I / I_{\fin}^{\idl{m}}$, still denoted by $\chi_{\infty}^+$. By Corollary \ref{HeckeCharViaExt}, we can parametrize $\chi$ by
	$$ \chi = (\chi_{\infty}^+, \chi_1, \chi_2, \chi_0), \quad \chi_1 \in C_1(\idl{m}), \chi_2 \in C_2(\idl{m}), \chi_0 \in \widehat{\gCl(\F)}. $$
	Since $\R_{>0} \simeq \R$ by the log-map, we have the canonical isomorphisms $I_{\infty}^+ \simeq \R^r$ and $I_{\infty}^{+,1} \simeq \R_0^r$ where $I_{\infty}^{+,1}$ is the subgroup of norm one elements of $I_{\infty}^+$ and $\R_0^r$ is the hyperplane orthogonal to the image of $\R_+$, i.e. $\R e_0$ with $e_0:=(1,\cdots,1)$ given by (implicitly we number the archimedean places as $v_j$ with $1 \leq j \leq r$)
	$$ \R_0^r := \left\{ \vec{x}=(x_j)_{1 \leq j \leq r} \in \R^r \ \middle| \ \sideset{}{_{j=1}^r} \sum x_j = 0 \right\}. $$
	Under these isomorphisms, we have $\widehat{I_{\infty}^+} \simeq \widehat{\R^r} \simeq \R^r$ and $\widehat{\R_+ \backslash I_{\infty}^+} \simeq \widehat{I_{\infty}^{+,1}} \simeq \widehat{\R_0^r} \simeq \R_0^r$, where the pairing between $\R^r$ and $\widehat{\R^r} \simeq \R^r$ is given by
\begin{equation} \label{EucPairing}
	\R^r \times \widehat{\R^r} \to \C^1, \quad (\vec{x}, \vec{\mu}) := \exp(i \sideset{}{_{j=1}^r} \sum x_j \mu_j). 
\end{equation}
	By Dirichlet's unit theorem, the group $\vo_+$ of totally positive units is isomorphic to $\Z^{r-1}$, hence is a lattice of $I_{\infty}^{+,1} \simeq \R_0^r$. The same holds for its subgroup $\vo_+^{\idl{m}}$ of finite index. We denote the image lattice of $\vo_+$ resp. $\vo_+^{\idl{m}}$ in $\R_0^r$ by $\Gamma$ resp. $\Gamma_{\idl{m}}$. Their dual lattices with respect to (\ref{EucPairing}) are denoted by $\Gamma^{\vee}$ resp. $\Gamma_{\idl{m}}^{\vee}$ respectively. Then there is a bijection between $\Gamma_{\idl{m}}^{\vee}$ and $\widehat{\R_+ \vo_+^{\idl{m}} \backslash I_{\infty}^+}$ given by $\vec{\mu} \to \chi_{\infty}^+(\vec{\mu})$ with
	$$ \chi_{\infty}^+(\vec{\mu})((t_j)_{1 \leq j \leq r}) := \sideset{}{_{j=1}^r} \prod t_j^{i\mu_j}. $$
	Call $\vec{\mu}$ the parameter of $\chi_{\infty}^+(\vec{\mu})$. The box $B^{\vee}(\chi_0,\epsilon)$ becomes a box $B^{\vee}:=\sideset{}{_{j=1}^r} \prod [-B_j,B_j]$ in the Euclidean space $\R^r$ with $B_j = T_0 (1+T_j) \log^2(1+T_j)$ and $T_0 := \Cond(\chi_0)^{\frac{\epsilon}{r}}, T_j := T_{v_j}$. Denote the image under the orthogonal projection from $\R^r$ to $\R_0^r$ by 
	$$ B_0^{\vee} := \left\{ \vec{\mu} \in \R_0^r \ \middle| \ \exists \tau \in \R \text{ such that } \vec{\mu}+\tau e_0 \in B^{\vee} \right\}. $$
	We can thus decompose the right hand side of (\ref{1stBdM4}) as (abbreviating $(\chi_{\infty}^+(\vec{\mu}), \chi_1, \chi_2, \chi_0)$ as $\chi(\vec{\mu})$)
	
\begin{align*}
	S &= \sum_{\chi_0 \in \widehat{\gCl(\F)}} \sum_{\chi_1 \in C_1(\idl{m})} \sum_{\chi_2 \in C_2(\idl{m})} \sum_{\vec{\mu} \in \Gamma_{\idl{m}}^{\vee}} \int_{\R} \extnorm{L \left( \frac{1}{2}+i\tau,\chi(\vec{\mu}) \right)}^4 \prod_{j=1}^r \extnorm{\widetilde{h}_{j}(i(\mu_j+\tau),\varepsilon_j)} \ud \tau \\
	&= \sum_{\chi_0 \in \widehat{\gCl(\F)}} \sum_{\chi_1 \in C_1(\idl{m})} \sum_{\chi_2 \in C_2(\idl{m})} \sum_{\vec{\mu} \in \Gamma_{\idl{m}}^{\vee} \cap B_0^{\vee}} \int_{\vec{\mu}+\tau e_0 \in B^{\vee}} \extnorm{L \left( \frac{1}{2}+i\tau,\chi(\vec{\mu}) \right)}^4 \prod_{j=1}^r \extnorm{\widetilde{h}_j(i(\mu_j+\tau),\varepsilon_j)} \ud \tau \\
	&+ \sum_{\chi_0 \in \widehat{\gCl(\F)}} \sum_{\chi_1 \in C_1(\idl{m})} \sum_{\chi_2 \in C_2(\idl{m})} \sum_{\vec{\mu} \in \Gamma_{\idl{m}}^{\vee} \cap B_0^{\vee}} \int_{\vec{\mu}+\tau e_0 \notin B^{\vee}} \extnorm{L \left( \frac{1}{2}+i\tau,\chi(\vec{\mu}) \right)}^4 \prod_{j=1}^r \extnorm{\widetilde{h}_j(i(\mu_j+\tau),\varepsilon_j)} \ud \tau \\
	&+ \sum_{\chi_0 \in \widehat{\gCl(\F)}} \sum_{\chi_1 \in C_1(\idl{m})} \sum_{\chi_2 \in C_2(\idl{m})} \sum_{\vec{\mu} \in \Gamma_{\idl{m}}^{\vee}, \vec{\mu} \notin B_0^{\vee}} \int_{\R} \extnorm{L \left( \frac{1}{2}+i\tau,\chi(\vec{\mu}) \right)}^4 \prod_{j=1}^r \extnorm{\widetilde{h}_j(i(\mu_j+\tau),\varepsilon_j)} \ud \tau \\
	&=: S_0 + S_1 + S_2.
\end{align*}

	By Proposition \ref{BasisTransProp} (1), we have $\widetilde{h}_j(i(\mu_j+\tau),\varepsilon_j) \ll 1$ with absolute implied constants. Hence
	$$ S_0 \ll \sum_{\chi \in \widehat{\F^{\times} \R_{>0} \backslash \A^{\times}}} \int_{\R} \extnorm{L \left( \frac{1}{2}+i\tau,\chi \right)}^4 \id_{B^{\vee}(\chi_0,\epsilon)}(\chi \norm_{\A}^{i\tau}) \ud \tau. $$
	It remains only to bound $S_1$ and $S_2$ as remainder terms.

	\subsection{Bounds of Remainder Terms}
	
	Above all, we notice that the box $B^{\vee}$ is \emph{balanced}, in the sense that there is a constant $A>0$ depending only on $\epsilon$ such that $T_0 \leq B_i \leq T_0^{1+A}$ for all $i$ (say $A=2r\epsilon^{-1}$).
	
	Let $\vec{\mu} \in \Gamma_{\idl{m}}^{\vee} \cap B_0^{\vee}$. Since $B^{\vee}$ is convex and compact, there exist $\tau_{\min}(\vec{\mu}) \leq \tau_{\max}(\vec{\mu})$ such that
	$$ \left\{ \tau \in \R \ \middle| \ \vec{\mu}+\tau e_0 \in B^{\vee} \right\} = [\tau_{\min}(\vec{\mu}), \tau_{\max}(\vec{\mu})]. $$
	Hence the condition $\vec{\mu}+\tau e_0 \notin B^{\vee}$ is equivalent to $\tau > \tau_{\max}(\vec{\mu})$ or $\tau < \tau_{\min}(\vec{\mu})$. Moreover, there exist $1 \leq j_{\min}, j_{\max} \leq r$ such that
	$$ \tau > \tau_{\max}(\vec{\mu}) \Leftrightarrow \mu_{j_{\max}}+\tau > B_{j_{\max}}, \quad \tau < \tau_{\min}(\vec{\mu}) \Leftrightarrow \mu_{j_{\min}}+\tau < -B_{j_{\min}}. $$
	By Proposition \ref{DualWtRD} and the convex bound, we have for any constant $C \gg 1$

\begin{align*}
	&\quad \int_{\mu_{j_{\max}}+\tau > B_{j_{\max}}} \extnorm{L \left( \frac{1}{2}+i\tau,\chi(\vec{\mu}) \right)}^4 \prod_{j=1}^r \extnorm{\widetilde{h}_j(i(\mu_j+\tau),\varepsilon_j)} \ud \tau \\
	&\ll_{\epsilon,C} \Nr(\idl{m})^{1+\epsilon} \int_{\mu_{j_{\max}}+\tau > B_{j_{\max}}} \norm[\mu_{j_{\max}}+\tau]^{-C-1} \sideset{}{_{j \neq j_{\max}}} \prod (1+T_j)^{1+\epsilon} \ud \tau \\
	&\ll 
	\Cond(\chi_0)^{1+\epsilon} \cdot B_{j_{\max}}^{-C} \leq T_0^{\frac{A}{2}+r-C}.
\end{align*}
	We have similarly
	$$ \int_{\mu_{j_{\min}}+\tau < -B_{j_{\min}}} \extnorm{L \left( \frac{1}{2}+i\tau,\chi(\vec{\mu}) \right)}^4 \prod_{j=1}^r \extnorm{\widetilde{h}_j(i(\mu_j+\tau),\varepsilon_j)} \ud \tau \ll_{\epsilon,C} T_0^{\frac{A}{2}+r-C}. $$
	Hence for an individual $\vec{\mu} \in \Gamma_{\idl{m}}^{\vee} \cap B_0^{\vee}$, we have
\begin{align}
	&\quad \sum_{\chi_0 \in \widehat{\gCl(\F)}} \sum_{\chi_1 \in C_1(\idl{m})} \sum_{\chi_2 \in C_2(\idl{m})} \int_{\vec{\mu}+\tau e_0 \notin B^{\vee}} \extnorm{L \left( \frac{1}{2}+i\tau,\chi(\vec{\mu}) \right)}^4 \prod_{j=1}^r \extnorm{\widetilde{h}_j(i(\mu_j+\tau),\varepsilon_j)} \ud \tau \nonumber \\
	&\ll_{\epsilon,C} \cdot \norm[\gCl(\F)] \cdot \norm[C_1(\idl{m})] \cdot \norm[C_2(\idl{m})] \cdot T_0^{\frac{A}{2}+r-C}. \label{S1IndividualBd}
\end{align} 

\begin{lemma} \label{LatticePtCount}
	There is a constant $B>0$ depending only on $\F$ such that $\norm[\Gamma_{\idl{m}}^{\vee} \cap B_0^{\vee}] \ll \Cond(\chi_0)^B$.
\end{lemma}
\begin{proof}
	Let $B_m:= \max_{j} B_j$. We enlarge $B^{\vee}$ to $D^{\vee} := [-B_m, B_m]^r$. If $d_1 \mid d_2 \mid \cdots d_r$ are the elementary divisors of the abelian group $\Gamma/\Gamma_{\idl{m}}$, we consider the lattice
	$$ d_r^{-1} \Gamma^{\vee} := \left\{ \vec{\mu} \in \R_0^r \ \middle| \ d_r \vec{\mu} \in \Gamma^{\vee} \right\}, $$
	which satisfies $\Gamma^{\vee} < \Gamma_{\idl{m}}^{\vee} < d_r^{-1} \Gamma^{\vee}$. Thus $\Gamma_{\idl{m}}^{\vee} \cap B_0^{\vee} \subset d_r^{-1} \Gamma^{\vee} \cap D_0^{\vee}$, where $D_0^{\vee}$ is the orthogonal projection of $D^{\vee}$ onto $\R_0^r$. Note that $D_0^{\vee}$ is a regular $2r$-polyhedron in the $r-1$ dimensional Euclidean space $\R_0^r$, whose circumcircle has radius $\leq \sqrt{r} B_m$. Note also that $\Gamma^{\vee}$ is a lattice in $\R_0^r$ depending only on $\F$. The trivial lattice point counting in $\R_0^r$ implies
	$$ \norm[d_r^{-1} \Gamma^{\vee} \cap D_0^{\vee}] = [\Gamma^{\vee} \cap d_rD_0^{\vee}] \ll_{\F} (d_r \sqrt{r} B_m)^{r-1} $$
	with implied constant depending only on $\F$ (the shape and covolume of $\Gamma^{\vee}$). We conclude by the trivial bounds $d_r \leq [\Gamma : \Gamma_{\idl{m}}] = [\vo_+ : \vo_+^{\idl{m}}] \leq \Nr(\idl{m})$ and $B_m \leq \Cond(\chi_0)^{1+\frac{\epsilon}{r}}$.
\end{proof}

\noindent Since $\norm[C_1(\idl{m})] \leq 2^r$ and $\norm[C_2(\idl{m})] \leq \Nr(\idl{m})$, we obtain by (\ref{S1IndividualBd}) and Lemma \ref{LatticePtCount}
	$$ S_1 \ll_{\F, \epsilon, C} T_0^{(1+B)r\epsilon^{-1}+A-C} \ll \Cond(\chi_0)^{-C'} $$
	for any constant $C' \gg 1$.
	
	Let $\vec{\mu} \in \Gamma_{\idl{m}}^{\vee}$ but $\vec{\mu} \notin B_0^{\vee}$. Consider the following function on $\vec{x} \in \R_0^r$
	$$ d(\vec{x}) := \min_{\tau \in \R} \Norm[\vec{x}+\tau e_0]_{\infty} = \min_{\tau \in \R} \max_{1 \leq j \leq r} \norm[x_j+\tau]. $$
	Since for any $\tau \in \R$ we have
	$$ \frac{1}{\sqrt{r}} \Norm[\vec{x}]_2 \leq \frac{1}{\sqrt{r}} \Norm[\vec{x}+\tau e_0]_2 \leq \Norm[\vec{x}+\tau e_0]_{\infty} \leq \Norm[\vec{x}+\tau e_0]_2, $$
	we deduce by taking the minimum over $\tau$ that
	$$ \frac{1}{\sqrt{r}} \Norm[\vec{x}]_2 \leq d(\vec{x}) \leq \Norm[\vec{x}]_2. $$
	Moreover, for $\vec{x} \notin B_0^{\vee}$ it is easy to see
	$$ d(\vec{x}) \geq \min_{1 \leq j \leq r} B_j \geq T_0. $$
	For any $\tau \in \R$, we claim that there is $1 \leq i_0 \leq r$ such that $\norm[\mu_{i_0}+\tau] \geq \max(B_{i_0}, d(\vec{\mu})^{\frac{1}{1+A}})$. In fact, since $\vec{\mu} \notin B_0^{\vee}$, there is $1 \leq i \leq r$ such that $\norm[\mu_i+\tau] > B_i$. There is $1 \leq j \leq r$ such that $\norm[\mu_j+\tau] = d(\vec{\mu})$. Either $\norm[\mu_j+\tau] > B_j$, then we choose $i_0=j$; or $\norm[\mu_j+\tau] \leq B_j \leq B_i^{1+A} < \norm[\mu_i+\tau]^{1+A}$, then we choose $i_0=i$. We have obtained
	$$ \{ \tau \in \R \} \subset \sideset{}{_{i=1}^r} \bigcup \left\{ \tau \in \R \ \middle| \ \norm[\mu_i+\tau] \geq \max(B_i, d(\vec{\mu})^{\frac{1}{1+A}}) \right\}. $$
	Hence for such an individual $\vec{\mu}$, we have for any $C \gg 1$

\begin{align*}
	&\quad \int_{\R} \extnorm{L \left( \frac{1}{2}+i\tau,\chi(\vec{\mu}) \right)}^4 \prod_{j=1}^r \extnorm{\widetilde{h}_j(i(\mu_j+\tau),\varepsilon_j)} \ud \tau \\
	&\ll_{\epsilon,C} \sum_{i=1}^r \int_{\norm[\mu_i+\tau] \geq \max(B_i, d(\vec{\mu})^{\frac{1}{1+A}})} \norm[\mu_i+\tau]^{-C-1} \prod_{j \neq i} T_j^{1+\epsilon} \ud \tau \\
	&\ll \Cond_{\infty}(\chi_0)^{1+\epsilon} B_i^{-\frac{C}{2}} d(\vec{\mu})^{-\frac{C}{2(1+A)}} \leq \Cond_{\infty}(\chi_0)^{1+\epsilon} \Cond(\chi_0)^{-\frac{C}{2}} \Norm[\vec{\mu}]_2^{-\frac{C}{2(1+A)}} r^{\frac{C}{4(1+A)}}.
\end{align*}
	
\noindent Taking $C$ sufficiently large and summing over $\vec{\mu}$, we deduce and conclude the proof of Proposition \ref{CubicMBd} by
	$$ S_2 \ll_{\F,\epsilon,C} \Cond(\chi_0)^{1+\epsilon-\frac{C}{2}}. $$

\section*{Acknowledgement}

	Han Wu would like to thank Dihua Jiang and Xincheng Miao for bring Sakellaridis's work into his attention, and Zhi Qi for valuable discussions. Han Wu was partly supported by the Leverhulme Trust Research Project Grant RPG-2018-401.

\bibliographystyle{acm}

\bibliography{mathbib}

\end{document}